\documentclass[12pt]{amsart}
\usepackage{amsmath,amssymb,amsthm,amscd}
\usepackage{color}
\usepackage[mathscr]{euscript}
\usepackage[arrow,curve,frame,color]{xy}
\usepackage{hyperref}
\hypersetup{breaklinks=true}
\usepackage[small,nohug,heads=vee]{diagrams}

\usepackage{graphicx}
\usepackage[labelformat=empty]{caption}
\usepackage{accents}

\usepackage[utf8]{inputenc} 
\usepackage[T1]{fontenc}
\usepackage{enumitem}

\numberwithin{equation}{section}
\theoremstyle{plain}

\newtheorem{theorem}[equation]{Theorem}
\newtheorem{thm}[equation]{Theorem}
\newtheorem{proposition}[equation]{Proposition}
\newtheorem{lemma}[equation]{Lemma}
\newtheorem{corollary}[equation]{Corollary}

\newtheorem{lem}[equation]{Lemma}

\newtheorem{prop}[equation]{Proposition}
\newtheorem{cor}[equation]{Corollary}

\theoremstyle{remark}
\newtheorem{remark}[equation]{Remark}

\theoremstyle{definition}
\newtheorem{definition}[equation]{Definition}

\newtheorem{question}[equation]{Question}
\newtheorem*{question*}{Question}

\newtheorem{example}[equation]{Example}

\newtheorem{claim}[equation]{Claim}

\newcommand{\A}{{\mathcal A}}
\newcommand{\B}{\mathcal{B}}

\newcommand{\E}{\mathbb E}

\newcommand{\G}{{\mathcal G}}
\renewcommand{\H}{\mathbb H}
\newcommand{\h}{{\mathcal H}}

\newcommand{\K}{{\mathcal K}}

\renewcommand{\P}{{\mathcal P}}

\newcommand{\R}{\mathbb R}
\newcommand{\calr}{{\mathcal R}}
\renewcommand{\S}{{\mathcal S}}
\newcommand{\T}{{\mathcal T}}

\newcommand{\W}{{\mathcal W}}

\newcommand{\Z}{\mathbb Z}

\newcommand{\acts}{\curvearrowright}
\newcommand{\al}{\alpha}
\newcommand{\be}{\beta}
\newcommand{\ben}{\begin{enumerate}}
\newcommand{\cat}{\operatorname{CAT}}

\newcommand{\ccc}{\operatorname{CCC}}
\newcommand{\De}{\Delta}
\newcommand{\diam}{\operatorname{diam}}

\newcommand{\een}{\end{enumerate}}
\newcommand{\face}{\operatorname{Face}}
\newcommand{\ga}{\gamma}
\newcommand{\Ga}{\Gamma}
\newcommand{\id}{\operatorname{id}}
\newcommand{\im}{\operatorname{Im}}
\newcommand{\Int}{\operatorname{Int}}
\newcommand{\isom}{\operatorname{Isom}}
\newcommand{\La}{\Lambda}

\newcommand{\lra}{\longrightarrow}
\newcommand{\out}{\operatorname{Out}}

\newcommand{\qi}{\operatorname{QI}}
\newcommand{\ra}{\rightarrow}
\newcommand{\res}{\mathsf{S}^r}
\newcommand{\restr}{\mbox{\Large \(|\)\normalsize}}
\newcommand{\si}{\sigma}
\newcommand{\stab}{\operatorname{Stab}}
\newcommand{\st}{\textrm{st}}
\newcommand{\Star}{\operatorname{St}}

\begin{document}

\title{Groups quasi-isometric to RAAG's}
\author{Jingyin Huang}
\author{Bruce Kleiner}
\thanks{The second author was supported by a NSF grant DMS-1405899 and a Simons Fellowship.}

\date{\today}
\maketitle

\begin{abstract}
We characterize groups quasi-isometric to a right-angled Artin group $G$ 
with finite outer automorphism group.   In particular all such groups admit a geometric action on 
a $\cat(0)$ cube complex that has an equivariant ``fibering'' over the Davis
building of $G$.
 
\end{abstract}

\section{Introduction}

\subsection*{Overview}
In this paper we will study right angled
Artin groups (RAAG's).
Like other authors, our motivation for considering these groups
stems from the fact that they are an easily defined yet
remarkably rich class of objects, exhibiting interesting features
from many different vantage points: 
algebraic structure (subgroup structure, automorphism groups)
\cite{droms1987isomorphisms,servatius1989automorphisms,laurence1995generating,charney_vogtmann_crisp}, finiteness properties
\cite{bestvina1997morse,brady2001connectivity}, representation varieties
\cite{kapovich_millson},
$CAT(0)$ geometry \cite{croke2000spaces}, cube complex geometry
\cite{wise,haglund_wise_special}, and coarse geometry 
\cite{wise1996non,burger_mozes,bks2,behrstock2008quasi,MR2727658,huang_quasiflat,huang2014quasi}. 
Further impetus for studying RAAG's comes from
their  role in the theory of special cube complexes, 
which was a key ingredient in Agol's spectacular solution of Thurston's
virtual Haken and virtual fibered conjectures  
\cite{agol,wise,haglund_wise_special,MR1347406,kahn_markovic}.

Our focus here is on 
quasi-isometric rigidity, which is part of 
Gromov's program for quasi-isometric classification of
groups and metric spaces.  
In this paper we build on \cite{bks1,bks2,huang_quasiflat,huang2014quasi}, 
which analyzed the structure
of individual quasi-isometries $G\ra G$, where $G$ is 
a RAAG with finite outer automorphism group.
Our main results are a structure theorem for groups of quasi-isometries 
(more precisely quasi-actions), and a characterization of finitely
generated groups quasi-isometric to such RAAG's.  Both are formulated 
using a geometric description
in terms of Caprace-Sageev restriction quotients  
\cite{caprace2011rank} and the Davis building \cite{davis_buildings_are_cat0}.

\subsection*{Background}
Prior results on quasi-isometric
classification of RAAG's may be loosely divided into two types: 
internal quasi-isometry
classification among (families of) RAAG's, 
and  quasi-isometry rigidity results characterizing arbitrary finitely
generated groups quasi-isometric to a given RAAG.
In the former category, it is known   
that to classify RAAG's up to quasi-isometry, it suffices to consider the 
case when the groups are $1$-ended and do not admit any nontrivial direct product
decomposition, or equivalently, when their defining graphs are connected,
contain more than one vertex, and do not admit a nontrivial join
decomposition
(\cite[Theorem 2.9]{huang2014quasi},
\cite{papasoglu2002quasi,kapovich1998quasi}).  
This covers, for instance, the classification  up to quasi-isometry of
RAAG's that may be formed inductively by taking products or free products,
starting from copies of $\Z$.
Beyond this, internal classification is known for RAAG's whose  defining graph 
is a tree \cite{behrstock2008quasi} or a higher dimension analog
\cite{MR2727658}, or when the outer automorphism
group is finite \cite{huang2014quasi,bks2}.
General quasi-isometric classification 
results in the literature are much more limited; if $H$ is a 
finitely generated group quasi-isometric to a RAAG $G$ then:
\begin{enumerate}[label=(\roman*)]
\item If $G$ is  free or free abelian,  $H$ is virtually
free or free abelian, respectively \cite{stallings,dunwoody1985accessibility,bass1972degree,gromov1981groups}.
\item If $G=F_k\times \Z^\ell$, then $H$ is virtually $F_k\times \Z^\ell$ \cite{whyte2010coarse}.
\item If the defining graph of 
$G$ is a tree, then $H$ is virtually the fundamental group of a non-geometric graph manifold that has nonempty boundary in every Seifert fiber space component, and moreover $H$ is virtually cocompactly cubulated \cite{behrstock2008quasi,kapovich1997quasi,hagen2015cocompactly}.
\item If $G$ is a product of free groups, then $H$ acts geometrically
on a product of trees \cite{ahlin,kapovich1998quasi,mosher2003quasi}.
\end{enumerate}
Unlike (i)-(iii), which give characterizations up to commensurability, 
the  characterization in (iv) only asserts the existence of an action on a 
good geometric model; the stronger commensurability assertion is false,
in view of examples of Wise and Burger-Mozes
\cite{wise1996non,burger_mozes}.

\subsection*{The setup}
We now recall some terminology and notation;
see Section \ref{sec_preliminaries} for more detail.

If $\Ga$ is a finite simplicial graph with vertex
set $V(\Ga)$, we denote the
associated right-angled Artin group by $G(\Ga)$.  
This is the fundamental group of the Salvetti complex $S(\Ga)$, a nonpositively
curved cube complex that may be constructed by choosing a pointed unit
length circle $(S^1_v,\star_v)$ for every vertex $v\in V(\Ga)$,
forming the pointed product torus $\prod_v (S^1_v,\star_v)$, 
and passing to the union of
the product subtori corresponding to the cliques (complete subgraphs) in $\Ga$.  
The clique subtori are the {\em standard tori} in $S(\Ga)$.

We denote the universal covering by $X(\Ga)\ra S(\Ga)$; here $X(\Ga)$
is a $\cat(0)$ cube complex on which $G(\Ga)$ acts geometrically by 
deck transformations.   
The inverse image of a standard torus
in $S(\Ga)$ under the universal covering $X(\Ga)\ra S(\Ga)$
breaks up into connected components;
these are the {\em standard flats} in $X(\Ga)$ which we partial order 
by inclusion.  Note that we include
standard tori and standard flats of dimension $0$.

The poset of standard flats in $X(\Ga)$ turns out to be
crucial to our story.  Using it one may define 
 a locally infinite  $\cat(0)$ cube complex $|\B|(\Ga)$ whose cubes
of dimension $k\geq 0$
are indexed by inclusions $F_1\subset F_2$, and $F_1,F_2$ are standard flats
where $\dim F_2=\dim F_1 + k$.  Elements of the $0$-skeleton
 $|\B|^{(0)}(\Ga)$ correspond to the trivial inclusions $F\subset F$
where $F$ is a standard flat, so we will identify $|\B|^{(0)}(\Ga)$
with the collection of standard flats, and define the rank of a vertex
of $|\B|(\Ga)$ to be the dimension of the corresponding standard flat;
in particular we may identify the $0$-skeleton
$X^{(0)}(\Ga)$ with the set of rank $0$ vertices of 
$|\B|^{(0)}$.
Since $G(\Ga)\acts X(\Ga)$ preserves the collection of standard flats, 
there is an induced action
$G(\Ga)\acts|\B|(\Ga)$ by cubical isomorphisms.
The above description is a slight variation on
the original construction of the same object given by 
Davis, in which one  views $|\B|(\Ga)$ as the Davis realization of a
certain right-angled building $\B(\Ga)$ associated with $G(\Ga)$, where
the apartments of $\B(\Ga)$ are
modelled on the right-angled Coxeter group 
$W(\Ga)$ with defining graph $\Ga$; see \cite{davis_buildings_are_cat0}
and Section \ref{sec_preliminaries}.   By abuse of terminology
we will refer to this cube complex as the {\em Davis building
associated with $G(\Ga)$}; it has been called the 
modified Deligne complex in \cite{charney_davis_kpi1_problem} 
and flat space in \cite{bks1}.

The following lemma is not difficult to prove.
\begin{lemma}

\mbox{}
\begin{itemize}
\item Every isomorphism $|\B|^{(0)}(\Ga)\ra |\B|^{(0)}(\Ga)$ 
of the poset of standard flats
extends to a unique cubical isomorphism $|\B|(\Ga)\ra |\B|(\Ga)$ $($Section \ref{building}$)$.
\item  Every cubical isomorphism of $|\B|\ra |\B|$ induces  a poset
isomorphism $|\B|^{(0)}\ra |\B|^{(0)}$ $($Lemma \ref{rank preserving}$)$.
\item A bijection $\phi^{(0)}:|\B|^{(0)}(\Ga)\supset
X^{(0)}(\Ga)\ra X^{(0)}(\Ga)\subset |\B|^{(0)}(\Ga)$  induces/extends to a 
poset isomorphism $|\B|^{(0)}(\Ga)\ra |\B|^{(0)}(\Ga)$ iff it
is
flat-preserving in the sense that for every standard flat
$F_1\subset X(\Ga)$, the $0$-skeleton $F_1^{(0)}$ is mapped bijectively
by $\phi^{(0)}$ onto the
$0$-skeleton of some standard flat $F_2$ $($Section \ref{subsection_canonical restriction quotient}$)$.
\end{itemize}
\end{lemma}

\begin{remark}
We caution the reader 
that a cubical isomorphism $|\B|(\Ga)\ra |\B|(\Ga)$ need not arise from
an isomorphism $\B(\Ga)\ra \B(\Ga)$ of the right-angled building.
\end{remark}

\subsection*{Rigidity and flexibility}
We now fix a finite graph $\Ga$ such that the outer
automorphism group $\out(G(\Ga))$ is finite;
by  work of \cite{charney2012random,day_out}, one may view this
as the generic case.   The reader may  find it helpful to keep in mind the case when 
$\Ga$ is a pentagon. 

Since there is no ambiguity 
in $\Ga$ we will often suppress it in the notation below.

It is known that in this case $X=X(\Ga)$ is not quasi-isometrically
rigid: there are quasi-isometries that are not at finite sup distance
from isometries, and there are finitely generated groups $H$ that
are quasi-isometric to $X$, but do not admit geometric actions on 
$X$ (Corollary \ref{no action}).  On the other hand, quasi-isometries 
exhibit a form of  partial
rigidity that is captured by the building $|\B|$:
\begin{theorem}[\cite{huang2014quasi,bks2}]
\label{thm_intro_vertex_rigidity}
Suppose $\out(G(\Gamma))$ is finite and $G(\Ga)\not\simeq\Z$. If $\phi:X^{(0)}\ra X^{(0)}$ is an $(L,A)$-quasi-isometry,
then there is a unique cubical isomorphism  $|\B|\ra |\B|$
such that associated flat-preserving bijection
$\bar\phi:X^{(0)}\ra X^{(0)}$ is at finite sup
distance from $\phi$, and moreover
$$
d(\bar\phi,\phi)=\sup\{v\in X^{(0)}\mid d(\bar\phi(v),\phi(v))\}
<D=D(L,A)\,.
$$
By the uniqueness assertion,
we obtain a cubical action $\qi(X)\acts |\B|$ of the quasi-isometry 
group of $X$ on $|\B|$.
\end{theorem}

We point out that the partial rigidity statement of the theorem does not hold
for general RAAG's: it only holds for the RAAG's covered by the theorem in
\cite {huang2014quasi}.

\bigskip\bigskip

\subsection*{The main results}
We will produce good geometric models quasi-isometric to $X(\Ga)$
that are simultaneously compatible with group actions,  the underlying
building $|\B|$, and cubical structure.   The key idea for expressing this is:

\begin{definition}
\label{def_restriction_quotient}
A cubical map 
 $q:Y\ra Z$ between $\cat(0)$ cube complexes (see Definition \ref{cubical map})
is a {\em restriction
quotient}
if it is  surjective, and the point inverse $q^{-1}(z)$  is a convex subset of $Y$ 
for every $z\in Z$.
\end{definition}

It turns out that restriction quotients as defined above are essentially 
equivalent to the class of mappings introduced by Caprace-Sageev 
\cite{caprace2011rank} with a different definition (see
Section \ref{sec_restriction_quotients} 
for the proof that the definitions are equivalent).  
Restriction quotients $Y\ra |\B|$ provide a means to ``resolve'' or 	``blow-up''
the locally infinite building $|\B|$ to a locally finite $\cat(0)$ cube complex.

\begin{theorem}
\label{thm_main_intro_quasi_action}
(See Section \ref{sec_preliminaries} for definitions.)
Let $H\acts X$ be a quasi-action of an arbitrary group on $X=X(\Ga)$,
where $\out(G(\Ga))$ is finite and $G(\Ga)\not\simeq \Z$.   Then there is an $H$-equivariant 
restriction quotient $H\acts Y\stackrel{q}{\lra}H\acts |\B|$
where: 
\begin{enumerate}[label=(\alph*)]
\item $H\acts |\B|$ is the cubical action arising from the quasi-action $H\acts X$
using Theorem \ref{thm_main_intro_quasi_action}, and $H\acts Y$ is a cubical  action.
\item The point inverse $q^{-1}(v)$ of every rank $k$ vertex $v\in |\B|^{(0)}$
is a copy of $\R^k$ with the usual cubulation.
\item $H\acts X$ is 
quasiconjugate to the cubical action $H\acts Y$.
\end{enumerate} 
\end{theorem}

\begin{theorem}
\label{thm_main_intro}
If $|\out(G(\Ga))|<\infty$ and $G(\Ga)\not\simeq\Z$, then  a finitely generated group $H$ 
is quasi-isometric to $G(\Ga)$ iff 
 there is an $H$-equivariant 
restriction quotient $H\acts Y\stackrel{q}{\lra}H\acts |\B|$
where 
\begin{enumerate}[label=(\alph*)]
\item $H\acts Y$ is a geometric cubical  action.
\item $H\acts |\B|$ is cubical.
\item The point inverse $q^{-1}(v)$ of every rank $k$ vertex $v\in |\B|^{(0)}$
is a copy of $\R^k$ with the usual cubulation.
\end{enumerate}
\end{theorem}
\begin{remark}
In fact the restriction quotient $Y\ra |\B|$ in Theorems
\ref{thm_main_intro_quasi_action} and \ref{thm_main_intro} has slightly
more structure, see Theorem \ref{thm_quasi_action1}.
\end{remark}

In particular, we have:
\begin{corollary}
\label{cor_intro_main} 
Any group quasi-isometric to $G$ is cocompactly cubulated, i.e.
it has a geometric cubical action on a  $\cat(0)$ cube complex.
\end{corollary}

One may compare Theorem \ref{thm_main_intro} with rigidity theorems for
symmetric spaces or products of trees, which characterize a quasi-isometry 
class of groups by the existence of a geometric action on a model space of
a specified type
\cite{sullivan,gromov_hyperbolic_manifolds_groups,tukia,pansu,kleiner1997rigidity,stallings,dunwoody1985accessibility,kapovich1998quasi,mosher2003quasi,ahlin}.  
As in the case of products of trees --- and unlike the
case of symmetric spaces ---
there are finitely generated groups $H$ as in 
Theorem \ref{thm_main_intro} which do not admit a
geometric action on the original model space $X$, so one is forced
to pass to a different space $Y$ \cite{bks2,huang2014quasi}.  Also, Theorems \ref{thm_main_intro_quasi_action} 
and \ref{thm_main_intro} fail for general RAAG's, for instance for 
free abelian groups of rank $\geq 2$, and for products of nonabelian
free groups $\prod_{1\leq j\leq k} G_j$, for $k\geq 1$.

The quasi-isometry invariance 
of the existence of a cocompact cubulation as asserted in Corollary
\ref{cor_intro_main} is false in general.
Some groups quasi-isometric to $\H^2\times\R$ 
admit a cocompact cubulation, while others are not virtually $CAT(0)$
\cite{bridson_haefliger}.   Combining 
\cite{leeb}, \cite{behrstock2008quasi} and \cite{hagen_przytycki}, 
it follows that there is a
pair of quasi-isometric $CAT(0)$ graph manifold groups, one of which 
is the fundamental group of a compact special cube complex, while the
other is not virtually cocompactly cubulated.
The  quasi-isometry invariance of cocompact cubulations fails to hold even among RAAG's:
for $n>1$
there are groups  
quasi-isometric  $\R^n$ that are not cocompactly cubulated
\cite{hagen_crystallographic}.

Earlier cocompact
cubulation theorems in the spirit of  Corollary \ref{cor_intro_main} 
include 
the cases of groups quasi-isometric to trees, products of trees, and 
hyperbolic $k$-space $\H^k$ for $k\in \{2,3\}$
\cite{stallings,dunwoody1985accessibility,kapovich1998quasi,mosher2003quasi,ahlin,gitik_widths,kahn_markovic,bergeron_wise}.  
It is worth noting that each case requires
different ingredients that are specific to the spaces in question.

%

\subsection*{Further results}
We briefly discuss some further results here, referring the reader to the
body of the paper for details.

One portion of the proof of Theorem \ref{thm_main_intro_quasi_action} has
to do with 
the geometry of restriction quotients, and more specifically, restriction
quotients with a right-angled building as target.  We view this as a 
contribution to cube complex geometry, and to the geometric theory of 
graph products; beyond the 
references mentioned already, our treatment has been
influenced by the papers of Januszkiewicz-Swiatkowski and Haglund  
\cite{januszkiewicz_swiatkowski,haglund2008finite}.
The main results on this are:

\begin{enumerate}[label=(\alph*)]
\item We show in Section \ref{sec_restriction_quotients}
that restriction quotients  may be characterized
in several different ways.  
\item We show that having a restriction quotient $q:Y \ra Z$
is equivalent to 
knowing certain ``fiber data'' living on the target complex $Z$.  
\item When $|\B|$ is the Davis realization of  a right-angled
building $\B$ and $Y\ra |\B|$ is a restriction
quotient whose fibers are copies of $\R^k$ with dimension specified 
as in Theorems \ref{thm_main_intro_quasi_action} and \ref{thm_main_intro},
the fiber data in (b) may be distilled even more, leading to what we call
``blow-up data''.   
\end{enumerate}

As by-products of (a)-(c), we obtain:

\begin{itemize}
\item A characterization of the quasi-actions $H\acts X(\Ga)$ that are 
quasiconjugate to isometric actions $H\acts X(\Ga)$ (Section \ref{subsection_nicer action}).
\item A characterization of the restriction quotients $Y\ra |\B|$
satisfying (b) of Theorem \ref{thm_main_intro_quasi_action}
for which $Y$ is quasi-isometric to $X$ (Corollary \ref{if and only if} and Theorem \ref{quasi-isometry to G}).  
\item A proof of uniqueness of the right-angled building modelled
on the right-angled Coxeter group $W(\Ga)$ with defining graph $\Ga$,
with countably infinite rank $1$ residues (Corollary \ref{uniqueness of building}).
\item Applications to more general graph products (Theorem \ref{geometric action}).
\end{itemize}

It follows from \cite{kleiner2001groups} that a finitely generated group $H$ 
quasi-isometric to a symmetric space of noncompact type $X$ admits
an epimorphism $H\ra \La$ with finite kernel, where $\La$ is a 
cocompact lattice in the isometry group $\isom(X)$.  
In contrast to this, we have the following result,
which is inspired by \cite[Theorem 9, Corollary 10]{mosher2003quasi}:

\begin{theorem}
\label{thm_no_overgroup}
(See Section \ref{subsection_nicer action})
Suppose $G$ is a RAAG with $|\out(G)|<\infty$.
Then there are finitely generated groups $H$ and $H'$ quasi-isometric to 
$G$  that do not admit discrete, virtually faithful cocompact
representations into
the same locally compact topological group.
\end{theorem} 

\subsection*{Open questions}
As mentioned above, Corollary \ref{cor_intro_main} may be considered part
of the quasi-isometry classification program for finitely generated groups.
The leads to:

\begin{question}
\label{que_commensurability}
If $\out(G(\Ga))$ is finite, what is the commensurability classification of
groups quasi-isometric to $G(\Ga)$?
Are they all commensurable to $G(\Ga)$? What about cocompact lattices
in the automorphism group of $X(\Ga)$?
\end{question}
For comparison, we recall that any  group quasi-isometric
to a tree is commensurable to a free group, but there are groups
quasi-isometric to a product of trees that contain no nontrivial 
finite index subgroups, and are therefore not commensurable to
a product of free groups \cite{wise1996non,burger_mozes}.

We mention that Theorem \ref{thm_main_intro}
will be used in forthcoming 
work to  answer Question \ref{que_commensurability} in certain cases.

\bigskip

\subsection*{Discussion of the proofs}
Before sketching the arguments for Theorems 
\ref{thm_main_intro_quasi_action} and
\ref{thm_main_intro}, we first illustrate
them in the tautological case when $H=G$ and the quasi-action is the 
deck group action
$G\acts X$.  In this case we cannot take $Y=X$, as there is no
$H$-equivariant restriction quotient $H\acts X\ra H\acts |\B|$
satisfying (c) of Theorem \ref{thm_main_intro}.  Instead, we
use a different geometric model.  

\begin{definition}[Graph products of spaces \cite{haglund2008finite}]
\label{def_graph_product_of_spaces}
For every vertex $v\in V(\Ga)$, choose a pointed geodesic metric space
$(Z_v,\star_v)$.  The {\em $\Ga$-graph product of  
$\{(Z_v,\star_v)\}_{v\in V(\Ga)}$}
is obtained by forming the product $\prod_v (Z_v,\star_v)$, and passing
to the union of the subproducts corresponding to the cliques in $\Ga$.
We denote this by $\prod_\Ga (Z_v,\star_v)$.
When the $Z_v$'s are nonpositively curved, then so is the
graph product \cite[Corollary 4.6]{haglund2008finite}.
\end{definition}

There are three graph products that are useful here:
\ben
\item The
Salvetti complex $S(\Ga)$ is the graph product $\prod_\Ga(S^1_v,\star_v)$,
where $(S^1_v,\star_v)$ is a pointed unit circle.
\item  For every $v\in V(\Ga)$, let $(L_v,\star_v)$
be a pointed lollipop, i.e. $L_v$ is the wedge of the 
unit  circle $S^1_v$
and a unit interval $I_v$, and the basepoint $\star_v\in L_v$ is the vertex
of valence $1$.   Then the graph product 
$\prod_\Ga(L_v,\star_v)$ is the {\em exploded Salvetti 
complex $S_e=S_e(\Ga)$.} We denote its universal covering 
by $X_e\ra S_e$.
\item If  $(Z_v,\star_v)$  is a unit interval and $\star_v\in Z_v$ is an endpoint
for every $v\in V(\Ga)$, then the graph product 
$\prod_\Ga (Z_v,\star_v)$ is the {\em Davis chamber}, 
i.e. it is a copy of the Davis realization $|c|$ of a chamber $c$ in 
$|\B|(\Ga)$;  
 for this reason we will denote it by $|c|_\Ga$.
\een
By collapsing  the  interval $I_v$ in each lollipop $L_v$
to a point, we obtain a cubical map $S_e\ra S$; this has  
contractible point inverses, and  is therefore  a homotopy 
equivalence.  If we collapse the circles $S^1_v\subset L_v$
to points instead, we get a  map $S_e \ra |c|_\Ga$ to the Davis chamber
whose point 
inverses are closed, locally convex tori.  The point inverses of the 
composition $X_e\ra S_e\ra |c|_\Ga$ cover the torus point inverses of 
$S_e\ra |c|_\Ga$, and their connected components form a ``foliation'' of $X_e$
by flat convex subspaces.  It turns out that by collapsing $X_e$ along
these flat subspaces, we obtain a copy of $|\B|$, and the quotient map
$X_e\ra |\B|$ is a restriction quotient $X_e\ra |\B|$.  Alternately,
one may take the collection $\mathcal{K}$ of hyperplanes  of $X_e$ 
dual to edges $\si\subset X_e$ whose projection under $X_e\ra |c|_\Ga$ is
an edge, and form the restriction quotient using the Caprace-Sageev
construction.

\begin{remark}
The exploded Salvetti complex and the restriction quotient $X_e\ra |\B|$
were discussed  in \cite{bks2} in the $2$-dimensional
case, using an ad hoc construction that was initially 
invented for ``ease of visualization''.  However, the authors were unaware of
the  general description above, and the notion of restriction 
quotient had not yet appeared. 
\end{remark}

We now discuss the proofs of Theorem \ref{thm_main_intro_quasi_action}
and the forward direction of \ref{thm_main_intro}.

The 
forward direction of Theorem \ref{thm_main_intro} reduces to Theorem 
\ref{thm_main_intro_quasi_action}, by the standard observation that
a quasi-isometry $H\ra G\stackrel{qi}{\simeq}X$ allows us to 
quasiconjugate the left translation action $H\acts H$ to a quasi-action
$H\acts X$.  Therefore we focus on Theorem \ref{thm_main_intro_quasi_action}.

Let $H\acts X$ be as in Theorem \ref{thm_main_intro_quasi_action}.  By a 
bounded perturbation, we may assume that this quasi-action preserves the
$0$-skeleton $X^{(0)}\subset X$.  Applying
Theorem \ref{thm_intro_vertex_rigidity}, we may further assume that we
have an action $H\acts X^{(0)}$ by flat-preserving quasi-isometries.  The 
fact the we have an action, rather than just a quasi-action, 
comes from the
uniqueness in Theorem \ref{thm_main_intro_quasi_action}; this turns
out to be a crucial point in the sequel.

Before proceeding further, we remark that if one is only interested in 
Corollary \ref{cor_intro_main} as opposed to the more refined statement in 
Theorem \ref{thm_main_intro}, then there is an alternate approach using 
wallspaces.  This is treated in Sections \ref{sec_construction of wallspace}
and \ref{sec_property of wallspace}.

Given a standard geodesic $\ell\subset X$, the parallel set $P_\ell\subset X$
decomposes as a product $\R_\ell\times Q_\ell$, where $\R_\ell$ is a copy
of $\R$;  likewise there
is a product decomposition of $0$-skeleta 
$P_\ell^{(0)}\simeq \Z_\ell\times Q_\ell^{(0)}$.  One argues that
the action $H\acts X^{(0)}$ permutes the collection of 
$0$-skeleta $\{P_\ell^{(0)}\}_{\ell}$, and that for any $\ell$, the stabilizer
$\stab(P_\ell^{(0)},H)$ 
of $P_\ell^{(0)}$ in $H$ acts on $P_\ell^{(0)}\simeq \Z_\ell\times Q_\ell^{(0)}$
preserving the product structure.  We call the 
action $\stab(P_\ell^{(0)},H)\acts \Z_\ell$ a {\em factor action}.  The factor
actions are by bijections with  quasi-isometry constants 
bounded uniformly independent of $\ell$.

It turns out that factor actions  play
a central role in the story.  For instance, when the action $H\acts X^{(0)}$ 
is the restriction of an action $H\acts X$ by cubical isometries, then 
the factor actions $H_{[\ell]}\acts \Z_{[\ell]}$ are also actions by isometries.
In general the factor actions can be arbitrary: up to isometric
conjugacy, any action $A\acts \Z$
by quasi-isometries with uniform constants can arise as a factor
action for some action as in Theorem \ref{thm_main_intro_quasi_action}.  
A key step in the proof is to show that such actions have a relatively
simple structure:
\begin{proposition}[Semiconjugacy]
\label{prop_intro_semiconjugacy}
Let $U\stackrel{\rho_0}{\acts} \Z$ be an action of an arbitrary
group by $(L,A)$-quasi-isometries. 
Then there is an
isometric action $U\stackrel{\rho_1}{\acts}\Z$ and  
surjective equivariant $(L',A')$-quasi-isometry 
$$
U\stackrel{\rho_0}{\acts} \Z\lra
U\stackrel{\rho_1}{\acts}\Z\,,
$$
where $L'$ and $A'$ depend only on $L$ and $A$.
\end{proposition}
The assumption that $\rho_0$ is an action, as opposed
to a quasi-action, is crucial:
if a group $U$ has a nontrivial
quasihomomorphism $\alpha:U\ra \R$, then the translation quasi-action
$U\stackrel{\hat\alpha}{\acts}\R$ defined by  $\hat\alpha(u)(x)=x+\alpha(u)$
is quasiconjugate to a quasi-action on $\Z$, but not to an isometric
action on $\Z$.  

It follows immediately from the
Proposition \ref{prop_intro_semiconjugacy} that $U\stackrel{\rho_0}{\acts}\Z$
is quasiconjugate to an isometric action on the tree $\R$.
In that respect Proposition \ref{prop_intro_semiconjugacy} is similar to
the theorem of Mosher-Sageev-Whyte 
about promoting quasi-actions on bushy trees to isometric actions on trees
\cite[Theorem 1]{mosher2003quasi}.
Since $\R$ is not 
bushy  \cite[Theorem 1]{mosher2003quasi} does not apply, and indeed the 
example
above shows that the assumption of bushiness is essential in that theorem.

Continuing with   the proof of Theorem \ref{thm_main_intro_quasi_action},
Proposition \ref{prop_intro_semiconjugacy} gives a good geometric
model for the factor action $\stab(P_\ell^{(0)},H)\acts \Z_\ell$: we 
simply extend each isometry $\Z_\ell\ra\Z_\ell$ to an isometry 
$\R_\ell\ra \R_\ell$, thereby obtaining a cubical action
$\stab(P_\ell^{(0)},H)\acts \R_\ell$.  
In vague terms, the remainder of the proof is concerned with combining 
these cubical models into models for the fibers of a restriction
quotient $Z\ra |\B|$, in an $H$-equivariant way.  This portion of the proof
is covered by more general results about restriction quotients, see
(b)-(c) in the subsection on further results above.

\subsection*{Organization of the paper}

The paper is divided into three relatively independent parts. We first outline the content of each part, then give suggestions 
for readers who want to focus on one particular part.

A summary of notation can found in Section \ref{sec_index_of_notation}. Section \ref{sec_preliminaries} contains some background material on quasi-actions, $CAT(0)$ cube complexes, RAAG's and buildings. One can proceed directly to later sections with Section \ref{sec_index_of_notation} and Section \ref{sec_preliminaries} as references.

The main part of the paper is Section \ref{sec_restriction_quotients} to Section \ref{sec_quasi action on Z}, where we prove Theorem \ref{thm_main_intro}. In Section \ref{sec_restriction_quotients} we discuss restriction
quotients, showing how to construct a restriction quotient $Y\ra Z$ starting
from the target $Z$ and an admissible assignment of fibres to the cubes of $Z$. Then we discuss equivariance properties and the coarse geometry of restriction quotients.

In Section \ref{sec_blow-up of building}, we introduce blow-ups of buildings based on Section \ref{sec_restriction_quotients}.  These are restriction quotients
$Y\ra |\B|$ where the target is a right-angled building and the fibres are Euclidean spaces of varying dimension. We motivate our construction in Section \ref{subsection_canonical restriction quotient} and Section \ref{subsection_Restriction quotients with Euclidean fibers}. Blow-ups of buildings are constructed in Section \ref{subsection_input for Z-blow-up}. Several properties of them are discussed in Section \ref{subsection_more properties} and Section \ref{subsection_morhpism of blow-up data}. We incorporate a group action into our construction in Section \ref{subsection_an equivariant constuction}.

In Section \ref{subsec_cubulation}, we apply the construction in Section \ref{subsection_an equivariant constuction} to RAAG's and prove Theorem \ref{thm_main_intro} modulo Theorem \ref{prop_intro_semiconjugacy}, which is postponed until Section \ref{sec_quasi action on Z}. In Section \ref{subsection_nicer action} we answer several natural questions motivated by Theorem \ref{thm_main_intro}, and prove Theorem \ref{thm_no_overgroup}.

The second part of the paper is Section \ref{sec_blow-up building more general version}. We discuss an alternative construction of blow-ups of buildings. In certain cases, this is more general than the construction in Section \ref{sec_blow-up of building}. We also discuss several applications of this construction to graph products.

The third part of the paper is Section \ref{sec_construction of wallspace} and Section \ref{sec_property of wallspace}. Using wallspaces we give an alternative way to cubulate groups quasi-isometric to RAAG's, and prove a weaker version of Theorem \ref{thm_main_intro}.

The reader can proceed directly to Section \ref{sec_blow-up building more general version} with reference to Section \ref{building} and Definition \ref{factor action}. The reader can also start with Section \ref{sec_construction of wallspace} with reference to Section \ref{subsec_raag}, and come back to Section \ref{sec_quasi action on Z} when needed.

\subsection*{Acknowledgements}
The second author would like to thank Mladen Bestvian and Michah Sageev for 
many discussions related to atomic RAAG's, which planted the seeds for the
cubulation result in this paper.  In particular, Proposition 
\ref{key lemma} originated in these discussions.

\tableofcontents

\section{Index of notation}
\label{sec_index_of_notation}
\begin{itemize}
\item $\mathcal{B}$: A combinatorial  building (Section \ref{building}).
\item $|\B|$: The Davis realization of a building (Section \ref{building}).
\item Chambers in the combinatorial building $\mathcal{B}$ are $c$, $c'$, $d$.
\item $|c|_\Ga$: the Davis chamber (the discussion after
Definition \ref{def_graph_product_of_spaces}, Section \ref{building}).
\item $\res$: the collection of all spherical residues in the building $\B$.
\item $\textmd{proj}_{\calr}:\B\to \calr$: the nearest point projection from $\B$ to a residue $\calr$ (Section \ref{building}).
\item $\Lambda_{\B}$: the collection of parallel classes of rank 1 residues in the combinatorial building $\B$. We also write $\Lambda$ when the building $\B$ is clear (Section \ref{subsection_input for Z-blow-up}).
\item $T$: a type map which assigns each residue of $\B$ a subset of $\Lambda_{\B}$ (Section \ref{subsection_input for Z-blow-up}).
\item $\beta$: a branched line (Section \ref{sec_construction of wallspace}).
\item $\ccc$: the category of nonempty $\cat(0)$ cube complexes with morphisms
given by convex cubical embeddings.
\item $P_C$: the parallel set of a closed convex subset of a $\cat(0)$ space
(Section \ref{subsec_cube complex}).
\item $W(\Ga)$:  The right-angled Coxeter group with defining graph $\Ga$
(Section \ref{building}).
\item $G(\Ga)$ the right-angled Artin group with defining graph $\Ga$.
\item $X(\Ga)\to S(\Ga)$ the universal covering of the Salvetti complex (Section \ref{subsec_raag})
\item $X_e(\Ga)\to S_e(\Ga)$ the universal covering of the exploded Salvetti complex (after Definition \ref{def_graph_product_of_spaces} and Section \ref{subsection_canonical restriction quotient}). We also write $X_e\ra S_e$ when the graph $\Ga$ is clear.
\item $\mathcal{P}(\Ga)$: the extension complex 
(Definition \ref{def_extension_complex}).
\item $X\ra X(\K)$: the restriction quotient arising from a set $\K$ of 
hyperplanes in a $\cat(0)$ cube complex (Definition \ref{def_restriction_quotient}).
\item $Lk(x,X)$ or $Lk(c,X)$: the link of a vertex $x$ or a cell $c$ in a polyhedral complex $X$.
\item $\Ga_1\circ\Ga_2$: the join of two graphs.
\item $K_1\ast K_2$: the join of two simplicial complexes.
\end{itemize}
\section{Preliminaries}
\label{sec_preliminaries}
\subsection{Quasi-actions} We recall several definitions from coarse geometry.
\begin{definition}
An $(L,A)$-quasi-action of a group $G$ on a metric space $Z$ is a map $\rho:G\times Z\to Z$ so that $\rho(\ga,\cdot):Z\to Z$ is an $(L,A)$ quasi-isometry for every $\ga\in G$, $d(\rho(\ga_1,\rho(\ga_2,z)),\rho(\ga_1\ga_2,z))<A$ for every $\ga_1,\ga_2\in G$, $z\in Z$, and $d(\rho(e,z),z)<A$ for every $z\in Z$.
\end{definition}

The action $\rho$ is \textit{discrete} if for any point $z\in Z$ and any $R>0$, the set of all $\ga\in G$ such that $\rho(\ga,z)$ is contained in the ball $B_{R}(z)$ is finite; $\rho$ is \textit{cobounded} if $Z$ coincides with a finite tubular neighbourhood of the \textquotedblleft orbit\textquotedblright\ $\rho(G,z)$. If $\rho$ is a discrete and cobounded quasi-action of $G$ on $Z$, then the orbit map $\ga\in G\to \rho(\ga,z)$ is a quasi-isometry. Conversely, given a quasi-isometry between $G$ and $Z$, it induces a discrete and cobounded action of $G$ on $Z$.

Two quasi-actions $\rho$ and $\rho'$ are \textit{equivalent} if there exists a constant $D$ so that $d(\rho(\ga),\rho'(\ga))<D$ for all $\ga\in G$.

\begin{definition}
Let $\rho$ and $\rho'$ be quasi-actions of $G$ on $Z$ and $Z'$ respectively, and let $\phi:Z\to Z'$ be a quasi-isometry. Then $\rho$ is \textit{quasiconjugate} to $\rho'$ via $\phi$ if there is a $D$ so that $d(\phi\circ\rho(\ga),\rho'(r)\circ\phi)<D$ for all $\ga\in G$.
\end{definition}
\subsection{$\cat(0)$ cube complexes}
\label{subsec_cube complex}
We refer to \cite{bridson_haefliger} for background about $\cat(0)$ spaces (Chapter II.1) and cube complexes (Chapter II.5), and \cite{MR1347406,sageevnotes} for $\cat(0)$ cube complexes and hyperplanes.

A unit Euclidean $n$-cube is $[0,1]^{n}$ with the standard metric. A \textit{mid-cube} is the set of fixed points of a reflection with respect to some $[0,1]$ factor of $[0,1]^{n}$. A cube complex $Y$ is obtained by taking a collection of unit Euclidean cubes and gluing them along isometric faces. The gluing metric on $Y$ is $\cat(0)$ if and only if $Y$ is simply connected and the link of each vertex in $Y$ is a flag simplicial complex (\cite{MR919829}), in this case, $Y$ is called a \textit{$\cat(0)$ cube complex}. 

Let $X$ be a $\cat(0)$ space and let $C\subset X$ be a closed convex subset. Then there is a well-defined nearest point projection from $X$ to $C$, which we denote by $\pi_{C}:X\to C$. Two convex subsets $C_{1}$ and $C_{2}$ are \textit{parallel} if $d(\cdot,C_{2})|_{C_{1}}$ and $d(\cdot,C_{1})|_{C_{2}}$ are constant functions. In this case, the convex hull of $C_{1}$ and $C_{2}$ is isometric to $C_{1}\times [0,d(C_{1},C_{2})]$.

For closed convex subset $C\subset X$, we define $P_{C}$, the \textit{parallel set} of $C$, to be the union of all convex subsets of $X$ which are parallel to $C$. If $C$ has geodesic extension property, then $P_{C}$ is also a closed convex subset and admits a canonical splitting $P_{C}\cong C\times C^{\perp}$ (\cite[Chapter II.2.12]{bridson_haefliger}).

Suppose $Y$ is a $\cat(0)$ cube complex. Then two edges $e$ and $e'$ are parallel if and only if there exists sequences of edges $\{e_{i}\}_{i=1}^{n}$ such that $e_{1}=e$, $e_{n}=e'$, and $e_{i},e_{i+1}$ are the opposite sides of a $2$-cube in $Y$. For each edge $e\subset Y$. Let $N_{e}$ be the union of cubes in $Y$ which contain an edge parallel to $e$. Then $N_{e}$ is a convex subcomplex of $Y$, moreover, $N_{e}$ has a natural splitting $N_{e}\cong h_{e}\times[0,1]$, where $[0,1]$ corresponds to the $e$ direction. The subset $h_{e}\times\{1/2\}$ is called the \textit{hyperplane} dual to $e$, and $N_{e}$ is called the \textit{carrier} of this hyperplane. Each hyperplane is a union of mid-cubes, hence has a natural cube complex structure, which makes it a $\cat(0)$ cube complex. The following are true for hyperplanes:
\begin{enumerate}
\item Each hyperplane $h$ is a convex subset of $Y$. Moreover, $Y\setminus h$ has exactly two connected components. The closure of each connected component is called a \textit{halfspace}. Each halfspace is also a convex subset.
\item Pick an edge $e\subset Y$. We identify $e$ with $[0,1]$ and consider the $\cat(0)$ projection $\pi_{e}:Y\to e\cong[0,1]$. Then $h=\pi^{-1}_{e}(1/2)$ is the hyperplane dual to $e$, and $\pi^{-1}_{e}([0,1/2]),\pi^{-1}_{e}([1/2,1])$ are two halfspaces associated with $h$. The closure of $\pi^{-1}_{e}((0,1))$ is the carrier of $h$.
\end{enumerate}

Let $Y$ be a $\cat(0)$ cube complex and let $l\in Y$ be a geodesic line (with respect to the $\cat(0)$ metric) in the 1-skeleton of $Y$. Let $e\subset l$ be an edge and pick $x\in e$. We claim that if $x$ is in the interior of $e$, then $\pi^{-1}_{l}(x)=\pi^{-1}_{e}(x)$. It is clear that $\pi^{-1}_{l}(x)\subset\pi^{-1}_{e}(x)$. Suppose $y\in \pi^{-1}_{e}(x)$. It follows from the splitting $N_{e}\cong h_{e}\times[0,1]$ as above that the geodesic segment $\overline{xy}$ is orthogonal to $l$, i.e. $\angle_{x}(y,y')=\pi/2$ for any $y'\in l\setminus\{x\}$, thus $y\in \pi^{-1}_{l}(x)$.

The above claim implies $\pi^{-1}_{l}(x)$ is a convex subset for any $x\in l$. Moreover, the following lemma is true.

\begin{lem}
\label{project onto line}
Let $Y$ and $l$ be as before. Pick an edge $e\subset Y$. If $e$ is parallel to some edge $e'\subset l$, then $\pi_{l}(e)=e'$, otherwise $\pi_{l}(e)$ is a vertex of $l$.
\end{lem}

Now we define an alternative metric on the $\cat(0)$ cube complex $Y$, which is called the $l^{1}$-metric. One can view the 1-skeleton of $Y$ as a metric graph with edge length $=1$, and this metric extends naturally to a metric on $Y$. The distant between two vertices in $Y$ with respect to this metric is equal to the number of hyperplanes separating these two vertices.

A \textit{combinatorial geodesic} in $Y$ is an edge path in $Y^{(1)}$ which is a geodesic with respect to the $l^{1}$ metric. However, we always refer to the $\cat(0)$ metric when we talk about a geodesic.

If $Y$ is finite dimensional, the $l^{1}$ metric and the $\cat(0)$ metric on $Y$ are quasi-isometric (\cite[Lemma 2.2]{caprace2011rank}). In this paper, we will use the $\cat(0)$ metric unless otherwise specified.

\begin{definition} (\cite[Section 2.1]{caprace2011rank})
\label{cubical map}
A cellular map between cube complexes is \textit{cubical} if its restriction $\sigma\to\tau$ between cubes factors as $\sigma\to\eta\to\tau$, where the first map $\sigma\to\eta$ is a natural projection onto a face of $\sigma$ and the second map $\eta\to\tau$ is an isometry. 
\end{definition}

\subsection{Right-angled Artin groups}
\label{subsec_raag}
Pick a finite simplicial graph $\Gamma$, recall that $G(\Ga)$ is the right-angled Artin group with defining graph $\Ga$. Let $S$ be a standard generating set for $G(\Gamma)$ and we label the vertices of $\Gamma$ by elements in $S$. $G(\Gamma)$ has a nice Eilenberg-MacLane space $S(\Gamma)$, called the Salvetti complex (see \cite{MR1368655,charney2007introduction}). Recall that $S(\Gamma)$ is the graph product $\prod_\Ga(S^1_v,\star_v)$, where $(S^1_v,\star_v)$ is a pointed unit circle (see Definition \ref{def_graph_product_of_spaces}).

The 2-skeleton of $S(\Gamma)$ is the usual presentation complex of $G(\Gamma)$, so $\pi_{1}(S(\Gamma))\cong G(\Gamma)$. The 0-skeleton of $S(\Gamma)$ consists of one point whose link is a flag complex, so $S(\Gamma)$ is non-positively curved and $S(\Gamma)$ is an Eilenberg-MacLane space for $G(\Gamma)$ by the Cartan-Hadamard theorem (\cite[Theorem II.4.1]{bridson_haefliger}).

The closure of each $k$-cell in $S(\Gamma)$ is a $k$-torus. Tori of this kind are called \textit{standard tori}. There is a 1-1 correspondence between the $k$-cells (or standard torus of dimension $k$) in $S(\Gamma)$ and $k$-cliques in $\Gamma$. We define the \textit{dimension} of $G(\Gamma)$ to be the dimension of $S(\Gamma)$.

Denote the universal cover of $S(\Gamma)$ by $X(\Gamma)$, which is a $\cat(0)$ cube complex. Our previous labelling of vertices of $\Gamma$ induces a labelling of the standard circles of $S(\Gamma)$, which lifts to a labelling of edges of $X(\Gamma)$. A \textit{standard k-flat} in $X(\Gamma)$ is a connected component of the inverse image of a standard $k$-torus under the covering map $X(\Gamma)\to S(\Gamma)$. When $k=1$, we also call it a \textit{standard geodesic}.

For each simplicial graph $\Gamma$, there is a simplicial complex $\mathcal{P}(\Gamma)$ called the \textit{extension complex}, which captures the combinatorial pattern of how standard flats intersect each other in $X(\Gamma)$. This object was first introduced in \cite{kim2013embedability}. We will define it in a slightly different way (see \cite[Section 4.1]{huang2014quasi} for more discussion).

\begin{definition}[Extension complex]
\label{def_extension_complex}
The vertices of $\mathcal{P}(\Gamma)$ are in 1-1 correspondence with the parallel classes of standard geodesics in $X(\Gamma)$. Two distinct vertices $v_{1},v_{2}\in\mathcal{P}(\Gamma)$ are connected by an edge if and only if there is standard geodesic $l_{i}$ in the parallel class associated with $v_{i}$ ($i=1,2$) such that $l_{1}$ and $l_{2}$ span a standard 2-flat. Then $\mathcal{P}(\Gamma)$ is defined to be the flag complex of its 1-skeleton, namely we build $\mathcal{P}(\Gamma)$ inductively from its 1-skeleton by filling a $k$-simplex whenever we see the $(k-1)$-skeleton of a $k$-simplex. 
\end{definition}

Since each complete subgraph in the 1-skeleton of $\mathcal{P}(\Gamma)$ gives rise to a collection of mutually orthogonal standard geodesics lines, there is a 1-1 correspondence between $k$-simplexes in $\mathcal{P}(\Gamma)$ and parallel classes of standard $(k+1)$-flats in $X(\Gamma)$. In particular, there is a 1-1 correspondence between maximal simplexes in $\mathcal{P}(\Gamma)$ and maximal standard flats in $X(\Gamma)$. Given standard flat $F\subset X(\Gamma)$, we denote the simplex in $\mathcal{P}(\Gamma)$ associated with the parallel class containing $F$ by $\Delta(F)$. 

\begin{lem}
\label{projection}
Pick non-adjacent vertices $v,u\in\mathcal{P}(\Gamma)$ and let $l_{u}$ be a standard geodesic such that $\Delta(l_{u})=u$. Then for any standard geodesic line $l$ with $\Delta(l)=v$, $\pi_{l_{u}}(l)$ is a vertex of $l_{u}$ which does not depend on the choice of $l$ in the parallel class. Thus we define this vertex to be the projection of $v$ on $l_{u}$ and denote it by $\pi_{l_{u}}(v)$.
\end{lem}

\begin{proof}
It follows from Lemma \ref{project onto line} that if two standard geodesic lines $l_{1}$ and $l_{2}$ are not parallel, then $\pi_{l_{1}}(l_{2})$ is a vertex of $l_{2}$. Thus $\pi_{l_{u}}(l)$ is a vertex of $l_{u}$. Suppose $l'$ is parallel to $l$, then there exists a chain of parallel standard geodesic lines $\{l_{i}\}_{i=1}^{n}$ such that $l_{1}=l$, $l_{n}=l'$ and $l_{i},l_{i+1}$ are in the same standard 2-flat $F_{i}$. Let $l^{\perp}\subset F_{i}$ be a standard geodesic orthogonal to $l_{i}$ and $l_{i+1}$. Since $v$ and $u$ are not adjacent in $\mathcal{P}(\Gamma)$, $l^{\perp}$ and $l_{u}$ are not parallel, thus $\pi_{l_{u}}(l^{\perp})$ is a point. It follows that $\pi_{l_{u}}(l_{i})=\pi_{l_{u}}(l_{i+1})$, hence $\pi_{l_{u}}(l)=\pi_{l_{u}}(l')$.
\end{proof}

\subsection{Right-angled buildings}
\label{building}
We will follow the treatment in \cite{davis_buildings_are_cat0,abramenko2008buildings,ronan2009lectures}. In particular, we refer to Section 1.1 to Section 1.3 of \cite{davis_buildings_are_cat0} for the definitions of chamber systems, galleries, residues, Coxeter groups and buildings. We will focus on right-angled buildings,  i.e. the associated Coxeter group is right-angled, though most of the discussion below is valid for general buildings.

Let $W=W(\Ga)$ be a right-angled Coxeter group with (finite) defining graph $\Ga$. Let $\B=\B(\Ga)$ be a right-angled building with the associated $W$-distance function denoted by $\delta:\B\times \B\to W$. We will also call $\B(\Ga)$ a \textit{right-angled $\Ga$-building} for simplicity.

Let $I$ be the vertex set of $\Ga$. Recall that a subset $J\subset I$ is \textit{spherical} if the subgroup of $W$ generated by $J$ is finite. Let $\mathsf{S}$ be the poset of spherical subsets of $I$ (including the empty set) and let $|\mathsf{S}|_{\Delta}$ be the geometric realization of $\mathsf{S}$, i.e. $|\mathsf{S}|_{\Delta}$ is a simplicial complex such that its vertices are in 1-1 correspondence to elements in $\mathsf{S}$ and its $n$-simplices are in 1-1 correspondence to $(n+1)$-chains in $\mathsf{S}$. Note that $|\mathsf{S}|_{\Delta}$ is isomorphic the simplicial cone over the barycentric subdivision of the flag complex of $\Ga$.

Recall that for elements $x\le y$ in $\mathsf{S}$, the \textit{interval} $I_{xy}$ between $x$ and $y$ is a poset consist of elements $z\in\mathsf{S}$ such that $x\le z\le y$ with the induced order from $\mathsf{S}$. There is a natural simplicial embedding $|I_{xy}|_{\Delta}\hookrightarrow|\mathsf{S}|_{\Delta}$. Each $|I_{xy}|_{\Delta}$ is a simplicial cone over the barycentric subdivision of a simplex, thus can be viewed a subdivision of a cube into simplices. It is not hard to check the collection of all intervals in $\mathsf{S}$ gives rise to a structure of cube complex on $|\mathsf{S}|_{\Delta}$. Let $|\mathsf{S}|$ be the resulting cube complex, then $|\mathsf{S}|$ is $CAT(0)$.

A residue is \textit{spherical} if it is a $J$-residue with $J\in \mathsf{S}$. Let $ \res$ be the poset of all spherical residues in $\B$. For $x\in  \res$ which comes from a $J$-residue, we define the \textit{rank} of $x$ to be the cardinality of $J$, and define a \textit{type map} $t: \res\to\mathsf{S}$ which maps $x$ to $J\in \mathsf{S}$. Let $|\res|_{\Delta}$ be the geometric realization of $\res$, then the type map induces a simplicial map $t:|\res|_{\Delta}\to |\mathsf{S}|_{\Delta}$. For $x\in \res$, let $\res_{x}$ be the sub-poset made of elements in $\res$ which is $\ge x$. If $x$ is of rank $0$, then $\res_{ x}$ is isomorphic to $\mathsf{S}$, moreover, there is a natural simplicial embedding $|\res_{x}|_{\Delta}\to|\res|_{\Delta}$ and $t$ maps the image of $|\res_{x}|_{\Delta}$ isomorphically onto $|\mathsf{S}|_{\Delta}$. 

As before, the geometric realization of each interval in $\res$ is a subdivision of a cube into simplices. Moreover, the intersection of two intervals in $\res$ is also an interval. Thus one gets a cube complex $|\B|$ whose cubes are in 1-1 correspondence with intervals in $\res$. $|\B|$ is called the \textit{Davis realization} of the building $\B$ and $|\B|$ is a $CAT(0)$ cube complex by \cite{davis_buildings_are_cat0}. Moreover, the above type map induces a cubical map $t:|\B|\to|\mathsf{S}|$. Let $\calr\subset \B$ be a residue. Since $\calr$ also has the structure of a building, there is an isometric embedding $|\calr|\to|\B|$ between their Davis realizations. $|\calr|$ is called a \textit{residue} in $|\B|$.

In the special case when $\B$ is equal to the associated Coxeter group $W$, there is a natural embedding from the Cayley graph of $W$ to $|\B|$ such that vertices of Cayley graph are mapped to vertices of rank $0$ in $|\B|$. And $|\B|$ can be viewed as the first cubical subdivision of the cubical completion of the Cayley graph of $W$ (the cubical completion means we attach an $n$-cube to the graph whenever there is a copy of the 1-skeleton of an $n$-cube inside the graph). 

Each vertex of $|\B|$ corresponds to a $J$-residue in $\B$, thus has a well-defined rank. For a vertex $x$ of rank $0$, the space $|\res_{x}|_{\Delta}$ discussed in the previous paragraph induces a subcomplex $|\B_x|\subset|\B|$. Note that $t$ maps $|\B_x|$ isomorphically onto $|\mathsf{S}|$. $|\B_x|$ is called a \textit{chamber} in $|\B|$, and there is a 1-1 correspondence between chambers in $|\B|$ and chambers in $\B$. Let $|\B_x|$ and $|\B_y|$ be two chambers in $|\B|$. Since there is an apartment $\A\subset \B$ which contains both $x$ and $y$, this induces an isometric embedding $|\A|\to |\B|$ whose image contains $|\B_x|$ and $|\B_y|$, here $|\A|$ is isomorphic to the Davis realization of the Coxeter group $W$. $|\A|$ is called an \textit{apartment} in $|\B|$.

\begin{definition}
\label{distance of residue}
For $c_1,c_2\in \B$, define $d(c_1,c_2)$ to be the minimal length of word in $W$ (with respect to the generating set $I$) that represents $\delta(c_1,c_2)$. For any two residues $\calr_1,\calr_2\in\B$, we define $d(\calr_1,\calr_2)=\min\{d(c,d)\mid c\in\calr_1,d\in\calr_2\}$. It turns out that for any $c\in \calr_1$ and $d\in\calr_2$ with $d(c,d)=d(\calr_1,\calr_2)$, $\delta(c,d)$ gives rise to the same element in $W$ (\cite[Chapter 5.3.2]{abramenko2008buildings}), this element is defined to be $\delta(\calr_1,\calr_2)$.
\end{definition}

\begin{lem}
\label{distance}
$d(c_1,c_2)=2d_{l^1}(c_1,c_2)$, here $d_{l^1}$ means the $l^1$-distance in $|\B|$. Since $c_1$ and $c_2$ can be also viewed as vertex of rank $0$ in $|\B|$, $d_{l^1}(c_1,c_2)$ makes sense.
\end{lem}

\begin{proof}
If $\B=W$, then this lemma follows from the above description of the Davis realization of a Coxeter group. The general case can be reduced to this case by considering an apartment $|\A|\subset|\B|$ which contains $c_1$ and $c_2$. Note that $|\A|$ is convex in $|\B|$.
\end{proof}

Given a residue $\calr\subset\B$, there is a well-defined nearest point projection map as follows.

\begin{theorem}[Proposition 5.34, \cite{abramenko2008buildings}]
\label{prj}
Let $\calr$ be a residue and $c$ a chamber. Then there exists a unique $c'\in\calr$ such that $d(c,c')=d(\calr,c)$.
\end{theorem}

This projection is compatible with several other projections in the following sense. Let $|\calr|\subset|\B|$ be the convex subcomplex corresponding to $\calr$. Let $c$ and $c'$ be as above. We also view them as vertex of rank $0$ in $|\B|$. Let $c_1$ be the combinatorial projection of $c$ onto $|\calr|$ (see \cite[Lemma 13.8]{haglund_wise_special}) and let $c_2$ be the $CAT(0)$ projection of $c$ onto $|\calr|$.

\begin{lem}
\label{equivalence of projections}
$c'=c_1=c_2$.
\end{lem} 

\begin{proof}
$c_1=c_2$ is actually true for any $CAT(0)$ cube complexes. By \cite[Lemma 2.3]{huang2014quasi}, $c_2$ is a vertex. If $c_2\neq c_1$, by \cite[Lemma 13.8]{haglund_wise_special}, the concatenation of the combinatorial geodesic $\omega_1$ which connects $c_2$ and $c_1$ and the combinatorial geodesic $\omega_2$ which connects $c_1$ and $c$ is a combinatorial geodesic connecting $c$ and $c_2$. Note that $\omega_1\subset |\calr|$. Let $e\subset\omega_1$ be the edge that contains $c_2$ and let $x$ be the other endpoint of $e$. Then $x$ and $c$ are in the same side of the hyperplane dual to $e$. It is easy to see $d(x,c)<d(c_2,c)$ (here $d$ denotes the $CAT(0)$ distance), which yields a contradiction.

To see $c'=c_1$, by Lemma \ref{distance}, it suffices to prove $c_1$ is of rank $0$. When $\B=W$, this follows from $c_1=c_2$, since we can work with the cubical completion of the Cayley graph of $W$ instead of $|W|$ (the latter is the cubical subdivision of the former) and apply \cite[Lemma 2.3]{huang2014quasi}. The general case follows by considering an apartment $|\A|\subset|\B|$ which contains $c_1$ and $c$, note that in this case $|\A|\cap|\calr|$ can be viewed as a residue in $|\A|$.
\end{proof}

\begin{definition}
\label{definition_parallel}
Let $\textmd{proj}_{\calr}$ be the map defined in Theorem \ref{prj}. Two residues $\calr_1$ and $\calr_2$ are \textit{parallel} if $\textmd{proj}_{\calr_1}(\calr_2)=\calr_1$ and $\textmd{proj}_{\calr_2}(\calr_1)=\calr_2$. In this case $\textmd{proj}_{\calr_1}$ and $\textmd{proj}_{\calr_2}$ induce mutually inverse bijections between $\calr_1$ and $\calr_2$. These bijections are called \textit{parallelism maps} between $\calr_1$ and $\calr_2$. They are also isomorphisms of chamber system i.e. they map residues to residues (\cite[Proposition 5.37]{abramenko2008buildings}).
\end{definition}

It follows from the uniqueness of the projection map that if $f:\calr\to\calr'$ is the parallelism map between two parallel residues and $\calr_1\subset\calr$ is a residue, then $\calr_1$ and $f(\calr_1)$ are parallel, and the parallelism map between $\calr_1$ and $f(\calr_1)$ is induced by $f$.

\begin{lem}
\label{parallel}
If $\calr_1$ and $\calr_2$ are parallel, then $|\calr_1|$ and $|\calr_2|$ are parallel with respect to the $CAT(0)$ metric on $|\B|$. Moreover, the parallelism maps between $\calr_1$ and $\calr_2$ induces by $\textmd{proj}_{\calr_1}$ and $\textmd{proj}_{\calr_2}$ is compatible with the $CAT(0)$ parallelism between $|\calr_1|$ and $|\calr_2|$ induced by $CAT(0)$ projections.
\end{lem}

\begin{proof}
By Lemma \ref{equivalence of projections}, it suffices to show for any residue $\calr\in\B$, $|\calr|$ is the convex hull of the vertices of rank $0$ inside $|\calr|$. This is clear when $\B=W$ if one consider the cubical completion of the Cayley graph of $W$. The general case also follows since $|\calr|$ is a union of apartments in $|\calr|$, and $|\calr|$ is convex in $|\B|$.
\end{proof}

It follows that if $\calr_1$ and $\calr_2$ are parallel residues, and $\calr_2$ and $\calr_3$ are parallel residues, then $\calr_1$ is parallel to $\calr_3$. Moreover, let $f_{ij}$ be the parallelism map from $\calr_i$ to $\calr_j$ induced by the projection map, then $f_{13}=f_{23}\circ f_{12}$.

Given chamber systems $C_1,\cdots,C_k$ over $I_1,\cdots,I_k$, their direct product $C_{1}\times \cdots\times C_k$ is a chamber system over the disjoint union $I_1\sqcup\cdots\sqcup I_k$. Its chambers are $k$-tuples $(c_1,\cdots,c_k)$ with $c_t\in C_t$. For $i\in I_t$, $(c_1,\cdots,c_k)$ is $i$-adjacent to $(d_1,\cdots,d_k)$ if $c_j=d_j$ for $j\neq t$ and $c_t$ and $d_t$ are $i$-adjacent.

Suppose the defining graph $\Ga$ of the right-angled Coxeter group $W$ admits a join decomposition $\Ga=\Ga_1\circ\Ga_2\circ\cdots\circ\Ga_k$. Let $I=\cup_{i=1}^{k}I_i$ be the corresponding decomposition of the vertex set of $\Ga$ and $W=\prod_{i=1}^{k}W_i$ be the induced product decomposition of $W$. Pick chamber $c\in\B$, and let $\B_i$ be the $I_i$-residue that contains $c$. Define a map $\phi:\B\to\B_1\times\B_2\times\cdots\times\B_k$ by $\phi(d)=(\textmd{proj}_{\B_1}(d),\textmd{proj}_{\B_2}(d),\cdots,\textmd{proj}_{\B_k}(d))$ for any chamber $d\in\B$.

\begin{thm}[Theorem 3.10, \cite{ronan2009lectures}]
\label{product decomposition}
The definition of $\phi$ does not depend on the choice of $c$, and $\phi$ is an isomorphism of buildings.
\end{thm}

It follows from the definition of the Davis realization that there is a natural isomorphism $|\B_1\times\B_2\times\cdots\times\B_k|\cong|\B_1|\times|\B_2|\times\cdots\times|\B_k|$, thus we have a product decomposition $|
\B|\cong|\B_1|\times|\B_2|\times\cdots\times|\B_k|$, where the isomorphism is induced by $CAT(0)$ projections from $|\B|$ to $|\B_i|$'s (this is a consequence of Lemma \ref{parallel}).

We define the \textit{parallel set} of a residue $\calr\subset\B$ to be the union of all residues in $\B$ that are parallel to $\calr$.

\begin{lem}
\label{parallel sets of residues}
Suppose $\calr$ is a $J$-residue. Let $J^{\perp}\subset I$ be the collection of vertices in $\Ga$ which are adjacent to every vertex in $J$. Then: 
\begin{enumerate}
\item If $\calr'$ is parallel to $\calr$, then $\calr'$ is a $J$-residue.
\item The parallel set of $\calr$ is the $J\cup J^{\perp}$-residue that contains $\calr$.
\end{enumerate}
\end{lem}

Note that this lemma is not true if the building under consideration is not right-angled.
\begin{proof}
Suppose $\calr'$ is a $J_1$-residue. Let $w=\delta(\calr,\calr')$ (see Definition \ref{distance of residue}). It follows from (2) of \cite[Lemma 5.36]{abramenko2008buildings} that $\calr'$ is a $(J\cap w J_1 w^{-1})$-residue. Since $\calr$ and $\calr'$ are parallel, they have the same rank, thus $J=w J_{1} w^{-1}$. By considering the abelianization of the right-angled Coxeter group $W$, we deduce that $J=J_1$ (this proves the first assertion of the lemma) and $w$ commutes with each element in $J$. Thus $w$ belongs to the subgroup generated by $J^{\perp}$ and $\calr'$ is in the $J\cup J^{\perp}$-residue $\S$ that contains $\calr$. Then the parallel set of $\calr$ is contained in $\S$. It remains to prove every $J$-residue in $\S$ is parallel to $\calr$, but this follows from Theorem \ref{product decomposition}.
\end{proof}

Pick a vertex $v\in|\B|$ of rank $k$ and let $\calr=\prod_{i=1}^{k}\calr_i$ be the associated residue with its product decomposition. Let $\{v_{\lambda}\}_{\lambda\in\Lambda}$ be the collection of vertices that are adjacent to $v$. Then there is a decomposition $\{v_{\lambda}\}_{\lambda\in\Lambda}=\{v_{\lambda}\le v\}\sqcup \{v_{\lambda}>v\}$, where $\{v_{\lambda}>v\}$ denotes the collection of vertices whose associated residues contain $\calr$. This induces a decomposition $Lk(v,|\B|)=K_1\ast K_2$ of the link of $v$ in $|\B|$ (\cite[Definition I.7.15]{bridson_haefliger}) into a spherical join of two $CAT(1)$ all-right spherical complexes. Note that $K_2$ is finite, since $\{v_{\lambda>v}\}$ is finite. Moreover, $K_1\cong Lk(v,|\calr|)$. However, $|\calr|\cong\prod_{i=1}^{k}|\calr_i|$, thus $K_1$ is the spherical join of $k$ discrete sets such that elements in each of these discrete sets are in 1-1 correspondence to elements in some $\calr_i$. Now we can deduce from this the following result.
\begin{lem}
\label{rank preserving}
Suppose $\B$ is a right-angled building such that each of its residues of rank 1 
contains infinitely many elements. If $\alpha:|\B|\to|\B|$ is a cubical isomorphism, then $\alpha$ preserves the rank of vertices in $|\B|$.
\end{lem}

\section*{Part I: A fibration approach to cubulating RAAG's}

\section{Restriction quotients}
\label{sec_restriction_quotients}
In this section we study restriction quotients, 
a certain type of mapping between $\cat(0)$ cube complexes 
introduced by Caprace and Sageev in \cite{caprace2011rank}.  These
play a central role in our story. 

We first show in Subsection \ref{subsec_quotient_maps_cube_complexes}
that restriction quotients can be characterized in 
several different ways, see Theorem \ref{thm_restriction_quotient_characterization}.  
We then show in Subsection \ref{sec_restriction_maps_versus_fiber_functors}
that a restriction
quotient $f:Y\ra Z$ determines fiber data that satisfies
certain conditions; conversely, given such fiber data, one may construct
a restriction quotient inducing the given data, which is unique up to equivalence.
This correspondence will later be applied to construct restriction quotients over
right-angled buildings.  Subsections \ref{sec_equivariance_properties}
and \ref{sec_quasiisometric_properties} deals with the behavior of restriction
quotients under group actions and quasi-isometries.

\subsection{Quotient maps between $\cat(0)$ cube complexes}
\label{subsec_quotient_maps_cube_complexes}
We recall the notion of restriction quotient from \cite[Section 2.3]{caprace2011rank}; see Section \ref{subsec_wallspaces} for the background on wallspaces.

\begin{definition}
\label{def_restriction_quotient}
Let $Y$ be a $\cat(0)$ cube complex and let $\h$ be the collection of walls in the $0$-skeleton $Y^{(0)}$ corresponding to the hyperplanes in $Y$. Pick a subset $\K\subset \h$ and let $Y(\K)$ be the $\cat(0)$ cube complex associated with the wallspace $(Y^{(0)},\K)$. Then every 0-cube of the wallspace $(Y^{(0)},\h)$ gives rise to a 0-cube of $(Y^{(0)},\K)$ by restriction. This can be extended to a surjective cubical map $q:Y\to Y(\K)$, which is called the {\em restriction quotient} arising from the subset $\K\subset\h$.
\end{definition}

The following example motivates many of the constructions
in this paper:
\begin{example}[The canonical restriction quotient of a RAAG]
\label{the key example}
For a fixed graph $\Ga$,
let $S_e\ra |c|_\Ga$ and $X_e\ra S_e$ be the mappings associated with the 
exploded Salvetti complex, as defined in the introduction after
Definition \ref{def_graph_product_of_spaces}.
Let $\K$ be the collection
of hyperplanes in $X_e(\Ga)$ dual to edges
$e\subset X_e$ that project to edges under the composition 
$X_e\ra S_e\ra |c|_\Ga$.  Then the {\em canonical restriction quotient of $G=G(\Ga)$}
is the restriction quotient arising from $\K$.
\end{example}

Let $q:Y\to Y(\K)$ be a restriction quotient. Pick an edge $e\subset Y$. If $e$ is dual to some element in $\K$, then $q(e)$ is an edge, otherwise $q(e)$ is a point. The edge $e$ is called \textit{horizontal} in the former case, and \textit{vertical} in the latter case. We record the following simple observation.
\begin{lem}
\label{descending automorphism general case}
Let $\alpha: Y\ra Y$ be a cubical $\cat(0)$  automorphism of $Y$ that maps vertical 
edges to vertical 
edges and horizontal edges to horizontal edges. Then $\alpha$ descends to an 
automorphism  $Y(\K)\ra Y(\K)$.
\end{lem}

The following result shows that restriction quotients may be characterized in 
several different ways.
\begin{thm}
\label{thm_restriction_quotient_characterization}
If $f:Y\to Z$ is a surjective cubical map between two $\cat(0)$ cube complexes, then the following conditions are equivalent:
\begin{enumerate}
\item The inverse image of each vertex of $Z$ is convex.
\item The inverse image of every point in $Z$ is convex.
\item The inverse image of every convex subcomplex of $Z$ is convex.
\item The inverse image of every hyperplane in $Z$ is a hyperplane.
\item $f$ is equivalent to a restriction quotient, i.e. for some set of walls
$\K$ in $Y$, there is a cubical isomorphism $\phi:Z\ra Y(\K)$
making the following   diagram commute:
\begin{diagram}
Y       &&\\
\dTo^f  & \rdTo^q &\\
Z       & \rTo_\phi   & Y(\K)
\end{diagram}
\end{enumerate}
\end{thm}

The proof of Theorem 
\ref{thm_restriction_quotient_characterization} will take several lemmas.
For the remainder of this subsection we 
fix  $\cat(0)$ cube complexes $Y$ and $Z$ and a (not necessarily
surjective) cubical map $f:Y\to Z$. 

\begin{lemma}
\label{lem_fiber_functor}
Let $\sigma\subset Z$ be a cube and let $Y_{\sigma}$ be the be the union of cubes in $Y$ whose image under $f$ is exactly $\sigma$. Then:
\begin{enumerate}
\item If $y\in \sigma$ is an interior point, then $f^{-1}(y)\subset Y_{\sigma}$.
\item $f^{-1}(y)$ has a natural induced structure as a cube complex; moreover, there is a natural isomorphism of cube complexes $Y_{\sigma}\cong f^{-1}(y)\times \sigma$.
\item If $\si_1\subset\si_2$ are cubes of $Z$ and $y_i\in\si_i$ are interior
points, then there is a canonical embedding $f^{-1}(y_2)\hookrightarrow f^{-1}(y_1)$.
Moreover, these embeddings are compatible with composition of inclusions.
\end{enumerate}
\end{lemma}

\begin{lem}

\mbox{}
\begin{enumerate}
\item For every  $y\in Z$, every 
connected component of $f^{-1}(y)$ is a convex subset of $Y$.
\item For every convex subcomplex $A\subset Z$, every connected component of $f^{-1}(A)$ is a convex subcomplex of $Y$.
\end{enumerate}
\end{lem}

\begin{proof}
First we prove (1). Let $\sigma$ be the support of $y$ and let $Y_{\sigma}\cong f^{-1}(y)\times \sigma$ be the subcomplex defined as above. It suffices to show $Y_{\sigma}$ is locally convex. Pick vertex $x\in Y_{\sigma}$, and let $\{e_{i}\}_{i=1}^{n}$ be a collection of edges in $Y_{\sigma}$ that contains $x$. It suffices to show if these edges span an $n$-cube $\eta\subset Y$, then $\eta\subset Y_{\sigma}$. It suffices to consider the case when all $e_{i}$'s are orthogonal to $\sigma$, in which case it follows from Definition \ref{cubical map} that $\eta\times \sigma\subset Y_{\sigma}$.

To see (2), pick an $n$-cube $\eta\subset Y$ and let $\{e_{i}\}_{i=1}^{n}$ be the edges of $\eta$ at one corner $c\subset \eta$.
It suffices to show if $f(e_{i})\subset A$, then $f(\eta)\subset A$. Note that $f(\eta)$ is a cube, and every edge of this cube which emanates from the corner $f(c)$ is contained in $A$.  Thus $f(\eta)\subset A$ by the convexity of $A$.
\end{proof}

\begin{lem}
\label{inverse image of hyperplane}
Let $f:Y\to Z$ be a cubical map as above. Then:
\begin{enumerate}
\item The inverse image of each hyperplane of $Z$ is a disjoint union of hyperplanes in $Y$.
\item If the inverse image of each hyperplane of $Z$ is a single hyperplane, then for each point $y\in Z$, the point inverse $f^{-1}(y)$ is connected, and hence convex. 
\end{enumerate}
\end{lem}

\begin{proof}
It follows from Definition \ref{cubical map} that the inverse image of each hyperplane of $Z$ is an union of hyperplanes. If two of them were to 
intersect, then there would be a 2-cube in $Y$ with two consecutive edges  
mapped to the same edge in $Z$, which is impossible.

Now we prove (2). It suffices to consider the case that $y$ is the center of some cube in $Z$. In this case, $y$ is a vertex in the first cubical subdivision of $Z$, and $f$ can viewed as a cubical map from the first cubical subdivision of $Y$ to the first cubical subdivision of $Z$ such that the inverse image of each hyperplane is a single hyperplane, thus it suffices to consider the case that $y$ is a vertex of $Z$.

Suppose $f^{-1}(y)$ contains two connected components $A$ and $B$. Pick a combinatorial geodesic $\omega$ of shortest distant that connects vertices in $A$ and vertices in $B$. Note that $f(\omega)$ is a non-trivial edge-loop in $Z$, otherwise we will have $\omega\subset f^{-1}(y)$. It follows that there exists two different edges $e_{1}$ and $e_{2}$ of $\omega$ mapping to parallel edges in $Y$. The hyperplanes dual to $e_{1}$ and $e_{2}$ are different, yet they are mapped to the same hyperplane in $Y$, which is a contradiction.
\end{proof}

\begin{lem}
\label{vertex to hyperplane}
If $f$ is surjective, and for any vertex $v\in Z$, $f^{-1}(v)$ is connected, then the inverse image of each hyperplane of $Z$ is a single hyperplane.
\end{lem}

\begin{proof}
Let $h\subset Z$ be a hyperplane, by Lemma \ref{inverse image of hyperplane}, $f^{-1}(h)=\sqcup_{\lambda\in\Lambda}h_{\lambda}$ where each $h_{\lambda}$ is a hyperplane in $Y$. Since $f$ is surjective, $\{f(h_{\lambda})\}_{\lambda\in\Lambda}$ is a collection of subcomplexes of $h$ that cover $h$. Thus there exists $h_1,h_2\in\{h_{\lambda}\}_{\lambda\in\Lambda}$ and vertex $u\in h$ such that $u\subset f(h_1)\cup f(h_2)$. Let $e\subset Z$ be the edge such that $u=e\cap h$, then there exist edges $e_{1},e_{2}\subset Y$ such that $e_{i}\cap h_{i}\neq\emptyset$ and $f(e_{i})=e$ for $i=1,2$. Since $h_{1}\cap h_{2}=\emptyset$, a case study implies there exist $x_{1}$ and $x_{2}$ which are endpoints of $e_1$ and $e_2$ respectively such that 
\begin{enumerate}
\item these two points are separated by at least one of $h_{1}$ and $h_2$;
\item they are mapped to the same end point $y\in e$.
\end{enumerate}
It follows that $f^{-1}(y)$ is disconnected, which is a contradiction.
\end{proof}

\begin{remark}
If $f$ is not surjective, then the above conclusion is not necessarily true. Consider the map from $A=[0,3]\times[0,1]$ to the unit square which collapses the $[0,1]$ factor in $A$ and maps $[0,3]$ to 3 consecutive edges in the boundary of the unit square.
\end{remark}

\begin{lem}
\label{restriction quotient}
If $q:Y\to Y(\K)$ is the restriction quotient as Definition
\ref{def_restriction_quotient}, 
then the inverse image of each hyperplane in $Y(\K)$ is a single hyperplane in $Y$. Conversely, suppose $f:Y\to Z$ is a surjective cubical map between $\cat(0)$ cube complexes such that the inverse image of each hyperplane is a hyperplane. Let $\K$ be the collection of walls arising from inverse images of hyperplanes in $Z$. Then there is a natural isomorphism $i: Z\cong Y(\K)$ which fits into the following commutative diagram:
\begin{diagram}
Y &\rTo^f &Z\\
&\rdTo_q &\dTo^i \\
& &Y(\K)
\end{diagram}
\end{lem}

\begin{proof}
Define two vertices of $Y$ to be $\K$-equivalent if and only if they are not separated by any wall in $\K$. This defines an equivalence relation on vertices of $Y$, and the corresponding equivalent classes are called $\K$-classes. For each $\K$-class $C$ and every wall in $\K$, we may choose the halfspace that contains $C$; it follows that the points in $C$ are exactly the set of vertices contained in the intersection of such halfspaces, and thus $C$ is the vertex set of a convex subcomplex of $Y$. Note that each $\K$-class determines a 0-cube of $(Y^{0},\K)$, hence is mapped to this 0-cube under $q$. It follows that the inverse image of every vertex in $Y(\K)$ is convex, thus by Lemma \ref{vertex to hyperplane}, the inverse image of a hyperplane is a hyperplane.

It remains to prove the converse. Note that the inverse image of each halfspace in $Z$ under $f$ is a halfspace of $Y$. Moreover, the surjectivity of $f$ implies that $f$ maps hyperplane to hyperplane and halfspace to halfspace. Pick vertex $y\in Z$, let $\{H_{\lambda}\}_{\lambda\in\lambda}$ be the collection of hyperplanes in $Z$ that contains $y$. Then $f^{-1}(y)\subset \cap_{\lambda\in\Lambda}f^{-1}(H_{\lambda})$, and every vertex of $\cap_{\lambda\in\Lambda}f^{-1}(H_{\lambda})$ is mapped to $y$ by $f$, and thus the vertex set of $f^{-1}(y)$ is a $\K$-class. This induces a bijective map from $Z^{(0)}$ to the vertex set of $Y(\K)$, which extends to an isomorphism. The above diagram commutes since it commutes when restricted to the 0-skeleton.
\end{proof}

\begin{proof}[Proof of Theorem \ref{thm_restriction_quotient_characterization}]
The equivalence of (4) and (5) follows from Lemma \ref{restriction quotient}. $(1)\Rightarrow(4)$ follows from Lemma \ref{vertex to hyperplane}, $(4)\Rightarrow(2)$ follows from Lemma \ref{inverse image of hyperplane}, $(3)\Rightarrow(1)$ is obvious. It suffices to show $(2)\Rightarrow(3)$. Pick a convex subcomplex $K\subset Z$ and let $\{R_{\lambda}\}_{\lambda\in\Lambda}$ be the collection of cubes in $K$. For each $R_{\lambda}$, let $Y_{R_{\lambda}}$ be the subcomplex defined after Definition \ref{cubical map}. $Y_{R_{\lambda}}\neq\emptyset$ since $f$ is surjective and $Y_{R_{\lambda}}$ is connected by (2). If $R_{\lambda}\subset R_{\lambda'}$, then $Y_{R_{\lambda}}\cap Y_{R_{\lambda'}}\neq\emptyset$. Thus $f^{-1}(K)=\cup_{\lambda\in\Lambda} Y_{R_{\lambda}}$ is connected, hence convex.
\end{proof}

\subsection{Restriction maps versus fiber functors}
\label{sec_restriction_maps_versus_fiber_functors}


If $q:Y\ra Z$ is a restriction quotient between $\cat(0)$ cube complexes, then
 we may express the fiber structure in categorical
language as follows.  Let $\face(Z)$ denote the face poset of $Z$, viewed as a
category, and let 
 $\ccc$  
denote the category whose objects are  nonempty $\cat(0)$
cube complexes and whose morphisms are convex cubical embeddings. 
By Lemma \ref{lem_fiber_functor}, we obtain a contravariant functor 
$\Psi_q:\face(Z)\ra \ccc$.

\begin{definition}
The contravariant functor $\Psi_q$ is the {\em fiber functor} of the 
restriction quotient $q:Y\ra Z$. 
\end{definition}
For notational brevity, for any inclusion $i:\si_1\ra\si_2$, we 
will often denote the map $\Psi(i):\Psi(\si_2)\ra \Psi(\si_1)$ simply
by $\Psi(\si_2)\ra \Psi(\si_1)$, suppressing the name of the map.


Note that if $\si_1\subset\si_2\subset\si_3$, then the functor property
implies that the image of $\Psi(\si_3)\ra \Psi(\si_1)$ is a convex
subcomplex of the image of $\Psi(\si_2)\ra \Psi(\si_1)$.  In particular,
if $v$ is a vertex of a cube $\si$, then the image of
$\Phi(\si)\ra \Psi(v)$ is a convex subcomplex of the intersection
$$
\bigcap_{v\subsetneq e\subset\si^{(1)}}\im(\Psi(e)\ra \Psi(v))
$$

\begin{definition}
Let $Z$ be a cube complex.
A contravariant functor $\Psi:\face(Z)\ra \ccc$ is {\bf $1$-determined}
if for every cube $\si\in\face(Z)$, and every vertex $v\in\si^{(0)}$,
\begin{equation}
\label{eqn_1_determined}
\im(\Psi(\si)\lra \Psi(v))
=\bigcap_{v\subsetneq e\subset\si^{(1)}}\im(\Psi(e)\ra \Psi(v))\,.
\end{equation}
\end{definition}

\begin{lemma}
If $q:Y\ra Z$ is a  restriction quotient,
then the fiber functor $\Psi:\face(Z)\ra\ccc$ 
is $1$-determined.
\end{lemma}
\begin{proof}
 Pick $\si\in\face(Z)$, $v\in \si^{(0)}$.
We know that $\im(\Psi(\si)\ra \Psi(v))$ is a nonempty convex subcomplex
of $\cap_{v\subsetneq e\subset \si^{(1)}}\im(\Psi(e)\ra \Psi(v))$, so to 
establish (\ref{eqn_1_determined}) we need only show that the two convex subcomplexes
have the same $0$-skeleton.

Pick a vertex  $w\in \im(\Psi(\si)\ra \Psi(v))$, and let
$w'\in \cap_{v\subsetneq e\subset \si^{(1)}}\im(\Psi(e)\ra \Psi(v))$ be a vertex
adjacent to $w$.  We let $\tau\in\face(Y)$ denote  the edge spanned by $w,w'$.
For every edge $e$ of $Z$ with 
$v\subsetneq e\subset\si^{(1)}$, let $\hat e\subset Y^{(1)}$
denote the edge with $q(\hat e)=e$ that contains $w$.  By assumption, the collection
of edges $\{\tau\}\cup \{\hat e\}_{v\subsetneq e\subset \si^{(1)}}$  determines
a complete graph in the link of $w$, and therefore is contained in a 
cube $\hat \si$ of dimension $1+\dim\si$.  Then $q(\hat\si)=\si$ and 
$\tau \subset\hat\si$; this implies that $\tau\subset \im(\Psi(\si)\ra\Psi(v))$.

Since the $1$-skeleton of 
$\cap_{v\subsetneq e\subset \si^{(1)}}\im(\Psi(e)\ra \Psi(v))$ is connected,
we conclude that it coincides with the $1$-skeleton of $\im(\Psi(\si)\ra\Psi(v))$.
By convexity, we get (\ref{eqn_1_determined}).
\end{proof}

\begin{theorem}
\label{thm_fiber_functor_gives_restriction_quotient}
Let $Z$ be a $\cat(0)$ cube complex, and $\Psi:\face(Z)\ra \ccc$ be a $1$-determined
contravariant functor.  Then there is a restriction quotient
$q:Y\ra Z$ such that
the associated fiber functor $\Psi_q:\face(Z)\ra \ccc$ is equivalent by a natural
transformation to $\Psi$. 
\end{theorem}
\begin{proof}
We first construct the cube complex $Y$, and then verify that it has the
desired properties.

We begin with the disjoint union
$\bigsqcup_{\si\in\face(Z)}(\si\times \Psi(\si))$, and for every inclusion
$\si\subset\tau$, we glue the subset 
$\si\times \Psi(\tau)\subset \tau\times\Psi(\tau)$ 
to $\si\times\Psi(\si)$ by using the map
$$
\si\times\Psi(\tau)
\xrightarrow{\id_\si\times\Psi(\si\subset\tau)}\si\times\Psi(\si)\,.
$$
One checks that the cubical structure on 
$\bigsqcup_{\si\in\face(Z)}(\si\times \Psi(\si))$ descends to the quotient $Y$,
 the projection maps $\si\times\Psi(\si)\ra \si$ descend to a
cubical map $q:Y\ra Z$, and for every $\si\in\face(Z)$, the union of the
cubes $\hat\si\subset Y$ such that $f(\hat\si)=\si$ is a copy of 
$\si\times \Psi(\si)$.

We now verify that links in $Y$ are flag complexes.

Let $v$ be a $0$-cube in $Y$, and suppose $\si_1,\ldots,\si_k$ are $1$-cubes
containing $v$, such that for all $1\leq i\neq j\leq k$ the $1$-cubes $\si_i,\si_j$
span a $2$-cube $\si_{ij}$ in the link of $v$.  We may assume after reindexing
that for some $h\geq 0$  the image $q(\si_i)$ is
 a $1$-cube in $Z$ if $i\leq h$ and a $0$-cube if $i>h$.  
 
Since $\Psi(v)$ is a $\cat(0)$ cube complex, the edges $\{\si_i\}_{i>h}$ 
span a cube $\si_{vert}\subset q^{-1}(v)$.  

For $1\leq i\neq j\leq h$, the $2$-cube $\si_{ij}$ projects to a 
$2$-cube $q(\si_{ij})$ spanned by the two edges  $q(\si_i),q(\si_j)$.
Since $Z$ is a $\cat(0)$ cube complex, the edges $\{q(\si_i)\}_{i\leq h}$
span an $h$-cube   $\bar\si_{hor}\subset Z$.  By the $1$-determined
property, we get that $\im(\Psi(\bar\si_{hor})\ra \Psi(v))$ contains $v$,
and so there is an $h$-cube $\si_{hor}\subset Y$ containing $v$ such
that $q(\si_{hor})=\bar\si_{hor}$.  
 
Fix $1\leq i\leq h$.  Then for $j>h$, the $2$-cube $\si_{ij}$ projects to 
$q(\si_i)$, and hence $\si_j$ belongs to $\im(\Psi(\si_i)\ra\Psi(v))$.  
If $j,k>h$, then $\si_j,\si_k$ both belong to $\im(\Psi(\si_i)\ra\Psi(v))$,
and by the convexity of $\im(\Psi(\si_i)\ra\Psi(v))$ in $\Psi(v)$, we get
that $\si_{jk}$ also belongs to $\im(\Psi(\si_i)\ra\Psi(v))$.  Applying
convexity again,
we get that $\si_{vert}\subset \im(\Psi(\si_i)\ra\Psi(v))$.  By the $1$-determined
property, it follows that $\si_{vert}\subset \im(\Psi(\bar\si_{hor})\ra \Psi(v))$.
This yields a $k$-cube $\si\subset Y$ containing $\si_{hor}\cup\si_{vert}$,
which is spanned by $\si_1,\ldots,\si_k$.  

Thus we have shown that links in $Y$ are flag complexes.  The fact that the
fibers of $f:Y\ra Z$ are contractible implies that $Y$ is contractible
(in particular simply connected), so $Y$ is $\cat(0)$.

\end{proof}

We now observe that the construction of restriction quotients is compatible
with product structure:
\begin{lemma}[Behavior under products]
\label{lem_behavior_under_products}
For $i\in \{1,2\}$ let $q_i:Y_i\ra Z_i$ be a restriction quotient with fiber
functor $\Psi_i:\face(Z_i)\ra \ccc$.  Then the
product $q_1\times q_2:Y_1\times Y_2\ra Z_1\times Z_2$ is a
restriction quotient with fiber functor given by the product:
$$
\face(Z_1\times Z_2)\simeq \face(Z_1)\times \face(Z_2)
\xrightarrow{\Psi_1\times\Psi_2}\ccc\times\ccc
\xrightarrow{\times}\ccc\,.
$$

In particular, if one starts with  $\cat(0)$ cube complexes $Z_i$ and
fiber functors $\Psi_i:Z_i\ra \ccc$ for $i\in\{1,2\}$, then the product
fiber functor defined as above is the fiber functor of the product of the 
restriction quotients associated to the $\Psi_i$'s.
\end{lemma}

\subsection{Equivariance properties}
\label{sec_equivariance_properties}

\mbox{}
We now discuss isomorphisms between restriction quotients, and the naturality properties of the restriction quotient 
associated with a fiber functor.

Suppose we have a commutative diagram
\begin{diagram}
Y_1    &\rTo^{\hat\al}&Y_2\\
\dTo^{q_1} &  &\dTo^{q_2}\\
Z_1      & \rTo^\al  & Z_2
\end{diagram}
where the $q_i$'s are restriction quotients and $\al,\hat\al$ are cubical isomorphisms.   Let $\Psi_i:\face(Z_i)\ra \ccc$
be the fiber functor associated with $q_i$.   Notice that the pair $\al,\hat\al$ allows us to compare  the two
fiber functors, since for every $\si\in\face(Z_1)$,  the map $\hat\al$ induces a cubical isomorphism
between $\Psi_1(\si)$ and $\Psi_2(\al(\si))$, and this is compatible with maps induced with inclusions of
faces.   This may be stated more compactly by saying that $\hat\al$ induces a natural isomorphism
between the
 fiber functors $\Psi_1$ and $\Psi_2\circ \face(\al)$, where $\face(\al):\face(Z_1)\ra\face(Z_2)$
is the poset isomorphism induced by $\al$.  Here the term {\em natural isomorphism} is being used
in the sense of category theory, i.e. a natural transformation 
that has an inverse that is also a natural transformation.

Now suppose  that for $i\in \{1,2\}$ we have a $\cat(0)$ cube complex $Z_i$
and a $1$-determined fiber functor $\Psi_i:\face(Z_i)\ra \ccc$.   Let $f_i:Y_i\ra Z_i$ be the 
associated restriction quotients.  If we have a pair $\al,\be$, where $\al:Z_1\ra Z_2$ is a
cubical isomorphism, and $\be$ is a natural isomorphism between the
fiber functors $\Psi_1$ and $\Psi_2\circ \face(\al)$, then we get an induced map $\hat\al:Y_1\ra Y_2$,
which may be defined by using the description of $Y_i$ as the quotient of the disjoint collection
$\{\si\times \Psi_i(\si)\}_{\si\in \face(Z_i)}$.  

As a consequence of the above, having an action of a group $G$ on a restriction quotient $f:Y\ra Z$
is equivalent to having an action $G\acts Z$ together with a compatible ``action'' on the fiber functor
$\Psi_f$, i.e. a family $\{(\al(g),\be(g))\}_{g\in G}$ as above that
also satisfies an appropriate composition rule.

\subsection{Quasi-isometric properties} 
\label{sec_quasiisometric_properties}

We now consider the coarse geometry of restriction quotients; this amounts
to a ``coarsification'' of the discussion in the preceding subsection.

The relevant definition is a coarsification of the natural isomorphisms between
fiber functors.

\begin{definition}
Let $Z$ be a $\cat(0)$ cube complex and $\Psi_i:\face(Z)\ra\ccc$
be fiber functors for $i\in\{1,2\}$.
An {\em $(L,A)$-quasi-natural isomorphism from $\Psi_1$ to $\Psi_2$}
is a collection $\{\phi(\si):\Psi_1(\si)\ra\Psi_2(\si)\}_{\si\in\face(Z)}$
such that
$\phi(\si)$ is an $(L,A)$-quasi-isometry for all $\si\in\face(Z)$,
and for every inclusion $\si\subset\tau$, the diagram
\begin{diagram}
\Psi_1(\tau)    &\rTo^{\phi(\tau)}&\Psi_2(\tau)\\
\dTo &  &\dTo\\
\Psi_1(\si)     & \rTo^{\phi(\si)}  & \Psi_2(\si)
\end{diagram}
commutes up to error $L$.
\end{definition}

Now for $i\in \{1,2\}$
let $f_i:Y_i\ra Z$ be a finite dimensional restriction quotient, with
respective fiber functor $\Psi_i:\face(Z)\ra\ccc$.  For any 
$\si\in\face(Z)$, we identify $\Psi_i(\si)$ with $f_i(b_\si)$,
where $b_\si\in\si$ is the barycenter.  

\begin{lemma}
\label{lem_quasiisometry_gives_quasiisomorphism}
Suppose we have a commutative diagram
\begin{diagram}
Y_1   &            &\pile{\rTo^{\phi}\\\lTo_{\phi'}}& &Y_2\\
	  &\rdTo_{f_1} &                 &\ldTo_{f_2}&\\
    && Z  && 
\end{diagram}
where $\phi,\phi'$ are $(L,A)$-quasi-isometries that are $A$-quasi-inverses,
i.e. the compositions $\phi\circ\phi'$, $\phi'\circ\phi$ are at distance
$<A$ from the identity maps.
Then the collection 
$$
\{\Psi_1(\si)=f_1^{-1}(b_\si)\stackrel{\phi\restr_{f_1^{-1}}(b_\si)}{\lra}
f_2^{-1}(b_\si)=\Psi_2(\si)\}_{\si\in\face(Z)}
$$
is an $(L',A')$-quasi-natural isomorphism where $L'=L'(L,A,\dim Y_i)$,
$A'=A'(L,A,\dim Y_i)$.
\end{lemma}
\begin{proof}
By Theorem \ref{thm_restriction_quotient_characterization}, 
the fiber $f_i^{-1}(b_\si)$ is a convex subset of $Y_i$, and
hence is isometrically embedded.  Therefore $\phi$ and $\phi'$
induce  $(L,A)$-quasi-isometric
embeddings $f_1^{-1}(b_\si)\ra f_2^{-1}(b_\si)$, 
$f_2^{-1}(\si_b)\ra f_1^{-1}(b_\si)$.  If $\si\subset\tau$, then any
point $x\in f_i^{-1}(b_\tau)$ lies at distance $<C=C(\dim Y_i)$
from a point in $f_i^{-1}(b_\si)$, and this implies that 
the collection of maps $\{\Psi_1(\si)\ra \Psi_2(\si)\}_{\si\in\face(Z)}$
is an $(L',A')$-quasi-natural isomorphism as claimed.
\end{proof}

\begin{lemma}
\label{lem_quasiisomorphism_gives_quasiisometry}
If $\{\phi(\si):\Psi_1(\si)\ra \Psi_2(\si)\}_{\si\in\face(Z)}$
is an $(L,A)$-quasi-natural isomorphism from $\Psi_1$ to $\Psi_2$,
then it arises from a commutative diagram as in the previous lemma,
where $\phi$, $\phi'$ are $(L',A')$-quasi-isometries that are $A'$-quasi-inverses,
and $L',A'$ depend only on $L,A$, and $\dim Y_i$.
\end{lemma}
\begin{proof}
For every $\si\in\face(Z)$, we may choose a quasi-inverse 
$\phi'(\si):\Psi_2(\si)\ra \Psi_1(\si)$ with uniform constants;
this is also a quasi-natural isomorphism.
Identifying $f_i^{-1}(\Int(\si))$ with the product
$\Int(\si)\times \Psi_i(\si)$, we define $\phi\restr_{f_1^{-1}(\Int(\si))}$ 
by
$$
f_1^{-1}(\Int(\si))=\Int(\si)\times \Psi_1(\si)
\stackrel{\id_{\Int(\si)}\times \phi(\si)}{\lra}
\Int(\si)\times \Psi_2(\si)=f_2^{-1}(\si)\,,
$$
and $\phi'$ similarly using $\{\phi'(\si)\}_{\si\in\face(Z)}$.
One readily checks that $\phi$, $\phi'$ are quasi-isometric embeddings
that are also quasi-inverses, where the constants depend on $L, A$, and
$\dim Y_i$.
\end{proof}
\section{The $\Z$-blow-up of right-angled Building}
\label{sec_blow-up of building}
In this section $\Ga$ will be an arbitrary finite simplicial graph, and
all buildings will be right-angled buildings modelled on the right-angled Coxeter group $W(\Ga)$ with defining graph $\Ga$.   The reader may wish to review Section 
\ref{building} for terminology and notation regarding buildings, before proceeding.

The goal of this section is examine restriction quotients
$q:Y\ra |\B|$, where the fibers are Euclidean spaces satisfying a
dimension condition as in Theorem \ref{thm_main_intro_quasi_action}
or \ref{thm_main_intro}.  For such restriction quotients, the fiber functor
may be distilled down to $1$-data, see Definition \ref{1-data}; this is discussed in 
Subsection \ref{subsection_Restriction quotients with Euclidean fibers}.
Conversely, given a building $\B$ and certain blow-up data (Definition
\ref{input data}), one can construct a corresponding $1$-determined
fiber functor as in Section \ref{sec_restriction_maps_versus_fiber_functors};
see Subsection \ref{subsection_input for Z-blow-up}.

\subsection{The canonical restriction quotient for a RAAG}
\label{subsection_canonical restriction quotient}
Let $G(\Gamma)$ be the RAAG with defining graph $\Ga$ and let $\B(\Ga)$ be the building associated with $G(\Ga)$  (see \cite[Section 5]{davis_buildings_are_cat0}). Then $G(\Gamma)$ can identified with the set of chambers of $\B(\Ga)$. Under this identification, the $J$-residues of $\B$, for $J$ a collection of vertices in $\Ga$, are the left cosets of the standard subgroups of $G(\Gamma)$ generated by $J$. Thus the poset of spherical residues is exactly the poset of left cosets of standard Abelian subgroups of $G(\Gamma)$, which is also isomorphic to the poset of standard flats in $X(\Ga)$.

We now revisit 
the discussion after Definition \ref{def_graph_product_of_spaces} and Example \ref{the key example} in more detail, and relate them to buildings. To simplify notation, we will write $G=G(\Ga)$, $\B=\B(\Ga)$ and $X=X(\Ga)$.

Let $|\B|$ be the Davis realization of the building $\B$.  Then we have an induced isometric action $G \acts |\B|$, which is cocompact, but not proper. It turns out there is natural way to blow-up $|\B|$ to obtain a space $X_e=X_{e}(\Ga)$ such that there is a geometric action $G \acts X_{e}$ and a $G$-equivariant restriction quotient map $X_{e}\to |\B|$.

$X_e$ can be constructed as follows. First we constructed the \textit{exploded Salvetti complex} $S_e=S_{e}(\Ga)$, which was introduced in \cite{bks2}, see also the discussion after Definition \ref{def_graph_product_of_spaces}. Suppose $L$ is the \textquotedblleft lollipop\textquotedblright, which is the union of a unit circle $S$ and a unit interval $I$ along one point. For each vertex $v$ in the vertex set $V(\Ga)$ of $\Ga$, we associate a copy of $L_{v}=S_v\cup I_v$, and let $\star_v\in L_v$ be the free end of $I_v$. Let $T=\prod_{v\in V(\Ga)} L_v$. Each clique $\Delta\subset \Ga$ gives rise to a subcomplex $T_{\Delta}=\prod_{v\in\Delta} L_{v}\times\prod_{v\notin\Delta}\{\star_v\}$. Then $S_e$ is the subcomplex of $T$ which is the union of all such $T_{\Delta}$'s, here $\Delta$ is allowed to be empty. It is easy to check $S_e$ is a non-positively curved cube complex. A \textit{standard torus} in $S_e$ is a subcomplex of form $\prod_{v\in\Delta} S_{v}\times\prod_{v\notin\Delta}\{\star_v\}$, where $\Delta\subset\Ga$ is a clique. Note that there is a unique standard torus of dimension $0$, which corresponds to the empty clique. There is a natural map $S_e=S_e(\Ga)\to S(\Ga)$ by collapsing the $I_v$-edge in each $L_v$-factor. This maps induces a 1-1 correspondence between standard tori in $S_e$ and standard tori in $S(\Ga)$. Notice that there is also a 1-1 correspondence between vertices in $S_e$ and standard tori in $S_e$.

Let $X_e$ be the universal cover of $S_e$. Then $X_e$ is a $CAT(0)$ cube complex and the action $G\acts X_e$ is geometric. The inverse images of standard tori in $S_e$ are called \textit{standard flats}. Note that each vertex in $X_e$ is contained in a unique standard flat. We define a map between the 0-skeletons $p: X^{(0)}_e(\Ga)\to |\B|^{(0)}$ as follow. Pick a $G$-equivariant identification between $0$-dimensional standard flats in $X$ and elements in $G$, and pick a $G$-equivariant map $\phi:X_e\to X$ induced by $S_e=S_e(\Ga)\to S(\Ga)$ described as above. Note that $c$ induces a 1-1 correspondence between standard flats in $X_e$ and standard flats in $X$. This gives rise to a 1-1 correspondence between standard flats in $X_e$ and left cosets of standard Abelian subgroups of $G$. For each $x\in X^{(0)}_e(\Ga)$, we define $p(x)$ to be the vertex in $|\B|^{(0)}$ that represents the left coset of the standard Abelian subgroup of $G$ which corresponds to the unique standard flat that contains $x$.

A \textit{vertical edge} of $X_e$ is an edge which covers some $S_v$-circle in $S_e$. A \textit{horizontal edge} of $X_e$ is an edge which covers some $I_v$-interval in $S_e$. Two endpoints of every vertical edge are in the same standard flat, thus they are mapped by $p$ to the same point in $|\B|^{(0)}$. More generally, for any given \textit{vertical cube}, i.e. every edge in this cube is a vertical edge, its vertex set is mapped by $p$ to one point in $|\B|^{(0)}$. Pick a horizontal edge and let $F_1,F_2\subset X_e$ be standard flats which contain the two endpoints of this edge respectively. Then $\phi(F_{1})$ and $\phi(F_2)$ are two standard flats in $X$ such that one is contained as a codimension 1 flat inside another. More generally, if $\sigma$ is a horizontal cube, i.e. each edge of $\sigma$ is a horizontal edge, then by looking the image of $\sigma$ under the covering map $X_e\to S_e$, we know the vertex set of $\sigma$ corresponds to an interval in the poset of standard flats of $X$. Every cube in $X_e$ splits as a product of a vertical cube and a horizontal cube (again this is clear by looking at cells in $S_e$). Thus we can extend $p$ to a cubical map $p:X_e\to |\B|$.

By construction, for a vertex $v\in |\B|$ of rank $k$, $p^{-1}(v)$ is isometric to $\E^n$. It follows from Theorem~\ref{thm_restriction_quotient_characterization} that $p$ arises from a restriction quotient, and this is called the \textit{canonical restriction quotient} for the RAAG $G$. This restriction quotient is exactly the one described in Example~\ref{the key example}, since hyperplanes in $\K$ of Example~\ref{the key example} are those which are dual to horizontal edges. We record following immediate consequence of the this construction.
\begin{lem}
\label{inverse image are flats}
Let $\sigma\subset|\B|$ be a cube and let $v\in \sigma$ be the vertex of minimal rank in $\sigma$. Then for any interior point $x\in \sigma$, $p^{-1}(x)$ is isometric to $\E^{rank(\sigma)}$.
\end{lem}

\begin{remark}
\label{rem_charney_folding_map}
In the literature, there is a related cubical map $X\to |\B|$ defined as follows. First we recall an alternative description of $X$. Actually similar spaces can be defined for all Artin groups (not necessarily right-angled) and was introduced by Salvetti. We will follow the description in \cite{charneyproblems}. Let $G\to W(\Ga)$ be the natural projection map. This map has a set theoretic section defined by representing an element $w\in W$ by a minimal length positive word with respect to the standard generating set and setting $\sigma(w)$ to be the image of this word $G$. It follows from fundamental facts about Coxeter groups that $\sigma$ is well-defined. Let $I$ be the vertex set of $\Ga$, and for any $J\subset I$, let $W(J)$ be the subgroup of $W(\Ga)$ generated by $J$. Let $K$ be the geometric realization of the following poset:
\begin{center}
$\{g\sigma(W(J))\mid g\in G, J\subset I, W(J)$ is finite$\}$.
\end{center}
It turns out that $K$ is isomorphic to the first barycentric subdivision of $X$. Let $G(J)\le G$ be the subgroup generated by $J$. We associate each $g\sigma(W(J))$ with the left coset $g G(J)$, and this induces a cubical map from the first cubical subdivision of $X$ to $|\B|$. However, this map is not a restriction quotient, since it has a lot of foldings (think of the special case when $G\cong\Z$).
\end{remark}

\subsection{Restriction quotients with Euclidean fibers}
\label{subsection_Restriction quotients with Euclidean fibers}
We reminder the reader that in
this section,  $W=W(\Ga)$ will be the right-angled Coxeter group with defining graph $\Ga$ and standard generating set $I$. Let $\B$ be an arbitrary right-angled building modelled on $W$. Let $\mathsf{S}$ be the poset of spherical subsets of $I$ and let $|\B|$ be the Davis realization of $\B$.

Our next goal is to generalize the canonical restriction quotient mentioned in the previous subsection. However, to motivate our construction, we will first consider a restriction quotient  $q:Y'\to|\B|$ which satisfies the conclusion of Lemma~\ref{inverse image are flats}, and identify  several key features of $q$.  

Let $\Phi$ be the fiber functor associated with $q$ (see Section~\ref{sec_restriction_maps_versus_fiber_functors}). For any vertices $v,w\in |\B|$, we will write $v\le w$ if and only if the residue associated with $v$ is contained in the residue associated with $w$. 

Let $\res$ be the poset of spherical residues in $\B$. Then $\Phi$ induces a functor $\Phi'$ from $\res$ to $\ccc$ (Section~\ref{sec_restriction_maps_versus_fiber_functors}) as follows. Each element in $\res$ is associated with the fiber of the corresponding vertex in $|\B|$. If $s,t\in \res$ are two elements such that $rank(t)=rank(s)+1$ and $s<t$, then the associated vertices in $v_s,v_t\in|\B|$ are joined by an edge $e_{st}$. In this case $\Phi(e_{st})\to\Phi(v_s)$ is an isomorphism, so we define the morphism $\Phi'(s)\to\Phi'(t)$ to be the map induced by $\Phi(e_{st})\to\Phi(v_t)$. If $s,t\in\res$ are arbitrary two elements with $s\le t$, then we find an ascending chain from $s$ to $t$ such that the difference between the ranks of adjacent elements in the chain is $1$, and define $\Phi'(s)\to\Phi'(t)$ be the composition of those maps induced by the chain. It follows from the functor property of $\Phi$ that $\Phi'(s)\to\Phi'(t)$ does not depend on the choice of the chain, and $\Phi'$ is a functor. Recall that there is a 1-1 correspondence between elements in $\res$ and vertices of $|\B|$, so we will also view $\Phi'$ as a functor from the vertex set of $|\B|$ to $\ccc$. Let $\sigma_1\subset\sigma_2$ be faces in $|\B|$ and let $v_i$ be the vertex of minimal rank in $\sigma_i$ for $i=1,2$. Then by our construction, then morphism $\Phi(\sigma_2)\to\Phi(\sigma_1)$ is the same as $\Phi'(v_2)\to\Phi'(v_1)$.

\begin{definition}[1-data]
\label{1-data}
Pick a vertex $v\in|\B|$ of rank 1, and let $\calr_v$ be the associated residue. Let $\{v_\lambda\}_{\lambda\in\Lambda}$ be the collection of vertices in $|\B|$ which is $<v$ and let $e_{\lambda}$ be the edge joining $v$ and $v_{\lambda}$. Then there is a 1-1 correspondence between elements in $\calr_v$ and $v_{\lambda}$'s. Each $v_{\lambda}$ determines a point in $\Phi(v)$ by consider the image of $\Phi(e_{\lambda})\to\Phi(v)$. This induced a map $f_{\calr_v}:\calr_v\to \Phi(v)$. The collection of all such $f_{\calr_v}$'s with $v$ ranging over all rank 1 vertices of $|\B|$ is called the \textit{1-data} associated with the restriction quotient $q:Y'\to|\B|$.
\end{definition}

\begin{lem}
\label{compactible with parallelism}
Pick two vertices $v,u\in|\B|$ of rank 1, and let $\calr_{v},\calr_{u}$ be the corresponding residues. Suppose these two residues are parallel with the parallelism map given by $p:\calr_{v}\to\calr_{u}$. Then: 
\begin{enumerate}
\item $\Phi(v)$ and $\Phi(u)$, considered as convex subcomplexes of $Y'$, are parallel.
\item If $p':\Phi(v)\to\Phi(u)$ is the parallelism map, then the following diagram commutes:
\begin{center}
$\begin{CD}
\calr_{v}        @>p>>        \calr_{u}  \\
@Vf_{\calr_{v}}VV                                   @Vf_{\calr_{u}}VV\\
\Phi(v)                         @>p'>>       \Phi(u) 
\end{CD}$
\end{center}
\end{enumerate}
\end{lem}

\begin{proof}
It follows from Lemma~\ref{parallel sets of residues} that there is a finite chain of residues, starting at $\calr_v$ and ending at $\calr_{u}$, such that adjacent elements in the chain are parallel residues in a spherical residue of rank 2. Thus we can assume without loss of generality that $\calr_v$ and $\calr_{u}$ are contained in the a spherical residue $\S$ of type $J=\{j,j'\}$, and we assume both $\calr_v$ and $\calr_{u}$ are $j$-residues.

Pick $x\in\calr_v$. By Theorem \ref{product decomposition}, there is a $j'$-residue $\W$ which contains both $x$ and $p(x)$. Let $s,w\in|\B|$ be the vertex corresponding to $\S$ and $\W$. Note that there is a 2-cube in $|\B|$ such that $v,w,s$ are its vertices. Since $\Phi$ is 1-determined, $\im(\Phi'(v)\to\Phi'(s))$ and $\im(\Phi'(w)\to\Phi'(s))$ are orthogonal lines in the 2-flat $\Phi'(s)$. Moreover, the intersection these two lines is the image of $f_{\calr_v}(x)$ under the morphism $\Phi'(v)\to\Phi'(s)$. Similarly, the images of $\Phi'(u)\to\Phi'(s)$ and $\Phi'(w)\to\Phi'(s)$ are orthogonal lines $\Phi'(s)$, and their intersection is the image of $f_{\calr_u}(p(x))$ under $\Phi'(v)\to\Phi'(s)$. It follows that $\im(\Phi'(v)\to\Phi'(s))$ and $\im(\Phi'(u)\to\Phi'(s))$ are parallel, hence $\Phi(v)$ and $\Phi(u)$, considered as convex subcomplexes of $Y'$, are parallel. Moreover, since image of $f_{\calr_v}(x)$ under $\Phi'(v)\to\Phi'(s)$ and the image of $f_{\calr_u}(p(x))$ under $\Phi'(v)\to\Phi'(s)$ are in the line $\im(\Phi'(w)\to\Phi'(s))$, the diagram in (2) commutes.
\end{proof}

Pick a vertex $u\in|\B|$ of rank $=k$ and let $\calr_u$ be the corresponding $J$-residue with $J=\cup_{i=1}^{k}j_i$. Then there is a map $f_{\calr_u}:\calr_u\to \Phi'(\calr_u)=\Phi(u)$ defined by considering $\Phi'(x)\to\Phi'(\calr_u)$ for each element $x\in \calr_u$. This map coincides with the $f_{\calr_u}$ defined before when $u$ is rank 1. For $1\le i\le k$, let $\calr_i$ be a $j_i$-residue in $\calr_u$. Since $\Phi$ is 1-determined, $\{\im(\Phi'(\calr_i)\to\Phi'(\calr_u))\}_{i=1}^{k}$ are mutually orthogonal lines in $\Phi'(\calr_u)$. This induces an isomorphism 
\begin{equation}
\label{functor product decomposition}
i:\prod_{i=1}^{k}\Phi'(\calr_i)\to \Phi'(\calr_u).
\end{equation}
The following is a consequence of (2) of Lemma \ref{compactible with parallelism}.

\begin{cor}
\label{compactible with product}
The map $f_{\calr_u}$ satisfies $f_{\calr_u}=i\circ(\prod_{i=1}^{k}f_{\calr_i})\circ g$, where $g:\calr_u\to\prod_{i=1}^{k}\calr_i$ is the map in Theorem \ref{product decomposition}. In this case, we will write $f_{\calr_u}=\prod_{i=1}^{k}f_{\calr_i}$ for simplicity.
\end{cor}

Pick $J'\subset J$ and let $\calr_u'$ be a $J'$-residue in $\calr_u$. By Theorem~\ref{product decomposition}, $\calr_u'=\prod_{i\in J'}\calr_i\times\prod_{i\notin J'}\{a_{i}\}$ for $a_{i}\in \calr_i$. Then the following is a consequence of Corollary \ref{compactible with product} and the functorality of $\Phi'$.
\begin{cor}
\label{morphism in terms of product}
Let $h$ be the morphism between $\Phi'(\calr_{u'})$ and $\Phi'(\calr_u)=\Phi'(\calr_{u'})\times\prod_{i\notin J'}\Phi'(\calr_i)$. Then for $x\in \Phi'(\calr_{u'})$, we have $h(x)=\{x\}\times \prod_{i\notin J'}\{f_{\calr_i}(a_{i})\}$.
\end{cor}

\subsection{Construction of the $\Z$-blow-up}
\label{subsection_input for Z-blow-up}

In the previous section, we started from a restriction quotient $q:Y'\to|\B|$, and produced associated 1-data (Definition~\ref{1-data}), which is compatible with parallelism in the sense of Lemma~\ref{compactible with parallelism}. In this section, we will consider the inverse, namely we want construct a restriction quotient from this data. 

Let $\Lambda_{\B}$ be the collection of parallel sets of $i$-residues in $\B$ ($i$ could be any element in $I$). There is another type map $T$ which maps a spherical $J$-residue $\calr$ to $\{\lambda\in\Lambda_{\B}\mid$ $\lambda$ contains a representative in $\calr\}$. In other words, let $\calr\cong\prod_{i\in I}\calr_i$ be the product decomposition as in Theorem \ref{product decomposition}, where each $\calr_i$ is an $i$-residue in $\calr$ ($i\in J$). Then $T(\calr)$ is the collection of parallel sets represented by those $\calr_i$'s. Let $\Z^{T(\calr)}$ be the collection of maps from $T(\calr)$ to $\Z$, and let $\Z^{\emptyset}$ be a single point.

Our goal in this section is to construct a restriction quotient from the following data.
\begin{definition}[Blow-up data]
\label{input data}
For each $i$-residue $\calr\subset\B$, we associate a map $h_{\calr}:\calr\to\Z^{T(\calr)}$ such that if two $i$-residues $\calr_1$ and $\calr_2$ are parallel, let $h_{12}:\calr_1\to\calr_2$ be the parallelism map, then $h_{\calr_1}=h_{\calr_2}\circ h_{12}$. 
\end{definition}

If $\calr$ is a spherical residue with  product decomposition given by $\calr\cong\prod_{i\in I}\calr_i$, then the maps $h_{\calr_{i}}:\calr_{i}\to\Z$ induces a map $h_{\calr}:\calr\to \Z^{T(\calr)}$. It follows from the definition of $h_{\calr}$, and the discussion after Definition \ref{definition_parallel} that if $\calr,\calr'\in C$ are parallel and let $h:\calr\to\calr'$ be the parallelism map, then $h_{\calr}=h_{\calr'}\circ h$.

The following result is a consequence of Theorem \ref{product decomposition}:

\begin{lem}
Let $\T\in C$ be an $H$-residue. Let $g:\T\cong\prod_{i=1}^{n}\T_i$ be the product decomposition induced by $H=\sqcup_{i=1}^{n}H_i$ (see Theorem \ref{product decomposition}). Then $h_{\T}=(\prod_{i=1}^{n}h_{\T_i})\circ g$.
\end{lem}

To simplify notation, we will write $h_{\T}=\prod_{i=1}^{n}h_{\T_i}$ instead of $h_{\T}=(\prod_{i=1}^{n}h_{\T_i})\circ g$.

Let $J$ and $\calr=\prod_{i\in J}\calr_i$ be as before. A $J'$-residue $\calr'\subset\calr$ can be expressed as $(\prod_{i\in J'}\calr_i)\times(\prod_{i\in J\setminus J'}\{c_i\})$, here $c_i$ is a chamber in $\calr_i$. We define an inclusion $h_{\calr'\calr}:\Z^{T(\calr')}\ra\Z^{T(\calr)}$ by $h_{\calr'\calr}(a)=\{a\}\times \prod_{i\in J\setminus J'}\{h_{\calr_i}(c_i)\}$. Since $h_{\calr}=h_{\calr'}\times(\prod_{i\in J\setminus J'}h_{\calr_i})$, $h_{\calr'\calr}$ fits into the following commutative diagram:
\begin{center}
$\begin{CD}
\calr'                         @>>>       \calr\\
@Vh_{\calr'}VV                                   @Vh_{\calr}VV\\
\Z^{T(\calr')}        @>h_{\calr'\calr}>>        \Z^{T(\calr)}
\end{CD}$
\end{center}
Suppose $\calr''$ is a $J''$-residue such that $\calr''\subset\calr'\subset\calr$. Since $h_{\calr}=h_{\calr'}\times(\prod_{i\in J\setminus J'}h_{\calr_i})=h_{\calr''}\times (\prod_{i\in J'\setminus J''}h_{\calr_i})\times (\prod_{i\in J\setminus J'}h_{\calr_i})$, we have
\begin{equation}
\label{composition}
h_{\calr''\calr}=h_{\calr'\calr}\circ h_{\calr''\calr'}.
\end{equation}

Now we define a contravariant functor $\Psi:\face(|\B|)\ra \ccc$ as follows. Let $f$ be a face of $|\B|$ and let $v_f\in f$ be unique vertex which has minimal rank among the vertices of $f$. Let $\calr_f\subset\B$ be the residue associated with $v_f$. We define $\Psi(f)=\R^{T(\calr_f)}$ ($\R^{\emptyset}$ is a single point), here $\R^{T(\calr_f)}$ is endowed with the standard cubical structure and we identify $\Z^{T(\calr_f)}$ with the $0$-skeleton of $\R^{T(\calr_f)}$. 

An inclusion of faces $f\ra f'$ induces an inclusion $\calr_{f'}\ra\calr_{f}$. We define the morphism $\Psi(f')\ra\Psi(f)$ to be the embedding induce by $h_{\calr_{f'}\calr_f}:\Z^{T(\calr_{f'})}\ra\Z^{T(\calr_{f})}$. 

\begin{lem}
$\Psi$ is contravariant functor.
\end{lem}

\begin{proof}
It is easy to check that passing from an inclusion of faces $f\ra f'$ to $\calr_{f'}\ra\calr_{f}$ is a functor. And it follows from (\ref{composition}) that passing from $\calr_{f'}\ra\calr_{f}$ to $h_{\calr_{f'}\calr_f}:\Z^{T(\calr_{f'})}\ra\Z^{T(\calr_{f})}$ is a functor.
\end{proof}

\begin{lem}
$\Psi$ is 1-determined.
\end{lem}

\begin{proof}
Let $\sigma\subset|\B|$ be a face and pick a vertex $v\in\sigma$. Let $\{v_{i}\}_{i=1}^{k}$ be the vertices in $\sigma$ that are adjacent to $v$ along an edge $e_i$. Let $\sigma_{\le v}$ be the sub-cube of $\sigma$ that is spanned by $e_{i}$'s such that $v_{i}\ge v$. We define $\sigma_{>v}$ similarly ($\sigma_{>v}$ could be empty). Then $\sigma=\sigma_{\le v}\times\sigma_{>v}$. Moreover, $v$ is the maximal vertex in $\sigma_{\le v}$ and the minimum vertex in $\sigma_{>v}$. Note that $\Psi(e_{i})\to\Psi(v)$ is an isometry if $v_{i}>v$. Thus it suffices to consider the case where $v$ is the maximal vertex of $\sigma$.

Let $v_{m}$ be the minimal vertex of $\sigma$. Note that $\im(\Psi(\si)\to \Psi(v))\subset\cap^{k}_{i=1}\im(\Psi(e)\to \Psi(v))$ is a cubical convex embedding of Euclidean subspaces, it suffices to show they have the same dimension. Let $\calr(v)\subset C$ be the residue corresponding to the vertex $v$. Note that $T(\calr(v_m))=\cap_{i=1}^{k}T(\calr(v_i))$ ($T$ is the type map defined on the beginning of Section \ref{subsection_input for Z-blow-up}). Thus the dimension of $\cap^{k}_{i=1}\im(\Psi(e)\to \Psi(v))$ equals to the cardinality of $T(\calr(v_m))$, which is the dimension of $\im(\Psi(\si)\to \Psi(v))$.
\end{proof}

$\Psi$ is called the \textit{fiber functor associated with the blow-up data $\{h_{\calr}\}$}, and the restriction quotient $q:Y\to|\B|$ which arises from the fiber functor $\Psi$ (see Theorem~\ref{thm_fiber_functor_gives_restriction_quotient}) is called the \textit{restriction quotient associated with the blow-up data $\{h_{\calr}\}$}. It is clear from the construction that the $1$-data of $q$ (Definition \ref{1-data}) is the blow-up data $\{h_{\calr}\}$ (we naturally identify $\Z^{T(\calr)}$'s in the blow-up data with the 0-skeleton of the $q$-fibers of rank 1 vertices in $|\B|$). We summarize the above discussion in the following theorem.

\begin{thm}
Given the blow-up data $\{h_{\calr}\}$ as in Definition \ref{input data}, there exists a restriction quotient $q:Y\to|\B|$ whose 1-data is the blow-up data we start with.
\end{thm}

\begin{remark}
Here we blow up the building $\B$ with respect to a collection of $\Z$'s since we want to apply the construction for RAAG's. However, in other cases, it may be natural to blow up with respect to other objects. Here is a variation. For each parallel class of rank 1 residues $\lambda\in\Lambda_{\B}$, we associate a $CAT(0)$ cube complex $Z_{\lambda}$. For each rank 1 residue $\calr$ in the class $\lambda$, we define a map $h_{\calr}$ which assigns each element of $\calr$ a convex subcomplex of $Z_{\lambda}$. We require these $\{h_{\calr}\}$ to be compatible with parallelism between rank 1 residues. Given this set of blow-up data, we can repeat the previous construction to obtain a restriction quotient over $|\B|$. It is also possible to blow-up buildings with respect to spaces more general than $CAT(0)$ cube complexes. We give a slightly different approach in Section \ref{sec_blow-up building more general version}.
\end{remark}

Now we show that the construction in this section is indeed a converse to Section~\ref{subsection_Restriction quotients with Euclidean fibers} in the following sense. Let $q:Y'\to|\B|$ be a restriction quotient as in Section \ref{subsection_Restriction quotients with Euclidean fibers} and let $\Phi$ and $\Phi'$ be the functors introduced there. For each vertex $v\in|\B|$ of rank 1 and its associated residue $\calr_v$, we pick an isometric embedding $\eta_v:\Z^{T(\calr_v)}\to \Phi(v)$ such that its image is vertex set of $\Phi(v)$. We also require these $\eta_v$'s respect parallelism. More precisely, let $u\in|\B|$ be a vertex of rank 1 such that $\Phi(v)$ and $\Phi(u)$ (understood as subcomplexes of $Y'$) are parallel with the parallelism map given by $p:\Phi(v)\to\Phi(u)$. Then $p\circ\eta_v=\eta_u$ (note that $T(\calr_v)=T(\calr_u)$ by Lemma \ref{compactible with parallelism}).

Let $\Psi$ be the functor constructed in this section from the blow-up data $\{h_{\calr_v}=\eta^{-1}_v\circ f_{\calr_v}:\calr_v\to \Z^{T(\calr_v)}\}_{v\in |\B|}$, here $v$ ranges over all vertices of rank 1 in $|\B|$, $\calr_v$ is the residue associated with $v$ and $f_{\calr_v}$ is the map in Definition~\ref{1-data}. Pick a face $\sigma\in|\B|$ and let $u\in\sigma$ be the vertex of minimal rank. Let $\calr_u$ be the associated $J$-residue with its product decomposition given by $\calr_u=\prod_{j\in J}\calr_{v_j}$ ($v_j$'s are rank 1 vertices $\le u$). Let $\xi_{\sigma}:\Psi(\sigma)\to\Phi(\sigma)$ be the isometry induced by 
\begin{center}
$\prod_{j\in J}\eta_{v_j}:\Z^{T(\calr_u)}\to \prod_{j\in J}\Phi(v_{j})$
\end{center}
and the product decomposition $\prod_{j\in J}\Phi(v_{j})\cong \Phi(u)\cong\Phi(\sigma)$ which comes from (\ref{functor product decomposition}). The following is a consequence of Corollary \ref{compactible with product}, Corollary \ref{morphism in terms of product} and the discussion in this section.

\begin{cor}
\label{consistence}
The maps $\{\xi_{\sigma}\}_{\si\in \face(|\B|)}$ induce a natural isomorphism between $\Phi$ and $\Psi$. Thus for any restriction quotient $q:Y'\to|\B|$ which satisfies the conclusion of Lemma~\ref{inverse image are flats}, 
if $q'$ is the restriction quotient  whose blow-up data is the 1-data of $q$, then
$q'$ is equivalent to $q$ up to a natural isomorphism between their fiber functors.
\end{cor}

\begin{cor}
\label{product of restriction quotient with Euclidean fibers}
Let $q:Y\to|\B|$ be a restriction quotient which satisfies the conclusion of Lemma \ref{inverse image are flats}. Let $\B\cong\B_1\times\B_2$ be a product decomposition of the building $\B$ induced by the join decomposition $\Ga=\Ga_1\circ\Ga_2$ of the defining graph of the associated right-angled Coxeter group. Then there are two restriction quotients $q_{1}:Y_{1}\to|\B_1|$ and $q_{2}:Y_{2}\to|\B_2|$ such that $Y=Y_1\times Y_2$ and $q=q_1\times q_2$. Moreover, $q_1$ and $q_2$ also satisfy the conclusion of Lemma \ref{inverse image are flats}.
\end{cor}

\begin{proof}
By Corollary~\ref{consistence}, we can assume $q$ is the restriction quotient associated with a set of blow-up data $\{h_{\calr}\}$. For every $\B_1$-slice in $\B$, we can restrict $\{h_{\calr}\}$ to $\B_1$ to obtain a blow-up data for $\B_1$. This does not depend on our choice of the $\B_1$-slice, since the blow-up data respects parallelism. We obtain a blow-up data for $\B_2$ in a similar way. It follows from the above construction that the fiber functor associated with $\{h_{\calr}\}$ is the product of the fiber functors associated the blow-up data on $\B_1$ and $\B_2$. Thus this corollary is a consequence of Lemma~\ref{lem_behavior_under_products}.
\end{proof}

\subsection{More properties of the blow-up buildings}
\label{subsection_more properties}
In this section, we look at the restriction quotient $q:Y\to|\B|$ associated with the blow-up data $\{h_{\calr}\}$ as in Definition~\ref{input data} (or equivalently, a restriction quotient $q:Y\to|\B|$ which satisfies the conclusion of Lemma \ref{inverse image are flats}) in more detail, and record several basic properties of $Y$. A hurried reader can go through Definition~\ref{rank and inverse map}, then proceed directly to Section~\ref{subsection_morhpism of blow-up data} and come back to this part later.

\begin{definition}
\label{rank and inverse map}
A vertex $y\in Y$ is of \textit{rank $k$} if $p(y)$ is a vertex of rank $k$. Thus $q$ induces a bijection between rank $0$ vertices in $Y$ and rank $0$ vertices in $|\B|$. Since rank $0$ vertices in $|\B|$ can be identified with chambers in $\B$, $q^{-1}$ induces a well-defined map $q^{-1}:\B\to Y$ from the set of chambers of $\B$ (or rank $0$ vertices of $|\B|$) to rank $0$ vertices in $Y$.
\end{definition}

\begin{lem}
\label{restriction}
For any residue $\calr\subset\B$, we view $\calr$ as a building and restrict the blow-up data over $\B$ to a blow-up data over $\calr$. Let $q_{\calr}:Y_{\calr}\to |\calr|$ be the associated restriction quotient. Then there exists an isometric embedding $i:Y_{\calr}\to Y$ which fits into the following commutative diagram:
\begin{center}
$\begin{CD}
Y_\calr                         @>i>>       Y\\
@Vq_{\calr}VV                                   @VqVV\\
|\calr|       @>i'>>        |\B|
\end{CD}$
\end{center}
Moreover, $i(Y_{\calr})=q^{-1}(i'(|\calr|))$.
\end{lem}

The lemma is a direct consequence of the construction in Section \ref{subsection_input for Z-blow-up}.
 
Pick a vertex $v\in |\B|$. The \textit{downward complex of $v$} is the smallest convex subcomplex of $|\B|$ which contains all vertices which are $\le v$. If $\calr_v$ is the residue associated with $v$, then the downward complex is the image of the embedding $|\calr_v|\hookrightarrow |\B|$. The next result follows from Lemma \ref{restriction} and Corollary~\ref{product of restriction quotient with Euclidean fibers}.

\begin{lem}
\label{downward complex}
Let $D_v$ be the downward complex of a vertex $v\in\B$ and let $\calr_v=\prod_{i=1}^{k}\calr_i$ be the product decomposition of residue associated with $v$. Then $q^{-1}(D_v)$ is isomorphic to the product of the mapping cylinders of $\calr_i\xrightarrow{h_{\calr_i}}\Z^{T(\calr_i)}\to \R^{T(\calr_i)}$ $(1\le i\le k)$.
\end{lem}

\begin{lem}\
\label{locally-finite and cocompact}
\begin{enumerate}
\item If $h^{-1}_{\calr}(x)$ is finite for any rank 1 residue $\calr$ and $x\in\Z^{T(\calr)}$, then $Y$ is locally finite. If there is a uniform upper bound for the cardinality of $h^{-1}_{\calr}(x)$, then $Y$ is uniformly locally finite.
\item If there exists $D>0$ such that the image of each $h_{\calr}$ is $D$-dense in $\Z^{T(\calr)}$, then there exists $D'$ which depends on $D$ and the dimension of $|\B|$ such that the collection of inverse images of rank $0$ vertices in $|\B|$ is $D'$-dense in $Y$.
\end{enumerate}
\end{lem}

\begin{proof}
We prove (1) first. Pick a vertex $y\in Y$. Let $v=q(y)$. It suffices to show the set of edges in $|\B|$ which contain $v$, and can be lifted to an edge in $Y$ that contains $y$, is finite. Since there are only finitely many vertices in $|\B|$ which are $\ge v$, it suffices to consider the edges of the form $\overline{v_{\lambda}v}$ with $v_\lambda<v$. It follows from our assumption and Lemma \ref{downward complex} that there are only finitely many such edges which have the required lift. The proof of uniform local finiteness is similar. 

To see (2), notice that $\cup_{v\in|\B|}\Psi(v)$ is 1-dense in $Y$, here $v$ ranges over all vertices of $|\B|$. It follows from Lemma \ref{downward complex} that every point in $\Psi(v)$ be can approximated by the inverse image of some rank $0$ vertex up to distance $D'$.
\end{proof}

Next we discuss the relation between $Y$ and the exploded Salvetti complex $S_e=S_e(\Ga)$ introduced in Section~\ref{subsection_canonical restriction quotient}. Let $\Psi$ be the fiber functor associated with $q:Y\to|\B|$.

First we label each vertex $v\in Y$ by a clique in $\Ga$ as follows. Recall that $q(v)$ is associated with a $J$-residue $\calr\subset\B$, where $J$ is the vertex set of a clique in $\Ga$. Thus we label $v$ by this clique. We also label each vertex of $S_e$ by a clique. Any vertex $v\in S_e$ is contained in a unique standard torus. Recall that a standard torus arises from a clique in $\Ga$, thus we label $v$ by this clique. Note some vertices of $Y$ and $S_e$ are labelled by the empty set. There is a unique label-preserving map $p:Y^{(0)}\to S^{(0)}_e(\Ga)$.

An edge in $Y$ or $S_e$ is \textit{horizontal} if the labels on its two endpoints are different, otherwise, this edge is \textit{vertical}. When $Y=X_e$, this definition coincides with the one in Section \ref{subsection_canonical restriction quotient}. Moreover, horizontal (or vertical) edges in $X_e$ are lifts of horizontal (or vertical) edges in $S_e$.

Horizontal edges in $Y$ are exactly those ones whose dual hyperplanes are mapped by $q$ to hyperplanes in $|\B|$, and the $q$-image of any vertical edge is a point. Now we label each edge vertical edge of $Y$ by vertices in $\Ga$ as follows. Pick vertical edge $e\subset Y$ and let $v=q(e)$. Let $\calr=\prod_{i=1}^{k}\calr_i$ be the product decomposition of the residue associated with $v$. There is a corresponding product decomposition $\Psi(v)=\prod_{i=1}^{k}\ell_i$, where $\ell_i$ is a line which is parallel to $\Psi(v_i)$, here $v_i\in|\B|$ is the vertex associated with $\calr_i$, and we view $\Psi(v_i)$ and $\Psi(v)$ as subcomplexes of $Y$. If $e$ is in the $\ell_i$-direction, then we label $e$ by the type of $\calr_i$, which is a vertex in $\Ga$. A case study implies if two vertical edges are the opposite sides of a 2-cube, then they have the same label. Hence all parallel vertical edges have the same label. Now we label vertical edges in $S_e$. Recall that the map $S_e\to S(\Ga)$ induces a 1-1 correspondence between vertical edges in $S_e$ and edges in $S(\Ga)$, and edges in $S(\Ga)$ are labelled by vertices of $\Ga$. This induces a labelling of vertical edges in $S_e$.

We pick an orientation for each vertical edge in $S_e$, and orient every vertical edge in $Y$ in the following way. A \textit{vertical line} is a geodesic line made of vertical edges. It is easy to see every vertical edge is contained in a vertical line. For two vertical $\ell_1$ and $\ell_2$, if there exist edges $e_i\in\ell_i$ for $i=1,2$ such that they are parallel, then $\ell_1$ and $\ell_2$ are parallel. To see this, it suffices to consider the case where $e_1$ and $e_2$ are the opposite sides of a 2-cube, and this follows from a similar case study as before. Now we pick an orientation for each parallel class of vertical lines, and this induces well-define orientation on each vertical edge of $Y$, moreover, this orientation respects parallelism of edges.

There is a unique way to extend $p:Y^{(0)}\to S^{(0)}_e(\Ga)$ to $p:Y^{(1)}\to S^{(1)}_e(\Ga)$ such that $p$ preserves the orientation and labelling of vertical edges. One can further extend $p$ to higher-dimensional cells as follows. A cube $\sigma\subset Y$ is of \textit{type $(m,n)$} if $\sigma$ is the product of $m$ vertical edges and $n$ horizontal edges. We extend $p$ according to the type:
\begin{enumerate}
\item If $\sigma$ is of type $(m,0)$, then we can define $p$ on $\sigma$ since the orientation of vertical edges in $Y$ respects parallelism, and $p$ preserves labelling and orientation of vertical edges. In this case, $p(\sigma)$ is an $m$-dimensional standard torus.
\item If $\sigma$ is of type $(0,n)$, then we can define $p$ on $\sigma$ since $p$ preserves labelling of vertices. In this case, $p(\sigma)\cong[0,1]^n$.
\item If $\sigma$ is of type $(m,n)$, then we can define $p$ on $\sigma$ for similar reasons as before. In this case, $p(\sigma)\cong \Bbb T^m\times[0,1]^n$.
\end{enumerate}

Pick vertex $y\in Y$, then $p$ induces a simplicial map between the vertex links $p_y:Lk(y,Y)\cong Lk(p(y),S_e)$. The above case study implies $p_y$ is a \textit{combinatorial} map, i.e. $p_y$ maps each simplex isomorphically onto its image.

\begin{thm}
\label{isomorphic to canonical blow-up}
If each map $h_{\calr}$ in the blow-up data is a bijection, then $Y$ is isomorphic to $X_e=X_e(\Ga)$, which is the universal cover of the exploded Salvetti complex $S_e=S_e(\Ga)$.
\end{thm}

\begin{proof}
We prove the theorem by showing $p:Y\to S_e$ is a covering map. It suffices to show for each vertex $y\in Y$, the above map $p_y$ is an isomorphism. Suppose $y$ is labelled by a clique $\Delta\subset \Ga$. We look at edges which contain $y$, which fall into three classes:
\begin{enumerate}
\item vertical edges;
\item horizontal edges whose other endpoints are labelled by cliques in $\Delta$;
\item horizontal edges whose other endpoints are labelled by cliques that contain $\Delta$.
\end{enumerate}
Note that there is a 1-1 correspondence between edges in $(3)$ and cliques which contain $\Delta$ and have exactly one vertex not in $\Delta$. For any clique $\Delta'\subset\Delta$ which contains all but one vertex of $\Delta$, there exists a unique edge in (2) such that its other endpoint is labelled by $\Delta'$, since if such edge does not exist, then some $h_{\calr}$ will not be surjective; if there exists more than one such edges, then some $h_{\calr}$ will not be injective. Thus there is a 1-1 correspondence between horizontal edges which contains $y$ and horizontal edges which contains $p(y)$. Hence $p_y$ induces bijection between the 0-skeletons. Moreover, edges in (3) are orthogonal to edges in (1) and (2), so a case study implies if two edges at $p(y)$ form the corner of a 2-cube, then their lifts at $y$ (if any exist) also form the corner of a 2-cube. It follows that $p_y$ induces isomorphism between the 1-skeletons. Since both $Lk(y,Y)$ and $Lk(p(y),S_e)$ are flag complexes, $p_y$ is an isomorphism.
\end{proof}

\begin{remark}
If each map $h_{\calr}$ is injective (or surjective), then $p$ is locally injective (or locally surjective).
\end{remark}

\begin{cor}
\label{uniqueness of building}
Let $\B_1=\B_1(\Ga)$ and $\B_2=\B_2(\Ga)$ be two right-angled $\Ga$-buildings with countably infinite rank $1$ residues. Then they are isomorphic as buildings.
\end{cor}

\begin{proof}
We pick a blow-up for $\B_1$ such that each map in the blow-up data is a bijection. Let $Y\to|\B_1|$ be the associated restriction quotient and let $p:Y\to S_e$ be the covering map as in Theorem \ref{isomorphic to canonical blow-up}. Note that $p$ sends vertical edges to vertical edges and horizontal edges to horizontal edges, and $p$ preserves the labelling of vertices and edges. So does the lift $\tilde{p}:Y\to X_e$ of $p$. Lemma~\ref{descending automorphism general case} implies $\tilde{p}$ descends to a cubical isomorphism $|\B_1|\to|\B|$, where $|\B|$ is the building associated with $G(\Ga)$. Since $\tilde{p}$ is label-preserving, this cubical isomorphism induces a building isomorphism $\B_1\to\B$. Similarly, we can obtain a building isomorphism $\B_2\to\B$. Hence the corollary follows.
\end{proof}

\begin{thm}
\label{descending automorphism building case}
Suppose $\Ga$ does not admit a join decomposition $\Ga=\Ga_1\circ\Ga_2$ where that $\Ga_1$ is a discrete graph with more than one vertex.  If $\B$ is a $\Ga$-building
and $q:Y\to|\B|$ is a restriction quotient with  blow-up data $\{h_{\calr}\}$,
then any automorphism $\alpha:Y\to Y$ descends to an automorphism $\alpha':|\B|\to|\B|$.
\end{thm}

\begin{proof}
By Lemma~\ref{descending automorphism general case}, it suffices to show $\alpha$ preserves the rank (Definition \ref{rank and inverse map}) of vertices of $Y$. Let $F(\Ga)$ be the flag complex of $\Ga$. Here we change the label of each vertex in $Y$ from some clique in $\Ga$ to the associated simplex in $F(\Ga)$. Suppose $y\in Y$ is vertex of rand $k$ labelled by $\Delta$. Then Lemma~\ref{downward complex} and the proof of Theorem~\ref{isomorphic to canonical blow-up} imply $Lk(y,Y)\cong K_1\ast K_2\ast\cdots\ast K_k\ast Lk(\Delta,F(\Ga))$, where each $K_i$ is discrete with cardinality $\ge 2$, and $Lk(\Delta,F(\Ga))$ is understood to be $F(\Ga)$ when $\Delta=\emptyset$. Note that $\{K_i\}_{i=1}^{k}$ comes from vertices adjacent to $y$ of rank $\le k$, and $Lk(\Delta,F(\Ga))$ comes from vertices adjacent to $y$ of rank $> k$. Thus $\alpha$ preserves the collection of rank $0$ vertices.

Now we assume $\alpha$ preserves the collection of rank $i$ vertices for $i\le k-1$. A rank $k$ vertex in $Y$ is of \textit{type I} if it is adjacent to a vertex of rank $k-1$, otherwise it is a vertex of \textit{type II}. It is clear that $\alpha$ preserves the collection of rank $k$ vertices of type I. Before we deal with type II vertices, we need the following claim. Suppose $w\in Y$ is a vertex of rank $k$ such that $\alpha(w)$ is also of rank $k$. If there exist $k$ vertices $\{z_i\}_{i=1}^{k}$ adjacent to $w$ such that 
\begin{enumerate}
\item $rank(z_i)\le k$ and $rank(\alpha(z_i))\le k$;
\item the edges $\{\overline{z_iw}\}_{i=1}^{k}$ are mutually orthogonal,
\end{enumerate}
then $rank (\alpha(z))\le k$ for any $z$ adjacent to $w$ with $rank(z)\le k$. 

Let $w'=\alpha(w)$. Suppose $w$ and $w'$ are labelled by $\Delta$ and $\Delta'$. Then $\alpha$ induces an isomorphism between the links of $w$ and $w'$ in $Y$:
\begin{center}
$\alpha_{\ast}: K_1\ast \cdots\ast K_k\ast Lk(\Delta,F(\Ga))\to  K'_1\ast\cdots\ast K'_k\ast Lk(\Delta',F(\Ga))$.
\end{center}
Each edge $\overline{z_iw}$ gives rise to a vertex in $K_i$, and each edge $\overline{\alpha(z_i)w'}$ gives rise to a vertex in $K'_i$. Thus $\alpha_{\ast}(K_1\ast \cdots\ast K_k)=K'_1\ast\cdots\ast K'_k$. Since the edge $\overline{zw}$ gives rise to a vertex in $K_1\ast \cdots\ast K_k$, the edge $\overline{\alpha(z)w'}$ gives rise to a vertex in $K'_1\ast\cdots\ast K'_k$. Then $\alpha(z)$ is of rank $\le k$.

Let $y\in Y$ be a rank $k$ vertex of type II. Then there exists an edge path $\omega$ from $y$ to a type I vertex $y_1$ such that every vertex in $\omega$ is of rank $k$. Let $\{y_i\}_{i=1}^{m}$ be consecutive vertices in $\omega$ such that $y_m=y$. Note that there are $k$ vertices of rank $k-1$ adjacent to $y_1$. By the induction assumption, they are send to vertices of rank $k-1$ by $\alpha$. Moreover, $rank(\alpha(y_1))=k$ since $y_1$ is of type I. Thus the assumption of the claim is satisfied for $y_1$. Then $rank(\alpha(y_2))\le k$, hence $rank(\alpha(y_2))=k$ by the induction assumption. Next we show $y_2$ satisfies the assumption of the claim. Let $\{z_i\}_{i=1}^{k}$ be vertices of rank $k$ such that they are adjacent to $y_1$ and $\{\overline{z_iy_1}\}_{i=1}^{k}$ are mutually orthogonal. We also assume $y_2=z_1$. Then $rank(\alpha(z_i))=k$ for all $i$. Hence all $\alpha(\overline{z_iy_1})$'s are vertical edges. For $i\ge 2$, let $z'_i$ be the vertex adjacent to $y_2$ such that $\overline{z'_iy_2}$ and $\overline{z_iy_1}$ are parallel. Then $\alpha(\overline{z'_iy_2})$ is a vertical edge for $i\ge 2$. Thus $rank(\alpha(z'_i))=k$ and the assumption of the claim is satisfies for $y_2$. We can repeat this argument finite many times to deduce that $rank(\alpha(y))=k$.
\end{proof}

\begin{remark}
\label{rem_does_not_descend}
If the assumption on $\Ga$ in Theorem~\ref{descending automorphism building case} is not satisfied, then there exists a blow-up $Y\to|\B|$ and an automorphism of $Y$ such that it does not descend to an automorphism of $|\B|$. By Corollary \ref{product of restriction quotient with Euclidean fibers}, it suffices to construct an example in the case when $\Ga$ be a discrete graph with $n$ vertices with $n\ge 2$. If $n\ge 3$, then we define each $h_{\calr}$ to be a surjective map such that the inverse image of each point has $n-2$ points. Then $Y$ is a tree with valence $=n$. If $n=2$, then we define $h_{\calr}$ to be an injective map whose image is the set of even integers. Then $Y$ is isomorphic to the first subdivision of a tree of valence 3. In both case, it is not hard to find an automorphism of $Y$ which maps some vertex of rank 0 to a vertex of rank 1.
\end{remark}

\subsection{Morphisms between blow-up data}
\label{subsection_morhpism of blow-up data}
Let $\B$ and $\B'$ be two buildings modelled on the same right-angled Coxeter group $W(\Ga)$. An isomorphism $\eta:|\B|\to|\B'|$ is \textit{rank-preserving} if for each vertex $v\in|\B|$, $v$ and $\eta(v)$ have the same rank. Note that such $\eta$ induces a bijection $\eta':\B\to\B'$ which preserves the spherical residues. Conversely, every bijection $\B\to\B'$ which preserves the spherical residues induces a rank-preserving isomorphism $|\B|\to|\B'|$. Note that $\eta'$ maps parallel residues of rank 1 to parallel residues of rank 1, thus $\eta'$ induces a bijection $\bar{\eta}:\Lambda_{\B}\to\Lambda_{\B'}$, where $\Lambda_{\B}$ and $\Lambda_{\B'}$ denote the collection of parallel classes of residues of rank 1 in $\B$ and $\B'$ respectively (see Section~\ref{subsection_input for Z-blow-up}).

\begin{definition}[$\eta$-isomorphism]
\label{eta-isomorphism}
Suppose the blow-up data (Definition \ref{input data}) of $|\B|$ and $|\B'|$ are given by $\{h_{\calr}\}$ and $\{h'_{\calr}\}$ respectively. An \textit{$\eta$-isomorphism} between the blow-up data is defined to be a collection of isometries $\{f_{\lambda}:\Z^{\lambda}\to\Z^{\bar{\eta}(\lambda)}\}_{\lambda\in\Lambda_{\B}}$ such that the following diagram commutes for every rank 1 residue $\calr\subset\B$:
\begin{center}
$\begin{CD}
\calr                        @>h_{\calr}>>       \Z^{T(\calr)}\\
@V\eta'VV                                   @Vf_{T(\calr)}VV\\
\eta'(\calr)        @>h'_{\eta'(\calr)}>>        \Z^{\bar{\eta}(T(\calr))}
\end{CD}$
\end{center}
Here $T$ is the type map defined in the beginning of Section~\ref{subsection_input for Z-blow-up}. The map $h_{\calr}$ is \textit{nondegenerate} if its image contains more than one point. In this case, if $f_{T(\calr)}$ exists, then it is unique. If $h_{\calr}$ is degenerate, then we have two choices for $f_{T(\calr)}$.


Let $\eta_1:|\B_1|\to |\B_2|$, $\eta_2:|\B_2|\to |\B_3|$ and $\eta:|\B_1|\to |\B_3|$ be rank-preserving isomorphisms such that $\eta=\eta_2\circ\eta_1$. We fix a blow-up data for each $\B_i$. Let $\{f_{\lambda}:\Z^{\lambda}\to\Z^{\bar{\eta}_1(\lambda)}\}_{\lambda\in\Lambda_{\B_1}}$ and $\{g_{\lambda}:\Z^{\lambda}\to\Z^{\bar{\eta}_2(\lambda)}\}_{\lambda\in\Lambda_{\B_2}}$ be the $\eta_1$-isomorphism and $\eta_2$-isomorphism between the corresponding blow-up data. We define the \textit{composition} of them to be $\{g_{\bar{\eta}_1 (\lambda)}\circ f_{\lambda}\}_{\lambda\in\Lambda}$, which turns out to be an $\eta$-isomorphism.
\end{definition}

Let $\Psi$ and $\Psi'$ be the fiber functor associated with the blow-up data $\{h_{\calr}\}$ and $\{h'_{\calr}\}$ respectively, and let $Y\to|\B|$ and $Y'\to|\B'|$ be the associated restriction quotient.

\begin{lem}
\label{eta to natural}
Every $\eta$-isomorphism induces a natural isomorphism from $\Psi$ to $\Psi'$, hence by Section~\ref{sec_equivariance_properties}, it induces an isomorphism $Y\to Y'$ which is a lift of $\eta:|\B|\to|\B'|$. Moreover, composition of $\eta$-isomorphisms gives rise to composition of natural transformations of the associated fiber functors.
\end{lem}

\begin{proof}
For every spherical residue $\calr\subset\B$, $\eta'$ respects the product decomposition of $\calr$. Thus the following diagram commutes:
\begin{center}
$\begin{CD}
\calr                        @>h_{\calr}>>       \Z^{T(\calr)}\\
@V\eta'VV                                   @V\prod_{\lambda\in T(\calr)}f_{\lambda}VV\\
\eta'(\calr)        @>h'_{\eta'(\calr)}>>        \Z^{\bar{\eta}(T(\calr))}
\end{CD}$
\end{center}
Here $\prod_{\lambda\in T(\calr)}f_{\lambda}$ induces an isometry $\R^{T(\calr)}\to\R^{\bar{\eta}(T(\calr))}$. This gives rise to a collection of isometries between objects of $\Psi$ and $\Psi'$. It follows from the construction in Section \ref{subsection_input for Z-blow-up} that these isometries give the required natural isomorphism between $\Psi$ and $\Psi'$. The second assertion in the lemma is straightforward.
\end{proof}

\begin{remark}
If we weaken the assumption of Definition~\ref{eta-isomorphism} by assuming each $f_{\lambda}$ is a bijection, then we can obtain a bijection between the vertex sets of $Y$ and $Y'$. This bijection preserves the fibers, however, we may not be able to extend it to a cubical map. 
\end{remark}

\begin{thm}
If each map $h_{\calr}$ in the blow-up data is a bijection, then $Y$ is isomorphic to $X_e=X_e(\Ga)$, which is the universal cover of the exploded Salvetti complex.
\end{thm}

\begin{definition}[$\eta$-quasi-morphism]
We follow the notation in Definition \ref{eta-isomorphism}. An \textit{$(\eta,L,A)$-quasi-morphism} between the blow-up data $\{h_{\calr}\}$ and $\{h'_{\calr}\}$ is a collection of $(L,A)$-quasi-isometries $\{f_{\lambda}:\Z^{\lambda}\to\Z^{\bar{\eta}(\lambda)}\}_{\lambda\in\Lambda_{\B}}$ such that the diagram in Definition \ref{eta-isomorphism} commutes up to error $A$.
\end{definition}

\begin{lem}
\label{eta-quasi}
Each $(\eta,L,A)$-quasi-morphism between $\{h_{\calr}\}$ and $\{h'_{\calr}\}$ induces an $(L',A')$-quasi-isometry $Y\to Y'$ with $L',A'$ depending on $L,A$, and the dimension of $|\B|$.
\end{lem}

\begin{proof}
By Lemma~\ref{lem_quasiisomorphism_gives_quasiisometry}, it suffices to produce an $(L',A')$-quasi-natural isomorphism from $\Psi$ to $\Psi'$. This can be done by considering maps of form $\prod_{\lambda\in T(\calr)}f_{\lambda}$ as in Lemma \ref{eta to natural}.
\end{proof}

\begin{remark}[A nice representative]
\label{good representative}
Let $Y_0$ be the collection of rank $0$ vertices in $Y$ (Definition \ref{rank and inverse map}). We define $Y_0'$ similarly. If the assumption in Lemma~\ref{locally-finite and cocompact} (2) is satisfied, then $Y_0$ and $Y'_0$ are $D$-dense in $Y$ and $Y'$ respectively. In this case, the quasi-isometry $Y\to Y'$ in Lemma~\ref{eta-quasi} can be represented by $\phi:Y_0\to Y'_0$, where $\phi$ is the bijection induced by $\eta:|\B|\to|\B'|$ (recall that we can identify $Y_0$ and $Y'_0$ with rank $0$ vertices in $|\B|$ and $|\B'|$ respectively, see Definition \ref{rank and inverse map}). The fact that $\phi$ is a quasi-isometry follows from the construction in the proof Lemma~\ref{lem_quasiisomorphism_gives_quasiisometry}.
\end{remark}

\begin{cor}
\label{quasi-isometry criterion}
If there exists constant $D>0$ such that each map $h_{\calr}$ in the blow-up data satisfies:
\begin{enumerate}
\item For any $x\in\Z^{T(\calr)}$, $|h^{-1}_{\calr}(x)|\le D$.
\item The image of $h_{\calr}$ is $D$-dense in $\Z^{T(\calr)}$.
\end{enumerate}
Then $Y$ is quasi-isometric to $G(\Ga)$.
\end{cor}

\begin{proof}
By the assumptions, there exists another set of blow-up data $\{h'_{\calr}\}$ such that each $h'_{\calr}$ is a bijection, and an $(\eta,L,A)$-quasi-isomorphism $\{f_{\lambda}\}_{\lambda\in\Lambda_{\B}}$ from $\{h'_{\calr}\}$ to $\{h_{\calr}\}$ where $\eta$ is the identity map. It follows from Lemma~\ref{eta-quasi} and Theorem~\ref{isomorphic to canonical blow-up} that $Y$ is quasi-isometric to $X_e$, the universal cover of the exploded Salvetti complex; hence $Y$ is quasi-isometric to $G(\Ga)$.
\end{proof}

In the rest of this section, we look at the special case when $\B=\B(\Ga)$
is the Davis building of $G(\Ga)$ (see the beginning of Section \ref{subsection_canonical restriction quotient}), and record an observation for later use. In this case, we identify points in $G(\Ga)$ with chambers in $\B$.

We denote the word metric on $G(\Ga)$ by $d_w$. If we identify $G(\Ga)$ with chambers of the building $\B=\B(\Ga)$, then there is another metric on $G(\Ga)$ defined in Definition~\ref{distance of residue}. We caution the reader that these two metrics are not the same. We pick a set of blow-up data $\{h_{\calr}\}$ on $\B$ and let $q:Y\to |\B|$ be the associated restriction quotient. Recall that vertices of rank 0 on $|\B|$ can be identified with chambers in $\B$, hence can be identified with $G(\Ga)$. Thus the map $q^{-1}:G(\Ga)\to Y$ is well-defined. 

\begin{lem}
\label{quasi-isometry produced by inverse map}
If there exist $L,A>0$ such that all $\{h_{\calr}:\calr\to\Z^{T(\calr)}\}$ are $(L,A)$-quasi-isometries $($here we identify chambers in $\calr$ with a subset of $G(\Ga)$, hence $\calr$ is endowed with an induced metric from $d_w)$, then $q^{-1}:(G(\Ga),d_w)\to Y$ is an $(L',A')$-quasi-isometry with its constants depending on $L,A$ and $\Ga$.
\end{lem}

\begin{proof}
Let $q':X_e\to |\B|$ be the $G(\Ga)$-equivariant canonical restriction quotient constructed in Section~\ref{subsection_canonical restriction quotient}. In this case, $(q')^{-1}:G(\Ga)\to X_e$ is a quasi-isometry whose constants depend on $\Ga$. Let $h'_{\calr}$ be the blow-up data which arises from the 1-data (Definition~\ref{1-data}) of $q'$. Then each $h'_{\calr}$ is an isometry. It follows from the assumption that there exists an $(\eta,L,A)$-quasi-isomorphism from the blow-up data $\{h'_{\calr}\}$ to $\{h_{\calr}\}$ with $\eta$ being the identity map. Thus there exists a quasi-isometry $X_e\to Y$ which can be represented by a map $\phi$ of the form in Remark~\ref{good representative}. Since $q^{-1}=\phi\circ(q')^{-1}$, the lemma follows.
\end{proof}

\subsection{An equivariant construction}
\label{subsection_an equivariant constuction}
Let $\B=\B(\Ga)$ be a right-angled building. Let $K$ be group which acts on $|\B|$ by automorphisms which preserve the rank of its vertices and let $K\acts \B$ and $K\acts\Lambda_{\B}$ be the induced actions ($\Lambda_\B$ is defined in the beginning of Section \ref{subsection_input for Z-blow-up}). 

\begin{definition}[Factor actions]
\label{factor action}
Pick $\lambda\in\Lambda$ and let $\calr_\lambda\subset \B$ be a residue of rank 1 such that $T(\calr_\lambda)=\lambda$ ($T$ is the type map defined in Section~\ref{subsection_input for Z-blow-up}). Let $K_{\lambda}$ be the stabilizer of $\lambda$ with respect to the action $K\acts\Lambda_{\B}$ and let $P(\calr_\lambda)=\calr_\lambda\times\calr^{\perp}_\lambda$ be the parallel set of $\calr_\lambda$ with its product decomposition (see Lemma~\ref{parallel sets of residues} and Theorem~\ref{product decomposition}). Then $P(\calr_\lambda)$ is $K_\lambda$-invariant, and $K_\lambda$ respects the product decomposition of $P(\calr_\lambda)$. Let $\rho_\lambda:K_\lambda\acts \calr_\lambda$ be the action of $K_\lambda$ on the $\calr_\lambda$-factor. This action $\rho_\lambda$ is called a \textit{factor action}. 
\end{definition}

We construct  equivariant blow-up data as follows. Pick one representative from each $K$-orbit of $K\acts\Lambda_{\B}$ and form the set $\{\lambda_u\}_{u\in U}$. Let $K_u$ be the stabilizer of $\lambda_u$. Pick residue $\calr_u\subset \B$ of rank 1 such that $T(\calr_u)=\lambda_u$ and let $\rho_u:K_u\acts \calr_u$ be the factor action defined as above.

To obtain a $K$-equivariant blow-up data, we pick an isometric action $K_u\acts\Z^{\lambda_u}$ and a $K_u$-equivariant map $h_{\calr_u}:\calr_u\to \Z^{\lambda_u}$. If $\calr$ is parallel to $\calr_u$ with the parallelism map given by $p:\calr\to\calr_u$, we define $h_{\calr}=h_{\calr_u}\circ p$. By the previous discussion, there is a factor action $K_u\acts\calr$, and $h_{\calr}$ is $K_u$-equivariant. We run this process for each element in $\{\lambda_u\}_{u\in U}$. If $\lambda\notin \{\lambda_u\}_{u\in U}$, then we fix an element $g_{\lambda}\in K$ such that $g_{\lambda}(\lambda)\in \{\lambda_u\}_{u\in U}$. For rank 1 element $\calr$ with $T(\calr)=\lambda$, we define $h_{\calr}=\textmd{Id}\circ h_{g_{\lambda}(\calr)}\circ g_{\lambda}$, here $\textmd{Id}:\Z^{g_{\lambda}(\lambda)}\to\Z^{\lambda}$ is the identity map. Let $K_{\lambda}=g^{-1}_{\lambda}K_{g_{\lambda}(\lambda)}g_{\lambda}$ be the stabilizer of $\lambda$. We define the action $K_{\lambda}\acts \Z^{\lambda}$ by letting $g^{-1}_{\lambda}gg_{\lambda}$ acts on $\Z^{\lambda}$ by $\textmd{Id}\circ g\circ\textmd{Id}^{-1}$ ($g\in K_{g_{\lambda}(\lambda)}$). Then $h_{\calr}$ becomes $K_{\lambda}$-equivariant.

It follows from the above construction that we can produce an isometry $f_{g,\calr}:\Z^{T(\calr)}\to \Z^{T(g(\calr))}$ for each $g\in K$ and rank 1 residue $\calr\in\B$ such that the following diagram commutes
\begin{center}
$\begin{CD}
\calr                        @>h_{\calr}>>       \Z^{T(\calr)}\\
@VgVV                                   @Vf_{g,\calr}VV\\
g(\calr)        @>h_{g(\calr)}>>        \Z^{T(g(\calr))}
\end{CD}$
\end{center}
and $f_{g_1g_2,\calr}=f_{g_1,g_2(\calr)}\circ f_{g_2,\calr}$ for any $g_1,g_2\in K$. Let $\Psi$ be the fiber functor associated with the above blow-up data and let $q:Y\to|\B|$ be the corresponding restriction quotient. Lemma \ref{eta to natural} implies $K$ acts on $\Psi$ by natural transformations, hence there is an induced action $K\acts Y$ and $q$ is $K$-equivariant. 

\begin{remark}
The previous construction depends on several choices:
\begin{enumerate}
\item The choice of the set $\{\lambda_u\}_{u\in U}$.
\item The choice of the isometric action $K_u\acts\Z^{\lambda_u}$ and the $K_u$-equivariant map $h_{\calr_u}:\calr_u\to \Z^{\lambda_u}$.
\item The choice of the elements $g_{\lambda}$'s.
\end{enumerate}
\end{remark}

\section{Quasi-actions on RAAG's}
\label{section_quasi-action on RAAGs}

In this section we will apply the construction in Section \ref{subsection_an equivariant constuction} to study quasi-actions on RAAG's.

We assume $G(\Gamma)\neq\Z$ throughout Section \ref{section_quasi-action on RAAGs}.

\subsection{The cubulation}
\label{subsec_cubulation}

Throughout Subsection \ref{subsec_cubulation} we assume $G(\Ga)\not\simeq \Z$ is a RAAG with $|\out(G(\Ga))|<\infty$, and $\rho:H\acts G(\Gamma)$ is an $(L,A)$-quasi-action.

Recall that $G(\Gamma)$ acts on $X(\Gamma)$ by deck transformations, and this action is simply transitive on the vertex set of $X(\Gamma)$. By picking a base point in $X(\Gamma)$, we identify $G(\Gamma)$ with the 0-skeleton of $X(\Gamma)$. 

\begin{definition}
\label{def_flat_preserving}
A quasi-isometry $\phi:G(\Ga)\ra G(\Ga)$ is
{\em flat-preserving} if it 
is a bijection and for every standard flat $F\subset X(\Ga)$ there is a 
standard flat $F'\subset X(\Ga)$ such that $\phi$ maps
the $0$-skeleton of $F$ bijectively onto the $0$-skeleton of $F'$.
The standard flat $F'$ is uniquely determined, and we denote it by
$\phi_*(F)$. Note that if $\phi$ is flat-preserving, then $\phi^{-1}$ is also flat-preserving.
\end{definition}

By Theorem \ref{thm_intro_vertex_rigidity}, without loss of generality
we can assume $\rho:H\acts G(\Ga)$ is an action by flat-preserving bijections which are also $(L,A)$-quasi-isometries. 

On the one hand, we want to think $G(\Ga)$ as a metric space with the word metric with respect to its standard generating set, or equivalently, with the induced $l^1$-metric from $X(\Ga)$; on the other hand, we want to treat $G(\Ga)$ as a right-angled building (see Section~\ref{subsection_canonical restriction quotient}), more precisely, we want to identify points in $G(\Ga)$ with chambers in the associated right-angled building of $G(\Ga)$. Then the $\rho$ preserves the spherical residues in $G(\Ga)$, thus there is an induced $\rho_{|\B|}:H\acts |\B|$ on the Davis realization $|\B|$ of $G(\Ga)$.

Let $\Lambda$ be the collection of parallel classes of standard geodesic lines in $X(\Ga)$, in other words, $\Lambda$ is the collection of parallel classes of rank 1 residues in $G(\Ga)$, and let $T$ be the type map defined in the beginning of Section~\ref{subsection_input for Z-blow-up}. There is another induced action $\rho_{\Lambda}:H\acts\Lambda$. For each $\lambda\in\Lambda$, let $H_{\lambda}$ be the stabilizer of $\lambda$. Pick a residue $\calr$ in the parallel class $\lambda$, and let $\rho_{\lambda}:H_{\lambda}\acts \calr$ be the factor action in Definition~\ref{factor action}. Note that $\calr$ is an isometrically embedded copy of $\Z$ with respect to the metric on $G(\Ga)$; moreover, $\rho_{\lambda}$ is an action by $(L',A')$-quasi-isometries. Here we can choose $L'$ and $A'$ such that they depend only on $L$ and $A$, so in particular they do not depend on $\lambda$ and $\calr$. 

For the action $\rho_{\Lambda}:H\acts\Lambda$, we pick a representative from each $H$-orbit and form the set $\{\lambda_u\}_{u\in U}$. By the construction in Section~\ref{subsection_an equivariant constuction}, it remains to choose an isometric action $G_u\acts\Z^{\lambda_u}$ and a $G_u$-equivariant map $h_{\calr_u}:\calr_u\to \Z^{\lambda_u}$ for each $u\in U$ ($\calr_u$ is a residue in the parallel class $\lambda_u$). The choice is provided by the following result, whose proof is postponed to Section~\ref{sec_quasi action on Z}. 

\begin{proposition}
\label{key lemma}
If a group $K$ has an action on $\Bbb Z$ by $(L,A)$-quasi-isometries, then there exists another action $K\curvearrowright\Bbb Z$ by isometries which is related to the original action by a surjective equivariant $(L',A')$-quasi-isometry $f: \Bbb Z\to \Bbb Z$ where $L',A'$ depend on $L$ and $A$.
\end{proposition}

From the above data, we produce  $H$-equivariant blow-up data $h_{\calr}:\calr\to\Z^{T(\calr)}$ for each rank 1 residue $\calr\subset G(\Ga)$ as in Section~\ref{subsection_an equivariant constuction}. Note that each $h_{\calr}$ is an $(L'',A'')$-quasi-isometry with constants depending only on $L$ and $A$.

Let $q:Y\to |\B|$ be the restriction quotient associated with the above blow-up data. Then there is an induced action $H\acts Y$ by isomorphisms, and $q$ is $H$-equivariant. It follows from Lemma~\ref{locally-finite and cocompact} that $Y$ is uniformly locally finite. 

{\em Claim. There exists an $(L_1,A_1)$-quasi-isometry $G(\Ga)\to Y$ with $L_1,A_1$ depending only on $L$ and $A$.}

\begin{proof}[Proof of claim]
Let $h'_{\calr}:\calr\to \Z^{T(\calr)}$ be another blow-up data such that each $h'_{\calr}$ is an isometry (such blow-up data always exists), and let $q':Y'\to |\B|$ be the associated restriction quotient. By Theorem~\ref{isomorphic to canonical blow-up}, $Y'$ is isomorphic to $X_e$, which is the universal cover of the exploded Salvetti complex introduced in Section~\ref{subsection_canonical restriction quotient}. For any $\lambda\in\Lambda$, we define $f_{\lambda}=h_{\calr}\circ (h'_{\calr})^{-1}$, here $\calr$ is a residue such that $T(\calr)=\lambda$ and the definition of $f_{\lambda}$ does not depend on $\calr$. Each $f_{\lambda}$ is an $(L'',A'')$-quasi-isometry and the collection of all $f_{\lambda}$'s induces a quasi-isomorphism between the blow-up data $\{h'_{\calr}\}$ and $\{h_{\calr}\}$. It follows from Lemma~\ref{eta-quasi} that there exists a quasi-isometry between $\varphi:Y'\cong X_e\to Y$, and the claim follows.
\end{proof}

Let $B_0$ be the set of vertices of rank 0 in $|\B|$. There is a natural identification of $B_0$ with $G(\Ga)$. Letting $Y_0=q^{-1}(B_0)$, we get that $q$ induces a bijection between $Y_0$ and $B_0$. We define $Y'_{0}$ similarly. It follows from (2) of Lemma~\ref{locally-finite and cocompact} that $Y'_{0}$ and $Y_0$ are $D$-dense in $Y'$ and $Y$ respectively for $D$ depending on $L$ and $A$. Note that $q^{-1}:G(\Ga)\to Y_0$ is $H$-equivariant, and if the action $\rho:H\acts G(\Ga)$ is cobounded, then $H\acts Y$ is cocompact. 

The above quasi-isometry $\varphi$ can be represented by $q^{-1}\circ q':Y'_0\to Y_0$ (Remark~\ref{good representative}). By Lemma~\ref{quasi-isometry produced by inverse map}, $(q')^{-1}:B_0=G(\Ga)\to Y'_0$ is also a quasi-isometry, and thus $q^{-1}:G(\Ga)\to Y_0$ is a quasi-isometry. This map is $H$-equivariant, so if $\rho:H\acts G(\Ga)$ is proper, then $H\acts Y$ is also proper.

\begin{remark}
Here we discuss a refinement of the above construction. Instead of requiring each $h'_{\calr}$ to be an isometry, it is possible to choose each $h'_{\calr}$ such that:
\begin{enumerate}
\item $h'_{\calr}$ is a bijection.
\item $h'_{\calr}$ is an $(L_2,A_2)$-quasi-isometry.
\item $f_{\lambda}:\Z^{\lambda}\to\Z^{\lambda}$ is a surjective map which respects the order of the $\Z$, hence can be extended to a surjective cubical map $\R^{\lambda}\to\R^{\lambda}$.
\end{enumerate}
The surjectivity in (3) follows from our choice in Proposition~\ref{key lemma}. In this case, the space $Y'$ is still isomorphic to $X_e$ (Theorem~\ref{isomorphic to canonical blow-up}). Let $\Psi$ and $\Psi'$ be the fiber functors associated with the blow-up data $\{h_{\calr}\}$ and $\{h'_{\calr}\}$. As in the proof of Lemma~\ref{eta to natural}, the  $f_{\lambda}$'s induce a natural transformation from $\Psi'$ to $\Psi$ which is made of a collection of surjective cubical maps from objects in $\Psi'$ to objects in $\Psi$; moreover, these maps are quasi-isometries with uniform quasi-isometry constants. Recall that we can describe $Y$ as the quotient of the disjoint collection $\{\si\times \Psi(\si)\}_{\si\in \face(|\B|)}$ (see the proof of Theorem~\ref{thm_fiber_functor_gives_restriction_quotient}), and a similar description holds for $Y'$. Thus there is a surjective cubical map $\phi:Y'\to Y$ induced by the natural transformation. Actually $\phi$ is a restriction quotient, since the inverse image of each hyperplane is a hyperplane. We also know $\phi$ is a quasi-isometry by Lemma~\ref{lem_quasiisomorphism_gives_quasiisometry}.
\end{remark}

The following theorem summarizes the above discussion.
 
\begin{theorem}
\label{thm_quasi_action1}
If the outer automorphism group 
$\out(G(\Ga))$ is finite and $G(\Ga)\not\simeq \Z$, 
then any quasi-action $\rho:H\acts X(\Ga)$ is
quasiconjugate to an action $\hat{\rho}$ of $H$ by cubical isometries on a uniformly locally finite $CAT(0)$ cube complex $Y$. Moreover:
\begin{enumerate}
\item If $\rho$ is cobounded, then $\hat{\rho}$ is cocompact.
\item If $\rho$ is proper, then $\hat{\rho}$ is proper.
\item Let $|\B|$ be the Davis realization of the right-angled building associated with $G(\Ga)$,  let $H\acts |\B|$ be the induced action, and let $X_e=X_e(\Ga)$ be the universal cover of the exploded Salvetti complex for $G(\Ga)$. Then $Y$ fits into the following commutative diagram:
\begin{diagram}
X_{e}   &            &\pile{\rTo^{\phi}\\}& &Y\\
	  &\rdTo_{q'} &                 &\ldTo_{q}&\\
    && |\B|  && 
\end{diagram}
Here $q'$, $q$ and $\phi$ are restriction quotients. The map $\phi$ is a quasi-isometry whose constants depend on the constants of the quasi-action $\rho$, and $q$ is $H$-equivariant.
\end{enumerate}
\end{theorem}

\begin{cor}
\label{if and only if}
Suppose the outer automorphism group $\out(G(\Ga))$ is finite. Then $H$ is quasi-isometric to $G(\Ga)$ if and only if there exists an $H$-equivariant restriction quotient map $q:Y\to |\B|$ such that:
\begin{enumerate}
\item $|\B|$ is the Davis realization of some right-angled $\Ga$-building.
\item The action $H\acts Y$ is geometric.
\item If $v\in |\B|$ is a vertex of rank $k$, then $q^{-1}(v)=\E^k$.
\end{enumerate}
\end{cor}

\begin{proof}
The only if direction follows from Theorem~\ref{thm_quasi_action1}. For the if direction, it suffices to show $Y$ is quasi-isometric to $G(\Ga)$. Let $\Phi$ be the fiber functor associated with $q$.

Pick a vertex $v\in |\B|$ of rank $k$ and let $F_v=q^{-1}(v)$. We claim $\stab(v)$ acts cocompactly on $F_v$.  By a standard argument, to prove this it suffices to show that $\{h(F_v)\}_{h\in H}$ is a locally finite family in $Y$. Suppose there exists an $R$-ball $N\subset Y$ such that there are infinitely many distinct elements in $\{h(F_v)\}_{h\in H}$ which have nontrivial intersection of $N$. Since $Y$ admits a geometric action, it is locally finite, and thus there exists a vertex $x\in|\B|$ which is contained in infinitely many distinct elements in $\{h(F_v)\}_{h\in H}$. This is impossible, since if $h(F_v)\neq h'(F_v)$, then $q(h(F_v))$ and $q(h'(F_v))$ are distinct vertices in $|\B|$ by the $H$-equivariance of $q$.

Consider a cube $\sigma\subset |\B|$ and let $v$ be its vertex of minimal rank. We claim $\Phi(\sigma)\to\Phi(v)$ is surjective, hence is an isometry. By (3), the action $H\acts |\B|$ preserves the rank of the vertices, thus $\stab(\sigma)\subset\stab(v)$. We know that $\stab(v)$ acts cocompactly on $q^{-1}(v)$; since the poset $\{w \geq v\}$ is finite, $\stab(\sigma)$ has finite index in $\stab(v)$, and so $\stab(\sigma)$ also acts cocompactly on $q^{-1}(v)$. Now the image of $\Phi(\sigma)\to\Phi(v)$ is a convex subcomplex of $q^{-1}(v)$ that is $\stab(\sigma)$-invariant, so it coincides with $q^{-1}(v)$. 

By Corollary~\ref{consistence}, we can assume $q$ is the restriction quotient of a set of blow-up data $\{h_{\calr}\}$. Pick a vertex $v\in|\B|$ of rank 1 and let $D(v)$ be the downward complex of $v$ (see Section~\ref{subsection_input for Z-blow-up}). Let $\calr_v\subset\B$ be the associated residue and let $\calr_v\to \R^{T(\calr)}$ be the map induced by $h_{\calr_v}$. Then $q^{-1}(D_v)$ is isomorphic to the mapping cylinder of this map. Since the $\stab(v)$ acts cocompactly on $q^{-1}(D_v)$,  there are only finite many orbits of vertices of rank 1, and the assumptions of Corollary~\ref{quasi-isometry criterion} are satisfied. It follows that $Y$ is quasi-isometric to $G(\Ga)$.
\end{proof}

It is possible to drop the $H$-equivariant assumption on $q$ under the following conditions. Here we do not put any assumption on $\Ga$.

\begin{thm}
\label{quasi-isometry to G}
Let $\B$ be a right-angled $\Ga$-building. Suppose $q:Y\to|\B|$ is a restriction quotient such that for every cube $\sigma\subset |\B|$, and every interior point $x \in \si$, the point inverse $q^{-1}(x)$ is a copy of $\E^{rank(v)}$,  where
$v\in \si$ is the vertex of minimal rank.

If $H$ acts geometrically on $Y$ by automorphisms, then $H$ is quasi-isometric to $G(\Ga)$.  
\end{thm}

\begin{proof}
First we assume $\Ga$ satisfies the assumption of Theorem~\ref{descending automorphism building case}. Then the above result is a consequence of Corollary~\ref{consistence}, Theorem~\ref{descending automorphism building case} and the argument in Corollary~\ref{if and only if}.

For arbitrary $\Ga$, we make a join decomposition $\Ga=\Ga_1\circ\Ga_2\circ\cdots\Ga_k\circ\Ga'$ where $\Ga'$ satisfies the assumption of Theorem~\ref{descending automorphism building case}, and all $\Ga_i$'s are discrete graphs with more than one vertex. By Corollary \ref{product of restriction quotient with Euclidean fibers}, there are induced cubical product decomposition $Y=Y_1\times Y_2\times\cdots Y_k\times Y'$ and restriction quotients $q_i:Y_i\to |\B_i|$ and $q':Y'\to |\B'|$ which satisfy the assumption of the theorem. By \cite[Proposition 2.6]{caprace2011rank}, we assume $H$ respects this product decomposition by passing to a finite index subgroup. Since $Y'$ is locally finite and cocompact, the same argument in Corollary~\ref{if and only if} implies $Y'$ is quasi-isometric to $G(\Ga')$. Each $Y_i$ is a locally finite and cocompact tree which is not quasi-isometric to a line. So $Y_i$ is quasi-isometric to $G(\Ga')$. Thus $Y$ is quasi-isometric to $G(\Ga)$.
\end{proof}

\begin{cor}
Suppose $\out(G(\Ga))$ is finite and $G(\Ga)\not\simeq\Z$. Let $\B$ be the right-angled building of $G(\Ga)$. Then $H$ is quasi-isometric to $G(\Ga)$ if and only if $H$ acts geometrically on a blow-up of $\B$ in the sense of Section \ref{subsection_input for Z-blow-up} by automorphisms.
\end{cor}
\subsection{Reduction to nicer actions} 
\label{subsection_nicer action}
Though every action $\rho:H\acts G(\Gamma)$ by flat-preserving bijections which are also $(L,A)$-quasi-isometries is quasiconjugate to an isometric action $H\acts Y$ as in Theorem~\ref{thm_quasi_action1}, it is in general impossible to take $Y=X(\Gamma)$, even if the action $\rho$ is proper and cobounded. 

\begin{definition}
\label{standard action}
Let $H=\Z/2\oplus \Z$ with the generator of $\Z/2$ and $\Z$ denoted by $a$ and $b$
respectively. Let $H\stackrel{\rho_0}{\acts} \Z$ be the action
where $\rho_0(b)(n)=n+2$, and $\rho_0(a)$ acts on $\Z$ by flipping $2n$ and $2n+1$ for
all $n\in\Z$. An action $K\acts \Z$ is {\em 2-flipping} if it factors through the action $H\stackrel{\rho_0}{\acts} \Z$ via an epimorphism $K\to H$.
\end{definition}

\begin{lem}
\label{not conjugate}
Let $\rho_K:K\acts \Z$ be a 2-flipping action. Then $\rho_K$ is not conjugate to an action by isometries on $\Z$ (with respect to the word metric on $\Z$).
\end{lem}

\begin{proof}
Suppose there exists a permutation $p:\Z\to \Z$ which conjugates $\rho_K$ to an isometric action. Let $h:K\to G$ be the epimorphism in Definition~\ref{standard action}. Pick $k_1,k_2\in K$ such that $h(k_1)$ is of order $2$ and $h(k_2)$ is of order infinity. Then $pk_1p^{-1}$ is a reflection of $\Z$ and $pk_2p^{-1}$ is a translation. However, this is impossible since $h(k_1)$ and $h(k_2)$ commute. 
\end{proof}

\begin{lem}
\label{not compactible}
There does not exists an action $\rho_K:K\acts \Z$ by $(L,A)$-quasi-isometries with the following property. $K$ has two subgroup $K_1$ and $K_2$ such that $\rho_K|_{K_1}$ is conjugate to a 2-flipping action and $\rho_K|_{K_2}$ is conjugate to a transitive action on $\Z$ by translations.
\end{lem}

\begin{proof}
By Proposition~\ref{key lemma}, there exists an isometric action $\rho'_K:K\acts \Z$ and a $K$-equivariant surjective map $f:K\stackrel{\rho_K}{\acts} \Z\lra
K\stackrel{\rho'_K}{\acts}\Z$. We claim $f$ is also injective. Given this claim, we can deduce a contradiction to Lemma~\ref{not conjugate} by restricting the action to $K_1$. To see the claim, we restrict the action to $K_2$. Thus we can assume without loss of generality that $\rho_K$ is a transitive action by translations. Suppose $f(a_1)=f(a_1+k)$ for $a_1,k\in\Z$ and $k\neq 0$. Then the equivariance of $f$ implies $f(a_1)=f(a_1+nk)$ for any integer $n\in \Z$, which contradicts that $f$ is a quasi-isometry.
\end{proof}


\begin{theorem}
Suppose $G(\Ga)$ is a RAAG with $|\out(G(\Ga))|<\infty$ and $G(\Ga)\not\simeq\Z$.
Then there is a pair $H, H'$ of finitely generated groups quasi-isometric to 
$G(\Ga)$ that does not admit discrete, virtually faithful cocompact
representations into
the same locally compact topological group.
\end{theorem} 

Recall that a discrete, virtually faithful cocompact representation from $H$ to a locally compact group $\hat{G}$ is a homomorphism $h:H\to\hat{G}$ such that its kernel is finite, and its image is a cocompact lattice.

\begin{proof}
Pick a vertex $u\in \Ga$ and let $\Ga'$ be a graph obtained by taking two copies of $\Ga$ and gluing them along the closed star of $u$. There is a graph automorphism $\alpha:\Ga'\to\Ga'$ which fixes the closed star of $u$ pointwise and flips the two copies of $\Ga$. Then $\alpha$ induces an involution $\alpha:G(\Ga')\to G(\Ga')$, which gives rise to a semi-product $H=G(\Ga')\rtimes \Z/2$. 

Note that $G(\Ga')$ is a subgroup of index 2 in both $G(\Ga)$ and $H$.  Therefore
this induces a quasi-isometry $q:H\to G(\Ga)$, an also a quasi-action $\rho_{H}:H\acts G(\Ga)$. By the previous discussion, we can assume $\rho_H$ is an action by flat-preserving quasi-isometries. We look at the associated collection of factor actions $\{H_{\lambda}\acts \Z\}_{\lambda\in\Lambda}$ (see Definition~\ref{factor action}), recall that $\Lambda$ is the collection of parallel classes of rank 1 residues in $G(\Ga)$, and a rank 1 residue in some class $\lambda$ can be identified with a copy of $\Z$. Up to conjugacy by bijective quasi-isometries, these factor actions are either transitive actions on $\Z$ or 2-flipping actions.

We claim that $G(\Ga)$ and $H$ do not admit discrete, virtually faithful cocompact representations into the same locally compact topological group. Suppose such a topological group $\hat{G}$ exists. Then by \cite[Chapter 6]{mosher2003quasi}, $\hat{G}$ has a quasi-action on $G(\Ga)$. We assume $\hat{G}$ acts on $G(\Ga)$ by flat-preserving quasi-isometries as before. Then there are restriction actions $\rho'_{G(\Ga)}:G(\Ga)\acts G(\Ga)$ and $\rho'_H:H\acts G(\Ga)$ which are discrete and cobounded. Since any two discrete and cobounded quasi-actions $H\acts G(\Ga)$ are quasi-conjugate, it follows from Theorem~\ref{thm_intro_vertex_rigidity} that $\rho_H$ and $\rho'_{H}$ are conjugate by a flat-preserving quasi-isometry. Thus factor actions of $\rho'_H$ is conjugate to factor actions of $\rho_H$ by bijective quasi-isometries. Similarly, we deduce that the factor actions of $\rho'_{G(\Ga)}$ are conjugate to transitive actions by left translations on $\Z$ via bijective quasi-isometries. Note that the factor actions of $\rho'_{G(\Ga)}$ and the factor actions of $\rho'_H$ are both restrictions of factor actions of $\hat{G}\acts G(\Ga)$, however, this is impossible by Lemma~\ref{not compactible}.
\end{proof}

\begin{cor}
\label{no action}
The group $H=G(\Ga')\rtimes \Z/2$ cannot act geometrically on $X(\Ga)$.
\end{cor}

We now give a criterion for when one can quasi-conjugate a quasi-action on $X(\Ga)$ to an isometric action $H\acts X(\Ga)$. 
\begin{thm}
\label{isometric action}
Let $\rho:H\acts G(\Ga)$ be an action by flat-preserving bijections. If for each $\lambda\in\Lambda$, the factor action $\rho_{\lambda}:H_{\lambda}\acts\Z$ can be conjugate to an action by isometries with respect to the word metric of $\Z$, then there is an flat-preserving bijection $g:G(\Ga)\to G(\Ga)$ which conjugates $\rho:H\acts G(\Gamma)$ to an action $\rho':H\acts X(\Gamma)$ by flat-preserving isometries. If $\rho$ is also an action by $(L,A)$-quasi-isometries, then $g$ can be taken to be a quasi-isometry.
\end{thm}

\begin{proof}
We repeat the construction in Section~\ref{subsec_cubulation} and assume each $h_{\calr}:\calr\to\Z^{T(\calr)}$  is a bijection. Let $q:Y\to |\B|$, $Y_0$, $q^{-1}:G(\Ga)\to Y_0$ and the action $\hat{\rho}:H\acts Y$ by automorphisms be as in Section~\ref{subsec_cubulation}. Recall that $q^{-1}$ is $H$-equivariant. There is an isomorphism $i:Y\to X_e$ by Theorem~\ref{isomorphic to canonical blow-up}, moreover, $i(Y_0)$ is exactly the collection of $0$-dimensional standard flats $X_0$ in $X_e$. We deduce from the construction of $i$ that the isometric action $H\acts X_e$ induced by $\hat{\rho}$ preserves standard flats in $X_e$. By the construction of $X_e$, there exists a natural identification $f:X_0\to G(\Ga)$ such that any automorphism of $X_e$ which preserves its standard flats induces a flat-preserving isometry of $G(\Ga)$ (with respect to the word metric) via $f$. It suffices to take $g=f\circ i\circ q^{-1}$. 
\end{proof}

Suppose we have already conjugated the flat-preserving action $\rho:H\acts G(\Gamma)$ to an action $\rho':H\acts X(\Gamma)$ (or $H\acts G(\Gamma)$) by flat-preserving isometries. We ask whether it is possible to further conjugate $\rho'$ to an action by left translations. 

We can oriented each 1-cell in $S(\Ga)$ and label it by the associated generator. This lifts to orientations and labels of edges of $X(\Ga)$. If $H$ preserves this orientation and labelling, then $\rho'$ is already an action by left translations. In general, it suffices to require $H$ preserves a possibly different orientation and labelling which satisfy several compatibility conditions. Now we recall the following definitions from \cite{huang2014quasi}.

\begin{definition}[Coherent ordering]
A \textit{coherent ordering} for $G(\Ga)$ is a blow-up data for $G(\Ga)$ such that each map $h_{\calr}$ is a bijection. Two coherent orderings are \textit{equivalent} if the their maps agree up to translations.
\end{definition}

Let $\mathcal{P}(\Ga)$ be the extension complex defined in Section~\ref{subsec_raag}. Note that we can identify $\Lambda_{G(\Ga)}$ with the $0$-skeleton of $\P(\Ga)$. Any flat-preserving action $H\acts G(\Gamma)$ induces an action $H\acts\mathcal{P}(\Ga)$ by simplicial isomorphisms. Let $F(\Gamma)$ be the flag complex of $\Ga$.

\begin{definition}[Coherent labelling]
Recall that for each vertex $x\in X(\Ga)$, there is a natural simplicial embedding $i_{x}:F(\Ga)\to\mathcal{P}(\Ga)$ by considering the standard flats passing through $x$. A \textit{coherent labelling} of $G(\Ga)$ is a simplicial map $L:\mathcal{P}(\Ga)\to F(\Ga)$ such that $L\circ i_{x}: F(\Gamma)\to F(\Gamma)$ is a simplicial isomorphism for every vertex $x\in X(\Gamma)$.
\end{definition}

The next result follows from \cite[Lemma 5.7]{huang2014quasi}.

\begin{lem}
\label{action by translation}
Let $\rho':H\acts G(\Gamma)$ be an action by flat-preserving bijections and let $H\acts\mathcal{P}(\Ga)$ be the induced action. If there exists an $H$-invariant coherent ordering and an $H$-invariant coherent labelling, then $\rho'$ is conjugate to an action by left translations.
\end{lem}

Since each vertex of $\P(\Ga)$ corresponds to a parallel class of $v$-residues for vertex $v\in\Ga$, this gives a labelling of vertices of $\P(\Ga)$ by vertices of $\Ga$. We can extend this labelling map to a simplicial map $L:\mathcal{P}(\Ga)\to F(\Ga)$, which gives rise to a coherent labelling.
 
\begin{cor}
Let $\rho:H\acts G(\Gamma)$ be an action by flat-preserving bijections. Suppose: 
\begin{enumerate}
\item The induced action $H\acts\P(\Ga)$ preserves the vertex labelling of $\P(\Ga)$ as above.
\item For each vertex $v\in\mathcal{P}(\Gamma)$, the action $\rho_{v}:H_{v}\acts \Z$ is conjugate to an action by translations.
\end{enumerate}
Then $\rho$ is conjugate to an action $H\acts G(\Gamma)$ by left translations.
\end{cor}
Note that condition (2) is equivalent to the existence of an $H$-invariant coherent ordering.


\section{Actions by quasi-isometries on $\Bbb Z$}
\label{sec_quasi action on Z}
In this section we prove Proposition \ref{key lemma}.

\subsection{Tracks}
\label{subsec_track}
Tracks were introduced in \cite{dunwoody1985accessibility}. They are hypersurface-like objects in 2-dimensional simplicial complexes.

\begin{definition}[Tracks]
Let $K$ be 2-dimensional simplicial complex. A \textit{track} $\tau\subset K$ is a connected embedded finite simplicial graph such that:
\begin{enumerate}
\item For each 2-simplex $\Delta\subset K$, $\tau\cap\Delta$ is a finite disjoint union of curves such that the end points of each curve are in the interior of edges of $\Delta$.
\item For each edge $e\in K$, $\tau\cap e$ is a discrete set in the interior of $e$. Let $\{\Delta_{\lambda}\}_{\lambda\in\Lambda}$ be the collection of 2-simplices that contains $e$. If $v\in \tau\cap e$, then for each $\lambda$, $\tau\cap\Delta_{\lambda}$ contains a curve that contains $v$.
\end{enumerate}
\end{definition}

Given a track $\tau\subset K$, we defined the \textit{support} of $\tau$, denoted Spt$(\tau)$, to be the minimal subcomplex of $K$ which contains $\tau$.

We can view hyperplanes defined in Section \ref{subsec_cube complex} as analogue of tracks in the cubical setting. Each track $\tau\subset K$ has a regular neighbourhood which fibres over $\tau$. When $K$ is simply-connected, $K\setminus\tau$ has two connected components, moreover, the regular neighbourhood of $\tau$ is homeomorphic to $\tau\times(-\epsilon,\epsilon)$. 

Two tracks $\tau_{1}$ and $\tau_{2}$ are \textit{parallel} if $\textmd{Spt}(\tau_{1})=\textmd{Spt}(\tau_{2})$ and there is a region homeomorphic to $\tau_{1}\times(0,\epsilon)$ bounded by $\tau_{1}$ and $\tau_{2}$. A track $\tau\subset K$ is \textit{essential} if the components of $K\setminus\tau$ are unbounded. The following result follows from \cite[Proposition 3.1]{dunwoody1985accessibility}:
\begin{lem}
If $K$ is simply-connected and has more than one end, then there exists an essential track $\tau\subset K$.
\end{lem}

Next we look at essential tracks which are \textquotedblleft minimal\textquotedblright;
these turn out to behave like minimal surfaces. First we metrize $K$ as in \cite{scott1996algebraic}.

Let $\Delta=\Delta(\xi_{1}\xi_{2}\xi_{3})$ be an ideal triangle in the hyperbolic plane. We mark a point in each side of the triangle as follows. Let $\phi$ be the unique isometry which fixes $\xi_{3}$ and flips $\xi_{1}$ and $\xi_{2}$, we mark the unique point in $\overline{\xi_{1}\xi_{2}}$ which is fixed by $\phi$. Other sides of $\Delta$ are marked similarly. This is called a \textit{marked ideal triangle}.

We identify each 2-simplex of $K$ with a marked ideal triangle in the hyperbolic plane and glue these triangles by isometries which identify the marked points. This gives a collection of complete metrics on each connected component of $K-K^{(0)}$ which is not an interval. We denote this collection of metrics by $d_{\H}$. If a group $G$ acts on $K$ by simplicial isomorphisms, then $G$ also acts by isometries on $(K,d_{\H})$. The original definition in \cite{scott1996algebraic} does not required these marked points, see Remark \ref{why marked} for why we add them.

Each track $\tau$ has a well-defined length under $d_{\H}$, which we denote by $l(\tau)$. We also define the \textit{weight} of $\tau$, denoted by $w(\tau)$, to be cardinality of $\tau\cap K^{(1)}$. The \textit{complexity} $c(\tau)$ is defined to be the ordered pair $(w(\tau),l(\tau))$. We order the complexity lexicographically, namely $c(\tau_{1})<c(\tau_{2})$ if and only if $w(\tau_{1})<w(\tau_{2})$ or $w(\tau_{1})=w(\tau_{2})$ and $l(\tau_{1})<l(\tau_{2})$. 

The following result follows from \cite[Lemma 2.11 and Lemma 2.14]{scott1996algebraic}:
\begin{lem}
\label{cocompact minimal}
Suppose $K$ is a uniformly locally finite and simply-connected simplicial 2-complex with at least two ends. Suppose $K$ does not contain separating vertices. Then there exists an essential track $\tau\subset K$ which has the least complexity with respect to $d_{\H}$ among all essential tracks in $K$.
\end{lem}

\begin{remark}
\label{why marked}
Let $\{\tau_{i}\}_{i=1}^{\infty}$ be a minimizing sequence. Since $K$ is uniformly locally finite, there are only finitely many combinatorial possibilities for $\{\tau_{i}\}_{i=1}^{\infty}$. Thus we can assume all the $\tau_{i}$'s are inside a finite subcomplex $L$. Moreover, we can construct a hyperbolic metric $d_{\H}$ on $L$ as above and it suffices to work in the space $(L,d_{\H})$. However, if we do not use marked points in the construction of the hyperbolic metric on $K$, then each $\tau_{i}$ may sit inside a copy of $L$ with different shears along the edges of $L$.

In \cite{scott1996algebraic}, $K$ is assumed to be cocompact, so one does not need to worry about the above issue.
\end{remark}

\begin{remark}
If we metrize each simplex in $K$ with the Euclidean metric, then Lemma \ref{cocompact minimal} and Lemma \ref{subcomplex minimal} may not be true.  For example,
one can take the following picture, where the dotted line is part of some track $\tau$. Once we shorten $\tau$, it may hit the central vertex of the hexagon. However, this cannot happen if we have hyperbolic metrics on each simplex. Once $\tau$ gets too close to some vertex, then it takes a large amount of length for $\tau$ to escape that vertex since $d_{\H}$ is complete (actually it does not matter if $d_{\H}$ is not complete, since we also have a upper bound on the weight of $\tau$).
\begin{center}
\includegraphics[scale=0.5]{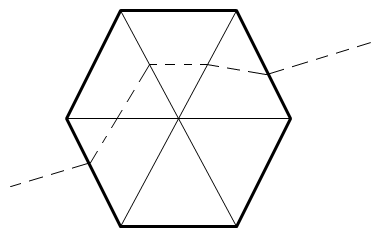}
\end{center}
\end{remark}

The next result can be proved in a similar fashion:

\begin{lem}
\label{subcomplex minimal}
Let $K$ be a simply-connected simplicial 2-complex. Let $A\subset K$ be a uniformly locally finite subcomplex such that 
\begin{enumerate}
\item $A$ contains an essential track of $K$.
\item $A$ does not contain any separating vertex of $K$.
\end{enumerate}
Then there exists an essential track $\tau$ of $K$ which has the least complexity among all essential tracks of $K$ with support in $A$.
\end{lem}

\begin{lem} \cite[Lemma 2.7]{scott1996algebraic}
\label{disjoint}
Let $\tau_{1}$ and $\tau_{2}$ be two essential tracks of $K$ which are minimal in the sense of Lemma \ref{cocompact minimal} or Lemma \ref{subcomplex minimal}, then either $\tau_{1}=\tau_{2}$, or $\tau_{1}\cap\tau_{2}=\emptyset$.
\end{lem}

\subsection{The proof of Proposition \ref{key lemma}} First we briefly recall the notion of Rips complex. See \cite[Chapter III.$\Gamma$.3]{bridson_haefliger} for more detail. Let $(X,d)$ be a metric space and pick $R>0$. The \textit{Rips complex} $P_{R}(X,d)$ is the geometric realization of the simplicial complex with vertex set $X$ whose $n$-simplices are the $(n+1)$-element subsets $\{x_{0},\cdots,x_{n}\}\subset X$ of diameter at most $R$.

Let $d$ be the usual metric on $\Bbb Z$. Define a new metric $\bar{d}$ on $\Bbb Z$ by 
\begin{equation*}
\bar{d}(x,y)=\sup_{g\in G}d(g(x),g(y))
\end{equation*}
Note that $(\Bbb Z,\bar{d})$ is quasi-isometric to $(\Bbb Z,d)$, and $G$ acts on $(\Bbb Z,\bar{d})$ by isometries. Since $(\Bbb Z,\bar{d})$ is Gromov-hyperbolic, the Rips complex $P_{R}(\Bbb Z,\bar{d})$ is contractible for some $R=R(L,A)$ (see \cite[Proposition III.$\Gamma$.3.23]{bridson_haefliger}). Let $K$ be the 2-skeleton of $P_{R}(\Bbb Z,\bar{d})$. Then $K$ is simply-connected, uniformly locally finite and 2-ended. 

We make $K$ a piecewise Euclidean complex by identifying each 2-face with an equilateral triangle and identifying each edge with $[0,1]$. Let $d_{\Bbb E}$ be the resulting length metric. There is an inclusion map $i:(\Bbb Z,d)\to (K,d_{\Bbb E})$ which is a quasi-isometry with  quasi-isometry constants depending only on $L$ and $A$.
 
\begin{claim}
\label{claim_exist_tracks}
There exist $D_{1}=D_{1}(L,A)$ and a collection of disjoint essential tracks $\{\tau_{i}\}_{i\in I}$ of $K$ such that
\begin{enumerate}
\item $\{\tau_{i}\}_{i\in I}$ is $G$-invariant.
\item The diameter of each $\tau_i$  with respect to $d_{\E}$ is $\le D_{1}$.
\item Each connected component of $K\setminus(\cup_{i\in I}\tau_{i})$ has diameter $\le D_{1}$.
\end{enumerate}
\end{claim}

In the following proof, we denote the ball of radius $D$ centered at $x$ in $K$ with respect to $d_{\E}$ by $B_{\E}(x,D)$. Let $\diam_{\E}$ be the diameter with respect to $d_{\E}$.

\begin{proof}[Proof of Claim \ref{claim_exist_tracks}]
First we assume $K$ does not have separating vertices. Since $K$ is quasi-isometric to $\Z$, there exists $D=D(L,A)$ such that $K\setminus B_{\E}(x,D)$ has at least two unbounded components for each $x\in K$. Thus every $(D+1)$-ball contains an essential track with  weight bounded above by $D'=D'(L,A)$. We put a $G$-invariant hyperbolic metric $d_{\H}$ on $K$ as in Section \ref{subsec_track}. By Lemma \ref{cocompact minimal}, there exists an essential track $\tau\subset K$ of least complexity. Note that $\diam_{\E}(\tau)\le D'$ since the weight $w(\tau)\le D'$. Lemma \ref{disjoint} implies the $G$-orbits of $\tau$ give rise to collection of disjoint essential tracks in $K$.

A collection of tracks $\{\tau_{i}\}_{i\in I}$ of $K$ is \textit{admissible} if
\begin{enumerate}
\item Each track in $\{\tau_{i}\}_{i\in I}$ is essential and different tracks have empty intersection.
\item No two tracks in $\{\tau_{i}\}_{i\in I}$ are parallel.
\item The collection $\{\tau_{i}\}_{i\in I}$ is $G$-invariant.
\item $\diam_{\E}(\tau_{i})\le D'$ for each $i\in I$.
\end{enumerate} 
There exists a non-empty admissible collection of tracks by previous discussion.

Let $\{\tau_{i}\}_{i\in I}$ be a maximal admissible collection of tracks. Then this collection satisfies the above claim with $D_{1}=2D'+5D$. To see this, let $C$ be one connected component of $K\setminus(\cup_{i\in I}\tau_{i})$. Since each track is essential and $K$ is 2-ended, either $\diam_{\E}(C)< \infty$ and $\bar{C}\setminus C$ ($\bar{C}$ is the closure of $C$) is made of two tracks $\tau_{1}$ and $\tau_{2}$, or $\diam_{\E}(C)=\infty$ and $\bar{C}\setminus C$ is made of one track. Let us assume the former case is true. The latter case can be dealt in a similar way. Let $A$ be the maximal subcomplex of $K$ which is contained in $C$. Then $A$ is uniformly locally finite and $C\setminus A$ is contained in the 1-neighbourhood of $\tau_{1}\cup\tau_{2}$.

Suppose $\diam_{\E}(C)\ge 2D'+5D$. Since $\diam_{\E}(\tau_{i})\le D'$ for $i=1,2$, there exists $x\in A$ such that $B_{\E}(x,2D)\subset A$. Thus $A$ contains an essential track of $X$ with its weight bounded above by $D'$. Let $\eta\subset A$ be an essential track of $K$ which has the least complexity in the sense of Lemma \ref{subcomplex minimal}, then $w(\eta)\le D'$, hence $\diam_{\E}(\eta)\le D'$. Moreover, by Lemma \ref{disjoint}, for each $g\in \stab(A)=\stab(C)$, either $g\cdot\eta=\eta$ or $g\cdot\eta\cap\eta=\emptyset$. Thus we can enlarge the original admissible collection of tracks by adding the $G$-orbits of $\eta$, which yields a contradiction.

The case when $K$ has separating vertices is actually easier, since one can find essential tracks on the $\epsilon$-sphere of each separating vertices. The rest of the proof is identical.
\end{proof}

We now continue with the proof of Proposition \ref{key lemma}.

Pick a regular neighbourhood $N(\tau_{i})$ for each $\tau_{i}$ such that collection $\{N(\tau_{i})\}_{i\in I}$ is disjoint and $G$-invariant. Then we define a map $\phi$ from $K$ to a tree $T$ by collapsing each component of $Y\setminus\cup_{i\in I}N(\tau_{i})$ to a vertex and collapsing each $N(\tau_{i})$, which is homeomorphic to $\tau_{i}\times (0,1)$, to the $(0,1)$ factor. It is easy to make $\phi$ equivariant under $G$, and by the above claim, $\phi$ is a quasi-isometry with  quasi-isometry constants depending only on $L$ and $A$. Note that $T$ is actually a line since $\tau$ is essential and $K$ is 2-ended. Then Proposition \ref{key lemma} follows by considering the $G$-equivariant map $\phi\circ i:(\Bbb Z,d)\to T$. 

\begin{remark}
If the action $G\curvearrowright\Bbb Z$ by $(L,A)$-quasi-isometries in Proposition \ref{key lemma} is not cobounded, then the resulting isometric action $\Lambda:G\acts\Z$ is also not cobounded, hence there are two possibilities:
\begin{enumerate}
\item if $G$ coarsely preserve the orientation of $\Z$, then $\Lambda$ is trivial;
\item otherwise $\Lambda$ factors through a $\Z/2$-action by reflection.
\end{enumerate}
\end{remark}
\section*{Part II: An alternate blow-up construction for buildings}

\section{Blowing-up buildings by metric spaces}
\label{sec_blow-up building more general version}
In Section \ref{sec_blow-up of building}
 we discussed a construction for blowing-up a right-angled building.  Here we give a somewhat different approach, which allows us to deal with more general situations.  It unifies several other objects discussed in this paper and has interesting applications to graph products. 
\subsection{Construction}
\label{construction}
We are motivated by the fact that the universal cover of the Salvetti complex can obtained by gluing a collection of standard flats in a way determined by the associated building. Similarly, we can construct the universal cover of exploded Salvetti complex by gluing a collection of branched flats. Here we specify the gluing pattern, and replace standard flats or branched flats by other product structures, which allows us to generalize the construction of these objects.

Let $\Ga$ be a finite simplicial graph, and let $\B$ be a building modelled on the right-angled Coxeter group $W(\Ga)$ with finite defining graph $\Ga$.

\textit{Step 1:}
For every vertex $v$ of $\Ga$, and every 
$v$-residue $r_v\subset \B$, we associate a metric space $Z_{r_v}$, and a map $f_{r_v}$ which maps chambers in $r_v$ to $Z_{r_v}$. If another $v$-residue $r'_v$ is parallel to $r_v$, then we associate $r'_v$ with the same space $Z_{r_v}$, and a map $f_{r'_v}$ which is induced by $f_{r_v}$ and the parallelism between $r_v$ and $r'_v$.

Let $M_{r_v}$ be the mapping cylinder of $f_{r_v}$. $M_{r_v}$ is obtained by attaching edges to $Z_{r_v}$, and the endpoints of these edges which correspond to the domain of $f_{r_v}$ are called \textit{tips} of $M_{r_v}$. There is a 1-1 correspondence between the tips of $M_{r_v}$ and chambers in $r_v$. Each edge in $M_{r_{v}}$ has length $=1$, and $M_{r_{v}}$ is endowed with the quotient metric (see \cite[Chapter I.5.19]{bridson_haefliger}).

In summary, for each parallel class of $v$-residues, we have associated a space $M_{r_v}$, whose tips are in 1-1 correspondence with chambers in any chosen $v$-residue in this parallel class.  The spaces $Z_{r_v}$'s and the maps $f_{r_v}$'s is called the \textit{blow-up data} on 
building $\B$.

\textit{Step 2:}

Let $\res$ be the poset of spherical residues in $\B$. We define a functor $\Phi$ from $\res$ to the category of nonempty metric spaces and isometric embeddings as follows. 

Let $r_J=\prod_{j\in J}r_j$ be the product decomposition of a spherical $J$-residue into its rank 1 residues. Define $\Phi(r_J)=\prod_{j\in J}M_{r_j}$, which is the Cartesian product of the metric spaces $M_{r_j}$'s.  

Suppose $r_{J'}$ is a spherical $J'$-residue such that $r_{J'}\subset r_J$. In this case $J'\subset J$. Moreover, $r_{J'}$ can be expressed as a product $\prod_{j\notin J'} u_j \times \prod_{j\in J'} r_j$, here $u_j\in r_{j}$ is a chamber. By step 1, each $u_j$ corresponds to a tip $t_j\in M_{r_{j}}$. Then we define the morphism $\Phi(r_{J'})\to\Phi(r_{J})$ to be the isometric embedding $\prod_{j\notin J'} t_j \times \prod_{j\in J'} M_{r_j}\to \prod_{j\in J}M_{r_j}$.

\textit{Step 3:}

We begin with the disjoint union
$\bigsqcup_{r_{J}\in\res}\Phi(r_{J})$, and for every inclusion
$r_{J'}\subset r_J$, we glue $\Phi(r_{J'})$ with a subset of $\Phi(r_J)$ by using the map $\Phi(r_{J'})\to\Phi(r_{J})$ defined in step 2.

We denote the resulting space with the quotient metric (\cite[Definition I.5.19]{bridson_haefliger}) by $\Pi$. $\Pi$ is called the \textit{blow-up} of $\B$ with respect to the blow-up data $Z_{r_{v}}$'s and $f_{r_{v}}$'s. 

The following lemma gives an alternative description of our gluing process. 

\begin{lem}
\label{gluing criterion}
Two points $x_1\in \Phi(r_{J_1})$ and $x_2\in \Phi(r_{J_2})$ are glued together if and only if $r_{J}=r_{J_1}\cap r_{J_2}\neq \emptyset$ and there exists a point $x\in \Phi(r_J)$ such that for $i=1,2$, its image under $\Phi(r_{J})\to\Phi(r_{J_i})$ is $x_i$.
\end{lem}

\begin{proof}
The if direction is clear. For the only if direction, define $x_1\sim x_2$ if the condition in the lemma is satisfied, it suffices to show $\sim$ is an equivalence relation. However, this follows from our construction.
\end{proof}

We can collapse each mapping cylinder in the construction to its range, namely $M_{r_{v}}\to Z_{r_{v}}$. This induces a collapsing $\Pi\to \bar{\Pi}$. $\bar{\Pi}$ is called the \textit{reduced blow-up} of $\B$ with respect to $Z_{r_{v}}$'s and $f_{r_{v}}$'s. 

\begin{remark}
There is some flexibility in this construction; for example, we can also collapse certain collections of mapping cylinders and keep other mapping cylinders to obtain a \textquotedblleft semi-reduced\textquotedblright\ blow-up. Also instead of requiring $\Phi(r_J)$ to be the Cartesian product of the metric spaces $M_{r_j}$'s, we can use $l^{p}$-product for $1\le p\le\infty$. We will see an example later where it is more natural to use $l^{1}$-product.
\end{remark}

Let $\B_1\subset\B$ be a residue. We restrict the blow-up data on $\B$ to $\B_1$ and obtain a blow-up $\Pi_1$ of $\B_1$. There is a natural map $\Pi_1\to \Pi$, which is injective by Lemma~\ref{gluing criterion}. If $\B_1$ is a spherical residue, then the image of $\Pi_1\to \Pi$ is a product of mapping cylinders. It is called a \textit{standard product}, and there is a 1-1 correspondence between spherical residues in $\B$ and standard products in $\Pi$.

\subsection{Properties}

\begin{proposition}
\label{Davis}
If each $Z_{r_{v}}$ in the above construction is a point, then the resulting blow-up $\Pi$ is isometric to the Davis realization of $\B$.
\end{proposition}

\begin{proof}
Let $|\B|$ be the Davis realization of $\B$. Let $r_{J}\subset\B$ be a spherical residue and let $D(r_{J})$ be the subcomplex of $|\B|$ spanned by those vertices corresponding to residues inside $r_{J}$. Note that $D(r_{J})$ is a convex subcomplex, and there is a naturally defined isomorphism $h_{r_J}:D(r_{J})\to \Phi(r_{J})$. One readily verifies the following:
\begin{enumerate}
\item The collection of all such $D(r_{J})$'s covers $|\B|$.
\item There is a 1-1 correspondence between $D(r_{J})$'s and $\Phi(r_{J})$'s in step 2 of Section~\ref{construction}.
\item The gluing pattern of these $D(r_{J})$'s is compatible with step 3 of Section~\ref{construction} via the maps $h_{r_J}$'s.
\end{enumerate}
It follows that $h_{r_J}$'s induce a cubical isomorphism $|\B|\to \Pi$.
\end{proof}

\begin{remark}
If each $Z_{r_{v}}$ is a point, then the corresponding reduced blow-up is a point. However, later we will see cases where the reduced blow-up is more interesting than the blow-up.
\end{remark}

It follows that for any blow-up $\Pi$ of $\B$, there is a surjective projection map from $\Pi$ to the Davis realization of $\B$ which is induced by collapsing each $Z_{r_{v}}$ to a point. This is called the \textit{canonical projection}. This map is a special case of the following situation.

Let $\Pi_1$ be the blow-up of $\B_1$ with respect to $Z_{r_{v}}$'s and $f_{r_{v}}$'s, and let $\Pi_2$ be the blow-up of $\B_2$ with respect to $Y_{r_{v}}$'s and $g_{r_{v}}$'s. Let $\bar{\Pi}_{1}$ and $\bar{\Pi}_{2}$ be the corresponding reduced blow-ups. Suppose: 
\begin{enumerate}
\item There exists a bijection $\eta:\B_1\to\B_2$ such that both $\eta$ and $\eta^{-1}$ preserve spherical residues.
\item For each $r_{v}$, there exists a map $h_{r_{v}}:Z_{r_{v}}\to Y_{\eta(r_{v})}$ such that the following diagram commutes:
\begin{center}
$\begin{CD}
r_v               @>>\eta>       \eta(r_v) \\
@VVf_{r_v}V                              @VVg_{\eta(r_v)}V     \\
Z_{r_{v}}       @>>h_{r_{v}}>      Y_{\eta(r_{v})}
\end{CD}$
\end{center}
\end{enumerate}
Then there are induced maps $h:\Pi_1\to \Pi_2$ and $\bar{h}:\bar{\Pi}_{1}\to\bar{\Pi}_{2}$. Let $\eta_\ast:|\B_1|\to |\B_2|$ be the map between Davis realizations induced by $\eta$. Then $h$ fits into the following commuting diagram:
\begin{center}
$\begin{CD}
\Pi_1               @>>h>       \Pi_2 \\
@VVV                              @VVV     \\
|\B_1|       @>>\eta_{\ast}>      |\B_2|
\end{CD}$
\end{center}

Let $\B'\subset\B$ be a residue. Let $\textmd{proj}_{\B'}:\B\to\B'$ be the projection map defined in Section~\ref{building}. Pick a spherical residue $r\in\B$. Suppose $r_0=\textmd{proj}_{r}(\B')$ and $\B'_0=\textmd{proj}_{\B'}(r)$. Let $r=r_0\times r^{\perp}_0$ be the product decomposition of $r$. It follows from \cite[Lemma 5.36]{abramenko2008buildings} that the map $\textmd{proj}_{\B'}: r\to \B'_0$ is a composition of the factor projection $r\to r_0$ and $\textmd{proj}_{\B'}:r_0\to \B'_0$. Recall that $r_0$ and $\B'_0$ are parallel and the second map induces the parallelism map between $r_0$ and $\B'_0$.  

Let $\Pi$ be a blow-up of $\B$. We restrict the blow-up data on $\B$ to $\B'$, and let $\Pi'$ be the resulting blow-up of $\B'$. Since there is a 1-1 correspondence between spherical residues in $\B$ or $\B'$ with standard products in $\Pi$ or $\Pi'$, then above discussion implies that we can assign to each standard product of $\Pi$ a standard product in $\Pi'$, together with a map between them. This assignment is compatible with the gluing pattern of these standard products, thus induces a map $\rho:\Pi\to \Pi'$. This map $\rho$ is a 1-Lipschitz retraction map, and is called a \textit{residue retraction} map. The following is an consequence of the existence of such map. 

\begin{cor}
\label{isometric embedding}
The inclusion $\Pi'\to \Pi$ is an isometric embedding.
\end{cor}
 
\begin{theorem}
\label{blow-up CAT(0)}
If each $Z_{r_{v}}$ is $CAT(0)$, then the blow-up $\Pi$ of a right-angled building $\B$ with the quotient metric is also $CAT(0)$.
\end{theorem}

\begin{proof}
We start with the following observation. If $\calr\subset \B$ is a residue, then $\calr$ is also a building. We can restrict any blow-up data of $\B$ to a blow-up data of $\calr$. Let $\Pi_{B}$ and $\Pi_{\calr}$ be the resulting blow-ups. And let $|\B|$ and $|\calr|$ be the Davis realization of $\B$ and $\calr$. Then the following diagram commutes:
\begin{center}
$\begin{CD}
\Pi_{\calr}              @>>i_{1}>       \Pi_{B} \\
@VVV                              @VVV     \\
|\calr|       @>>i_{2}>      |\B|
\end{CD}$
\end{center}
The two vertical maps are canonical projections, and $i_{1}$ and $i_2$ are isometric embeddings (Corollary~\ref{isometric embedding}). Moreover, the image of $i_{1}$ is exactly the inverse image of $|\calr|$ under the canonical projection (here $|\calr|$ is viewed as a subspace of $|\B|$).

We first show $\Pi_B$ is locally $CAT(0)$. Define the \textit{rank} of $\B$ to be the maximum possible number of vertices in a clique of $\Ga$. We induct on the rank of $\B$. The rank 1 case is clear. Now assume $\B$ is of rank $n$, and all blow-ups of buildings of rank $\le n-1$ are locally $CAT(0)$.

Pick vertex $x\in |\B|$ and suppose $x$ corresponds to a $J$-residue $r\subset \B$. Let $\pi:\Pi_B\to |\B|$ be the canonical projection and let $\st(x)$ be the open star of $x$ in $|\B|$. It suffices to show $\pi^{-1}(\st(x))$ is locally $CAT(0)$.

\textit{Case 1:} $J\neq\emptyset$. Let $J^{\perp}$ be the vertices in $\Ga$ which are adjacent to every vertex in $J$. Let $\calr$ be the $J\cup J^{\perp}$-residue that contains $r$. We restrict the blow-up data on $\B$ to $\calr$, and let $\Pi_{\calr}$ be resulting metric space. Let $\calr=\calr_{1}\times\calr_{2}$ be the product decomposition (see Section~\ref{building}) induced by $J\cup J^{\perp}$. Then we also have $|\calr|=|\calr_{1}|\times |\calr_{2}|$ and $\Pi_{\calr}=\Pi_{\calr_1}\times \Pi_{\calr_i}$ where $\Pi_{\calr_i}$ is some blow-up of $\calr_{i}$ for $i=1,2$. By the induction assumption, $\Pi_{\calr_2}$ is locally $CAT(0)$. Moreover, $\calr_{1}$ can be further decomposed into a product of buildings of rank 1, thus $\Pi_{\calr_1}$ is locally $CAT(0)$. It follows that $\Pi_{R}$ is locally $CAT(0)$. Since $\Star(x)\subset |\calr|$, $\pi^{-1}(\st(x))\subset \Pi_{R}$. Since  $\Pi_{R} \to \Pi_B$ is a local isometry, it follows that $\pi^{-1}(\st(x))$ is locally $CAT(0)$.

\textit{Case 2:} $J=\emptyset$. In this case, $\pi$ induces a simplicial isomorphism between $\pi^{-1}(\Star(x))$ and $\Star(x)$. It follows from the $CAT(0)$ property of $|\B|$ and \cite[Lemma I.5.27]{bridson_haefliger} that $\pi^{-1}(\st(x))$ is locally $CAT(0)$.

It remains to prove $\Pi_B$ is simply connected. Actually we show the canonical projection $\pi:\Pi_B\to|\B|$ is a homotopy equivalence. We construct a homotopy inverse $\phi:|\B|\to \Pi_B$ as follows. 

Recall that $|\B|$ be also be viewed as a blow-up of $\B$ where the associated metric spaces $Z'_{r_v}$'s are single points. Let $M'_{r_v}$ be the mapping cylinder of the constant map $r_v\to Z'_{r_v}$. There is a natural map $h_{r_v}:M_{r_v}\to M'_{r_v}$ induced by collapsing each $Z_{r_v}$ to $Z'_{r_v}$. 

For each $r_v$, we choose a map $h'_{r_v}:Z'_{r_v}\to Z_{r_v}$ such that if $r_v$ and $r_u$ are parallel residues of rank 1, then $h'_{r_v}$ and $h'_{r_u}$ have the same image. We extend $h'_{r_v}$ to a map $M'_{r_v}\to M_{r_v}$ as follows. Pick a chamber $c\in r_v$, and let $e_c$ and $e'_c$ be the edges in $M_{r_v}$ and $M'_{r_v}$ corresponding to the chamber $c$ respectively. We identify these two edges with $[0,1]$ such that the $0$-endpoints are the tips. We map the $[1/2,1]$ half of $e'_{c}$ to the constant speed geodesic joining the midpoint of $e_c$ and the image of $h'_{r_v}$ in $Z_{r_v}$ (such geodesic is unique since $M_{r_v}$ is $CAT(0)$), and map the $[0,1/2]$ half of $e'_c$ identically to the $[0,1/2]$ half of $e_c$.

One readily verify that one can take suitable products of the maps $h'_{r_v}:M'_{r_v}\to M_{r_v}$ and glue them together to obtain a map $\phi:|\B|\to \Pi_B$. Moreover, for each $r_v$, there is a geodesic homotopy between $h'_{r_v}\circ h_{r_v}$ to identity, and these homotopies induces a homotopy between $\phi\circ\pi$ and identity. Similarly we can produce a homotopy between $\pi\circ\phi$ and identity.
\end{proof}

\subsection{An equivariant construction}
The construction in this section is similar to Section~\ref{subsection_an equivariant constuction}.
\begin{definition}
Let $\B$ be a right-angled building. A bijection $f:\B\to\B$ is \textit{flat-preserving} if both $f$ and $f^{-1}$ map spherical residues to spherical residues.
\end{definition}
This definition is motivated by Definition~\ref{def_flat_preserving}.

Suppose $H$ acts on $\B$ by flat-preserving bijections.  Let $\Lambda$ be the collection of all parallel classes of rank 1 residues, then there is an induced action $H\acts\Lambda$. We pick a representative from each orbit of this action and denote the resulting set by $\{\lambda_{u}\}_{u\in U}$. 

Let $H_u$ be the stabilizer of $\lambda_u$. Pick rank 1 residue $r_u$ in the parallel class $\Lambda_u$ and let $\rho_u:H_u\acts r_u$ be the induced factor action (see Section~\ref{subsection_an equivariant constuction}). We pick a metric space $Z_{r_u}$, an isometric action $H_u\acts Z_{r_u}$ and a $H_u$-equivariant map $f_{r_u}:r_u\to Z_{r_u}$. And we deal with residues which are parallel to $r_u$ as in step 1 of Section~\ref{construction}.

We repeat this process for each element in $\{\lambda_u\}_{u\in U}$. If $\lambda\notin \{\lambda_u\}_{u\in U}$, then we fix an element $g_{\lambda}\in H$ such that $g_{\lambda}(\lambda)\in \{\lambda_u\}_{u\in U}$. For each rank 1 residue $r_v$ in the parallel class $\lambda$, we associate the metric space $Z_{g_{\lambda}(r_v)}$ and the map $f_{r_v}=f_{g_{\lambda}(r_v)}\circ g_{\lambda}$, and we deal with residues which are parallel to $r_v$ as before.

Let $\Pi$ be the blow-up of $\B$ with respect to the above choice of spaces and maps and let $\bar{\B}$ be the corresponding reduced blow-up.  Then there are induced actions $H\acts \Pi$ and $H\acts\bar{\Pi}$ by isometries. Moreover, the canonical projection $\Pi\to |\B|$ is $H$-equivariant.

Next we apply this construction to graph products of groups. Pick a finite simplicial graph $\Ga$ with its vertex set denoted by $I$, and pick a collection of groups $\{H_i\}_{i\in I}$, one for each vertex of $\Ga$. Let $H$ be the graph product of the $H_i$'s with respect to $\Ga$ and let $\B=\B(\Ga)$ be the right-angled building associated with this graph product (see \cite[Section 5]{davis_buildings_are_cat0}). We can identify $H$ with the set of chambers in $\B$. Under this identification, the rank 1 residues in $\B$ correspond to left cosets of $H_i$'s and spherical residues in $\B$ correspond to left cosets of $\prod_{j\in J} H_j$'s where $J$ is the vertex set of some clique in $\Ga$. 

In the following discussion, we will slightly abuse notation. When we say a residue in $H$, we actually mean the corresponding object in $\B$ under the above identification.

Let $\Lambda_H$ be the collection of all parallel classes of rank 1 residues in $H$, and let $\lambda_i\in\Lambda_H$ be the parallel class represented by $H_i$. Then there is a 1-1 correspondence between $\lambda_i$'s and $H$-orbits of the induced action $H\acts \Lambda_H$. Thus we take $\{\lambda_{u}\}_{u\in U}$ in the above construction to be $\{\lambda_i\}_{i\in I}$. For each $H_i$, pick a metric space $Z_i$, an isometric action $H_i\acts Z_i$, and a $H_i$-equivariant map $f_i:H_i\to Z_i$. Given such data, we can produce a blow-up $\Pi_{H}$ of the building $\B$.

\begin{thm}
\label{geometric action}
Let $H$ be the graph product of $\{H_i\}_{i\in I}$ with respect to a finite simplicial graph $\Ga$.  Suppose one of the following properties  is satisfied by all of $H_i$'s:
\begin{enumerate}
\item It acts geometrically on a $CAT(0)$ space;
\item It acts properly on a $CAT(0)$ space;
\item It acts geometrically on a $CAT(0)$ cube complex;
\item It acts properly on a $CAT(0)$ cube complex.
\end{enumerate}
Then $H$ also has the same property.
\end{thm}

\begin{proof}
We prove (1). The other assertions are similar. Suppose each $H_i$ acts geometrically on a $CAT(0)$ space $Z_i$. By Theorem~\ref{blow-up CAT(0)}, it suffices to show the action $H\acts \Pi_H$ is geometric.

We first show that $H\acts \Pi_H$ is cocompact. Note that $\Pi_H$ is locally compact, since each mapping cylinder in the construction of $\Pi_H$ is locally compact. We deduce that $\Pi_H$ is proper since $\Pi_H$ is $CAT(0)$. Thus it suffices to show $H\acts \Pi_H$ is cobounded. Let $|H|$ be the Davis realization of $\B$ and let $\pi:\Pi_H\to |H|$ be the $H$-equivariant canonical projection. Let $V_0\subset |H|$ be the collection of rank 0 vertices in $|H|$. Note that $\pi^{-1}(v)$ is exactly one point for any $v\in V_0$. The $H$-action on $V_0$, hence on $\pi^{-1}(V_0)$ is transitive. Since each action $H_i\acts Z_i$ is cobounded, the set of tips in the mapping cylinder of $f_i:H_{i}\to Z_i$ is cobounded. Thus $\pi^{-1}(V_0)$ is cobounded in $\Pi_H$ since it arises from tips of mapping cylinders in the construction of $\Pi_H$. It follows that $H\acts \Pi_H$ is cobounded, hence cocompact.

Now we prove $H\acts \Pi_H$ is proper. It suffices to show the intersection of $\pi^{-1}(V_0)$ with any bounded metric ball in $\Pi_H$ is finite. For each $\lambda\in\Lambda_H$, pick a residue $r_v$ in the parallel class $\lambda$. Let $\xi_{\lambda}:\Pi_H\to M_{r_v}$ be the residue projection. Note that $\xi_{\lambda}$ does not depend on the choice of $r_v$ in the parallel class. We claim there exists constant $L>0$ such that 
$$
L^{-1}d(x,y)\le \Sigma_{\lambda\in \Lambda_H} d(\xi_\lambda(x),\xi_\lambda(y))\le Ld(x,y)
$$
for any $x,y\in \pi^{-1}(V_0)$. 

To see the second inequality, pick a geodesic line $\ell\in \Pi_H$ connecting $x$ and $y$. Since each standard product in $\Pi_H$ is convex (Corollary~\ref{isometric embedding}), the intersection of $\ell$ with each standard product is a (possibly empty) segment. Thus we can assume $\ell$ is an concatenation of geodesic segments $\{\ell_i\}_{i=1}^{k}$ such that each $\ell_i$ is contained in a standard product $P_i$. Recall that each $P_i$ is a product of mapping cylinders. For each $i$, we construct another piecewise geodesic $\bar{\ell}_{i}$ such that:
\begin{enumerate}
\item $\bar{\ell}_i$ and $\ell_i$ have the same endpoints.
\item $\bar{\ell}_i$ is a concatenation of geodesic segments $\{\bar{\ell}_{ij}\}_{j=1}^{k_i}$ such that the projection of $\bar{\ell}_{ij}$ to all but one of the product factors of $P_i$ are trivial.
\item There exists $D$ which is independent of $x,y$ and $\ell_i$ such that length$(\bar{\ell}_i)\le D$ length$(\ell_i)$.
\end{enumerate}
It follows from the construction of residue projection that for each $\lambda$, $\xi_{\lambda}(\bar{\ell}_{ij})$ is either a point, or a segment which has the same length as $\bar{\ell}_{ij}$, and there is exactly one $\lambda$ such that $\xi_{\lambda}(\bar{\ell}_{ij})$ is non-trivial. Thus the second inequality follows.

To see the first inequality, let $c_x$ and $c_y$ be chambers in $\B$ that correspond to $x$ and $y$ respectively. Let $\G\subset \B$ be a shortest gallery connecting $c_x$ and $c_y$. Every two adjacent elements in the gallery are contained in a rank 1 residue, and different pairs of adjacent elements give rise to rank 1 residues in different parallel class (otherwise we can shorten the gallery). Note that $\G$ corresponds to a chain of points in $\pi^{-1}(V_0)$. Connecting adjacent points in this chain by geodesics induces a piecewise geodesic $\ell\subset \Pi_H$ connecting $x$ and $y$, moreover, the length of $\ell$ is equal to $\Sigma_{\lambda\in \Lambda_H} d(\xi_\lambda(x),\xi_\lambda(y))$, thus the first inequality holds.

Given $x\in \pi^{-1}(V_0)$, any other point $y\in \pi^{-1}(V_0)$ in the $R$-ball can be connected to $x$ via a chain of points in $\pi^{-1}(V_0)$ which is induced by a shortest gallery as above. The above claim implies that the size of the chain is bounded above in terms of $R$. Two adjacent points have distance in $\Pi_H$ bounded above in terms of $R$, and they correspond to chambers in the same rank 1 residue. There are only finite many such chains since our assumption implies there are only finitely many choices for each successive point of such a chain. Thus the intersection of $\pi^{-1}(V_0)$ with any bounded metric ball in $\Pi_H$ is finite.
\end{proof}

\begin{remark}
Suppose we use $l^1$-product instead of Cartesian product in step 2 of Section~\ref{construction}. Let $d_{l^1}$ be the resulting metric on $\Pi_H$. Then $d_{l^1}(x,y)=\Sigma_{\lambda\in \Lambda_H} d(\xi_\lambda(x),\xi_\lambda(y))$ for any $x,y\in \pi^{-1}(V_0)$. 
\end{remark}

The construction in this section unifies several other constructions as follows:
\begin{enumerate}
\item The Davis realization of right-angled building (Proposition \ref{Davis}).
\item If $\B$ is the right-angled building associated with a graph product of $\Z$'s, each space is $\R$ and each map is a bijection between chambers and integer points in $\R$, then the blow-up of $\B$ is the universal covering of the exploded Salvetti complex with the correct group action. 
\item The corresponding reduced blow-up in (2) is the universal covering of Salvetti complex.
\item Let $\B$ be as in (2). If each space is $\R$, and each map arises from Proposition~\ref{key lemma}, then the blow-up of $\B$ is the geometric model we construct for group quasi-isometric to RAAG's in Section~\ref{section_quasi-action on RAAGs}.
\item Suppose $H$ is a graph product of a collection of groups $\{H_{i}\}_{i\in I}$ and $\B$ is the associated building. Suppose each $H_{i}$ acts on a $CAT(0)$ cube complex $Z_{i}$ with a free orbit. And we pick a $H_i$-equivariant map from $H_i$ to the free orbit. Then the action of the $H$ on the associated reduced blow-up of $\B$ is equivariantly isomorphic to the graph product of group actions defined on \cite[Section 4.2]{haglund2008finite}
\end{enumerate}

\section*{Part III: A wallspace approach to cubulating RAAG's}
In the remainder of the paper, we use wallspaces to give a different approach to cubulating quasi-actions on RAAG's.  This gives a shorter path to the cubulation,
but gives less precise information about its structure.

We assume $G(\Gamma)\neq\Z$ in Section \ref{sec_construction of wallspace} and Section \ref{sec_property of wallspace}.
\section{Construction of the wallspace}
\label{sec_construction of wallspace}

\subsection{Background on wallspaces}
\label{subsec_wallspaces}
We will be following \cite{hruska2014finiteness} in this section.

\begin{definition}(Wallspaces)
Let $Z$ be a nonempty set. A \textit{wall} of Z is a partition of $Z$ into two subsets $\{U,V\}$, each of which is called a \textit{halfspace}. Two points $x,y\in Z$ are \textit{separated} by a wall if $x$ and $y$ are in different halfspaces. $(Z,\mathcal{W})$ is a \textit{wallspace} if $Z$ is endowed with a collection of walls $\mathcal{W}$ such that every two distinct points $x,y\in Z$ are separated by finitely many walls. Here we allow the situation that two points are not separated by any wall.
\end{definition}

Two distinct walls $W=\{U_{1},U_{2}\}$ and $W'=\{V_{1},V_{2}\}$ are \textit{transverse} if $U_{i}\cap V_{j}\neq\emptyset$ for all $1\le i,j\le 2$. 

An \textit{orientation} of a wall is a choice of one of its halfspaces. An orientation $\sigma$ of the wallspace $(Z,\mathcal{W})$ is a choice of orientation $\sigma(W)$ for each $W\in\mathcal{W}$. 

\begin{definition}
\label{0-cube}
A \textit{0-cube} of $(Z,\mathcal{W})$ is an orientation $\sigma$ such that
\begin{enumerate}
\item $\sigma(W)\cap \sigma(W')\neq\emptyset$ for all $W,W'\in\mathcal{W}$.
\item For each $x\in Z$, we have $x\in\sigma(W)$ for all but finitely many $W\in\mathcal{W}$.
\end{enumerate}
\end{definition}

Each $\cat(0)$ cube complex $C$ gives rise to a wallspace $(Z,\mathcal{W})$ where $Z$ is the vertex set of $C$ and the walls are induced by hyperplanes of $C$. Conversely, each wallspace $(Z,\mathcal{W})$ naturally give rises to a dual $\cat(0)$ cube complex $C(Z,\mathcal{W})$. The vertices of $C(Z,\mathcal{W})$ are in 1-1 correspondence with $0$-cubes of $(Z,\mathcal{W})$. Two $0$-cubes are joined by an edge if and only if their orientations are different on exactly one wall. 

\begin{lem}
\label{dimension}
$($\cite[Corollary 3.13]{hruska2014finiteness}$)$ $C(Z,\mathcal{W})$ is finite dimensional if and only if there is a finite upper bound on the size of collections of pairwise transverse walls. In this case, $\dim(C(Z,\mathcal{W}))$ is equal to the largest possible number of pairwise transverse walls.
\end{lem}

There is a 1-1 correspondence between walls in $(Z,\mathcal{W})$ and hyperplanes in $C(Z,\mathcal{W})$. We can define a map which associates each finite dimensional cube in $C(Z,\mathcal{W})$ a collection of pairwise transverse walls in $(Z,\mathcal{W})$ by considering the hyperplanes which have non-trivial intersection with the cube.  
 
\begin{lem}
\label{maximal cube}
$($\cite[Proposistion 3.14]{hruska2014finiteness}$)$ The above map induces a 1-1 correspondence between finite dimensional maximal cubes in $C(Z,\mathcal{W})$ and finite maximal collections of pairwise transverse walls in $(Z,\mathcal{W})$. 
\end{lem}

Two $0$-cubes $\sigma_{x}$ and $\sigma_{y}$ of $(Z,\mathcal{W})$ are \textit{separated} by a wall $W\in\mathcal{W}$ if and only if $\sigma_{x}(W)$ and $\sigma_{y}(W)$ are different halfspaces of $W$. In this case, the hyperplane in $C(Z,\mathcal{W})$ corresponding to $W$ separates the vertices associated with $\sigma_{x}$ and $\sigma_{y}$.

\subsection{Preservation of levels}
\label{subsec_flat-preserving}
In the rest of this section, we assume $\out(G(\Gamma))$ is finite and $\rho:H\acts G(\Gamma)$ is an $(L,A)$-quasi-action. We refer to Definition \ref{def_flat_preserving} for flat preserving bijections.


By Theorem \ref{thm_intro_vertex_rigidity}, we can assume $\rho$ is an action by flat-preserving bijections which are also $(L,A)$-quasi-isometries.

Let $\P(\Ga)$ be the extension complex, and for every standard flat $F\subset X(\Ga)$, let $\De(F)\in\P(\Ga)$ be the parallel class;  see Section~\ref{subsec_raag}. Note that every flat preserving bijection $h:G(\Ga)\to G(\Ga)$ induces a simplicial isomorphism $\partial h:\P(\Ga)\to\P(\Ga)$. Thus $\rho$ induces an action $H\acts \P(\Ga)$ by simplicial isomorphisms.
\label{subsec_level}
\begin{definition}[Levels]
Pick a vertex $v\in\mathcal{P}(\Gamma)$ and let $l\subset X(\Gamma)$ be a standard geodesics such that $\Delta(l)=v$. For every vertex $x\in l$, a \textit{$v$-level} (\textit{of height $x$}) is defined to be $\{z\in G(\Gamma)\mid \pi_{l}(z)=x\}$, here $\pi_{l}$ is the $\cat(0)$ projection onto $l_{v}$. Note that $\pi^{-1}_{l}(x)$ is a convex subcomplex of $X(\Gamma)$ and the $v$-level of height $x$ is exactly the vertex set of this convex subcomplex. The definition of $v$-level does not depend on the choice of the standard geodesic $l$ in the parallel class.
\end{definition}

\begin{lem}
\label{level}
Let $h: G(\Gamma)\to G(\Gamma)$ be a flat-preserving projection. Then for any $v\in\mathcal{P}(\Gamma)$, $h$ sends  $v$-levels to  $\partial h(v)$-levels.
\end{lem}

\begin{proof}
Pick standard geodesic $l$ such that $\Delta(l)=v$ and pick vertex $x\in l$. Suppose $l'$ is the standard geodesic such that $\alpha(v(l))=v(l')$. Let $y\in G(\Gamma)$ be such that $\pi_{l}(y)=x$. It suffices to show $\pi_{l'}(h(y))=h(x)$.

Let $\omega$ be a combinatorial geodesic joining $y$ and $x$ and let $\{x_{i}\}_{i=0}^{n}$ be vertices in $\omega$ such that (1) $x_{0}=y$ and $x_{n}=x$; (2) for $1\le i\le n$, there exists a standard geodesic  $l_{i}$ such that $x_{i}\in l_{i}$ and $x_{i-1}\in l_{i}$. Note that $\omega\subset \pi^{-1}_{l}(x)$ since $\pi^{-1}_{l}(x)$ is a convex subcomplex. Thus $l_{i}\subset \pi^{-1}_{l}(x)$ for all $i$, hence $\Delta(l_{i})\neq v$ for all $i$. Suppose $x'_{i}=h(x_{i})$ and let $l'_{i}$ be the standard geodesic which contains $x'_{i}$ and $x'_{i-1}$. Then $\partial h(\Delta(l_{i}))=\Delta(l'_{i})$, thus $\Delta(l'_{i})\neq\Delta(l')$. It follows that $\pi_{l'}(l'_{i})$ is a point, and hence $\pi_{l'}(x'_{i})=\pi_{l'}(x'_{i-1})$ for all $i$.
\end{proof}

\begin{lem}
\label{in one level}
Let $K\subset X(\Gamma)$ be a standard subcomplex and define $\Delta(K)$ to be the union of all $\Delta(F)$ with $F$ ranging over all standard flats in $K$. Pick vertex $v\in\mathcal{P}(\Gamma)$. If $v\notin \Delta(K)$, then $K^{(0)}$ is contained in a $v$-level.
\end{lem}

\begin{proof}
Let $l\in X(\Gamma)$ be a standard geodesic such that $\Delta(l)=v$. It suffices to prove for each edge $e\in K$, $\pi_{l}(e)$ is a point. Suppose the contrary is true. Then $\pi_{l}(e)$ is an edge of $l$ by Lemma \ref{project onto line}, and hence the standard geodesic  $l_{e}$ that contains $e$ is parallel to $l$. However, $l_{e}\subset K$ since $K$ is standard, which implies $v\in\Delta(K)$.
\end{proof}

\subsection{A key observation}
\label{subsec_key}
The flat-preserving action $H\curvearrowright G(\Gamma)$ induces an action $H\curvearrowright \mathcal{P}(\Gamma)$. For each vertex $v\in\mathcal{P}(\Gamma)$, let $H_{v}\subset H$ be the subgroup which stabilizes $v$. In other words, if $l$ is a standard geodesic such that $\Delta(l)=v$, then $H_{v}$ is the stabilizer of the parallel set of $l$. 

By Lemma \ref{level}, $H_{v}$ permutes $v$-levels. Recall that the $\cat(0)$ projection $\pi_{l}$ induces a 1-1 correspondence between $v$-levels and vertices of $l$; moreover, for any two $v$-levels $L_{1}$ and $L_{2}$, $d(L_{1},L_{2})=d(\pi_{l}(L_{1}),\pi_{l}(L_{2}))$ ($d$ is the word metric). Thus the collections of $v$-levels can be identified with $\Bbb Z$ (endowed with the standard metric), and we have an induced action $\rho_{v}:H_{v}\curvearrowright \Bbb Z$ by $(L',A')$-quasi-isometries; here $L',A'$ can be chosen to be independent of $v$. 

We recall the following result which is proved in Section~\ref{sec_quasi action on Z}.
\begin{proposition}
\label{key lemma1}
If a group $G$ has an action on $\Bbb Z$ by $(L,A)$-quasi-isometries, then there exists another action $G\curvearrowright\Bbb Z$ by isometries which is related to the original action by a surjective equivariant $(L',A')$-quasi-isometry $f: \Bbb Z\to \Bbb Z$ with $L',A'$ depending on $L$ and $A$.
\end{proposition}

\begin{definition}[Branched lines and flats]
A \textit{branched line} is constructed from a copy of $\Bbb R$ by attaching finitely many edges of length 1 to each integer point. We also require the valence of each vertex in a branched line is bounded above by a uniform constant. This space has a natural simplicial structure and is endowed with the path metric. A \textit{branched flat} is a product of finitely many branched lines.
\end{definition}

\begin{definition}[Branching number]
Let $\beta$ be a branched line. We define the\textit{ branching number} of $\beta$, denoted by $b(\beta)$, to be the maximum valence of vertices in $\beta$.
\end{definition}
The following picture is a branched line with branching number $=5$:

\includegraphics[scale=0.5]{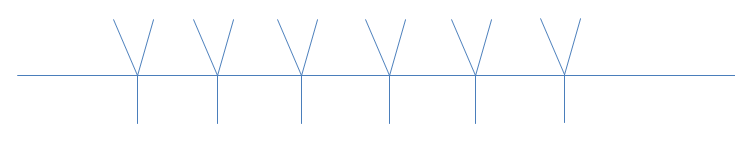}

\begin{remark}
A branched line can be roughly thought as a \textquotedblleft mapping cylinder\textquotedblright\ of the map $f$ in Proposition \ref{key lemma1}.
\end{remark}

\begin{definition}[Tips and wall structure on tips]
\label{wall on tips}
Let $\beta$ be a branched line. We call the vertices of valence $\le 2$ in $\beta$ the \textit{tips} of $\beta$, and the collection of all tips is denoted by $t(\beta)$. The set of hyperplanes in $\beta$ (namely the midpoints of edges) induces a natural wall structure on $t(\beta)$.

For a branched flat $F$, we define $t(F)$ to be the product of the tips of its branched line factors.
\end{definition}

\begin{cor}
\label{bijective}
If a group $G$ has an action on $\Bbb Z$ by $(L,A)$-quasi-isometries, then there exist a branched line $\beta$ and an isometric action $G\curvearrowright \beta$ which is related to the original action by a bijective equivariant map $q: \Bbb Z\to t(\beta)$. Moreover, there exists a constant $M$ depending only on $L$ and $A$ such that $q$ is an $M$-bi-Lipschitz map and each vertex in $\beta$ has valance $\le M$.
\end{cor}

\begin{proof}
Let $f:\Bbb Z\to\Bbb Z$ be the map in Proposition \ref{key lemma1}. We identify the range of $f$ as integer points of $\Bbb R$. Pick $x\in\Bbb Z$, if the cardinality $|f^{-1}(x)|\ge 2$, we attach $|f^{-1}(x)|$ many edges of length 1 to $\Bbb R$ along $x$. When $|f^{-1}(x)|=1$, no edge will be attached. Let $\beta$ be the resulting branched line. Then there is a natural bijective equivariant map $q: \Bbb Z\to t(\beta)$ and the corollary follows from Proposition \ref{key lemma1}.
\end{proof}

In the proof of Corollary \ref{bijective}, the $\cat(0)$ cube complex dual to $t(\beta)$ with the wall structure described in Definition \ref{wall on tips} is isomorphic to $\beta$. Note that this is not true for general branched lines.

\subsection{An invariant wallspace}
\label{subsec_invariant wallspace}
$G(\Gamma)$ has a natural wallspace structure which is induced from $X(\Gamma)$; however, $H$ may not act on this wallspace. We want to find an alternative wallspace structure on $G(\Gamma)$ which is consistent with the action of $H$.

\begin{definition}[$v$-walls and $v$-halfspaces]
Pick a vertex $v\in\mathcal{P}(\Gamma)$. It follows from Corollary \ref{bijective} that there exist a branched line $\beta_{v}$,  an isometric action $H\curvearrowright \beta_{v}$, and an $H_{v}$-equivariant surjective map $\eta_{v}: G(\Gamma)\to t(\beta_{v})$ such that the inverse image of each point in $t(\beta_{v})$ is a $v$-level. A \textit{$v$-wall} (or a \textit{$v$-halfspace}) of $G(\Gamma)$ is the pullback of some wall (or halfspace) of $t(\beta_{v})$ (see Definition \ref{wall on tips}) under the map $\eta_{v}$. The collection of all $v$-walls is denoted by $\mathcal{W}_{v}$. 

Now for each vertex $v\in\mathcal{P}(\Gamma)$, we want to choose a collection of $v$-walls, $\beta_{v}$ and $\eta_{v}$ in an $H$-equivariant way. Recall that $H\curvearrowright G(\Gamma)$ induces an action $H\curvearrowright\mathcal{P}(\Gamma)$. We pick one representative from each vertex orbit of $H\curvearrowright\mathcal{P}(\Gamma)$, which gives rise to a collection $\{v_{i}\}_{i\in I}$. We choose $\mathcal{W}_{v_{i}}$ as above and let $\mathcal{W}$ be the union of all walls in $\mathcal{W}_{v_{i}}$ for $i\in I$, together with the $H$-orbits of these walls. 
\end{definition}

Now we record several consequences of the above construction.
\begin{lem}\
\label{wallspace property}
\begin{enumerate}
\item Each $v$-halfspace is a union of $v$-levels, thus a $v$-wall and a $v'$-wall do not induce the same partition of $G(\Gamma)$ if $v\neq v'$.
\item Distinct $v$-walls in $\mathcal{W}$ are never transverse.
\item $C(G(\Gamma),\mathcal{W}_{v})$ is isomorphic to $\beta_{v}$.
\item Pick a standard geodesic line $l$ such that $\Delta(l)=v$, and identify the vertex set $v(l)$ with $\Bbb Z$. Then there exists $N>0$ independent of $v\in\mathcal{P}(\Gamma)$ such that the map $\xi_{v}:\Bbb Z\cong v(l)\to C(G(\Gamma),\mathcal{W}_{v})\cong \beta_{v}$ is an $N$-bi-Lipschitz bijection.
\item There exists $M>0$ independent of $v\in\mathcal{P}(\Gamma)$ such that the branching number $b(\beta_{v})$ is at most $M$.
\end{enumerate}
\end{lem}
 
$(4)$ and $(5)$ follow from Corollary \ref{bijective}.

\begin{lem}
\label{wallspace}
$(G(\Gamma),\mathcal{W})$ is a wallspace.
\end{lem}

\begin{proof}
For any pair of  points $x,y\in G(\Gamma)$, we need to show there are finitely many walls separating $x$ and $y$. If a $v$-wall separates $x$ from $y$, then $x$ and $y$ are in different $v$-level. There are only finitely many such vertices in $\mathcal{P}(\Gamma)$. Given such $v$, there are finitely many $v$-walls separating $x$ and $y$.
\end{proof}

By construction, $H$ acts on the wallspace $(G(\Gamma),\mathcal{W})$. Let $C$ be the $\cat(0)$ cube complex dual to $(G(\Gamma),\mathcal{W})$, then there is an induced action $H\curvearrowright C$. Now we look at several properties of $C$.
\section{Properties of the cubulation}
\label{sec_property of wallspace}

\subsection{Dimension}
\label{subsec_structure}

We need the following lemmas before we compute the dimension of the dual complex $C$.

\begin{lem}
\label{transverse}
Let $v_{1},v_{2}\in\mathcal{P}(\Gamma)$ be distinct vertices. Then a $v_{1}$-wall and a $v_{2}$-wall are transverse if and only if $v_{1}$ and $v_{2}$ are adjacent.
\end{lem}

\begin{proof}
The if direction is clear (one can consider a standard 2-flat corresponding to the vertices $v_{1}$ and $v_{2}$). The other direction follows from Lemma \ref{level projection} below.
\end{proof}

\begin{lem}
\label{level projection}
Let $v_{1},v_{2}\in\mathcal{P}(\Gamma)$ be non-adjacent vertices. For $i=1,2$, let $l_{i}$ be standard geodesic such that $\Delta(l_{i})=v_{i}$. Let $x_{2}\in l_{2}$ be the image $\pi_{l_{2}}(l_{1})$ of the $\cat(0)$ projection $\pi_{l_2}$ and let $x_{1}=\pi_{l_{1}}(l_{2})$. Then any $v_{2}$-level which is not of height $x_{2}$ is contained in the $v_{1}$-level of height $x_{1}$ $($see the picture below$)$.
\end{lem}

\begin{center} 
\includegraphics[scale=0.35]{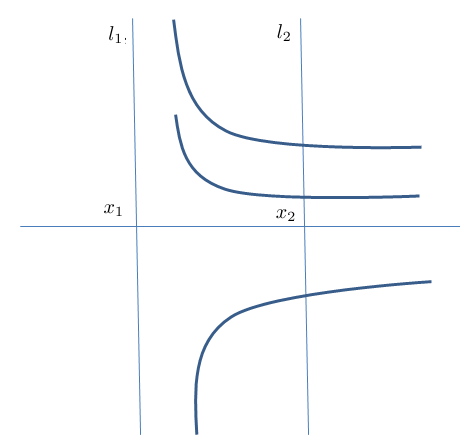}
\end{center}

\begin{proof}
Suppose the contrary is true.  Then there exists a vertex $x'_{2}\in l_{2}$ with $x'_{2}\neq x_{2}$, a $v_{2}$-level $K$ of height $x'_{2}$ and a vertex $x\in K$ such that $x$ is not inside the $v_{1}$-level of height $x_{1}$. Let $\omega\subset X(\Gamma)$ be a combinatorial geodesic connecting $x$ and $x'_{2}$. Since $\pi_{l_{1}}(x'_{2})=x_{1}\neq \pi_{l_{1}}(x)$, there exists an edge $e\subset \omega$ and edge $e_{1}\subset l_{1}$ such that $e$ and $e_1$ are parallel. Thus $\pi_{l_{2}}(e)=\pi_{l_{2}}(e_{1})=x_{2}$ by Lemma \ref{projection}. On the other hand, since $\pi^{-1}_{l_{2}}(x'_{2})$ is convex, $\pi_{l_{2}}(e)\subset\pi_{l_{2}}(\omega)=x'_{2}$, which is a contradiction.
\end{proof} 

\begin{lem}
$\dim(C)=\dim(G(\Gamma))$.
\end{lem}

\begin{proof}
Let $\{W_{i}\}_{i=1}^{l}\subset\mathcal{W}$ be a collection of pairwise intersecting walls and suppose $W_{i}$ is a $u_{i}$-wall. Then $u_{i}\neq u_{j}$ for $i\neq j$. Thus $\{u_{i}\}_{i=1}^{l}$ are vertices of a simplex in $\mathcal{P}(\Gamma)$ by Lemma \ref{transverse}. It follows from Lemma \ref{dimension} that $\dim(C)=\dim(\mathcal{P}(\Gamma))+1=\dim(G(\Gamma))$.
\end{proof}

\subsection{Standard branched flats}
Pick a maximal simplex $\Delta\in\mathcal{P}(\Gamma)$. Let $\mathcal{W}_{\Delta}$ be the collection of $u$-walls in $\mathcal{W}$ with $u$ ranging over all vertices of $\Delta$. Let $C(G(\Gamma),\mathcal{W}_{\Delta})$ be the $\cat(0)$ cube complex dual to the wallspace $(G(\Gamma),\mathcal{W}_{\Delta})$. 

\begin{lem}
\label{embedding}
There exists a natural isometric embedding 
\begin{center}
$C(G(\Gamma),\mathcal{W}_{\Delta})\hookrightarrow C(G(\Gamma),\mathcal{W})=C$.
\end{center}
\end{lem}

\begin{proof}
First we want to assign an orientation to each wall $W\in\mathcal{W}\setminus\mathcal{W}_{\Delta}$. If $W$ is a $u$-wall, then by the maximality of $\Delta$, there exists vertex $v\in\Delta$ which is not adjacent to $u$. Pick standard geodesics $l_{u}$ such that $\Delta(l_{u})=u$. Let $F$ be the maximal standard flat such that $\Delta(F)=\Delta$. Since $u\notin \Delta$, $F^{(0)}$ is contained in a $u$-level by Lemma \ref{in one level}. We orient $W$ such that
\begin{equation}
\label{contain}
v(F)\subset \sigma(W).
\end{equation}

Pick a $0$-cube $\sigma$ of $(G(\Gamma),\mathcal{W}_{\Delta})$.  We claim one can obtain a $0$-cube of $(G(\Gamma),\mathcal{W})$ by adding the orientations of walls in $\mathcal{W}\setminus\mathcal{W}_{\Delta}$ defined as above. Condition $(1)$ of Definition \ref{0-cube} follows from (\ref{contain}). To see $(2)$, suppose $W\in \mathcal{W}\setminus\mathcal{W}_{\Delta}$ such that $z\notin\sigma(W)$;  then $W$ separates $z$ from $v(F)$ by (\ref{contain}) and there are only finitely many such $W$ by Lemma \ref{wallspace}. Now we have a map from $0$-cubes of $(G(\Gamma),\mathcal{W}_{\Delta})$ to $0$-cubes of $(G(\Gamma),\mathcal{W})$, which  extends to higher dimensional cells. 

This embedding is isometric since two walls are transverse in the wallspace $(G(\Gamma),\mathcal{W}_{\Delta})$ if and only if they are transverse in $(G(\Gamma),\mathcal{W})$.
\end{proof}

The image of $C(G(\Gamma),\mathcal{W}_{\Delta})$ under the above embedding is called a \textit{maximal standard branched flat}. The following properties are immediate:
\begin{enumerate}
\item Each maximal standard branched flat is a convex subcomplex of $C$, since embedding in Lemma \ref{embedding} is isometric.
\item Each maximal standard branched flat splits as a product of branched lines $C(G(\Gamma),\mathcal{W}_{\Delta})\cong\Pi_{v\in\Delta} C(G(\Gamma),\mathcal{W}_{v})$ by Lemma \ref{transverse}.
\end{enumerate}

\begin{lem}
\label{branched flat covers}
Every maximal cube of $C$ is contained in a maximal standard branched flat, and hence every point of $C$ is contained in a maximal standard branched flat.
\end{lem} 

\begin{proof}
By Lemma \ref{maximal cube}, every maximal cube is determined by a maximal collection of pairwise transverse walls. By Lemma \ref{transverse}, these walls correspond to a maximal simplex in $\Delta\subset\mathcal{P}(\Gamma)$.  Thus the lemma follows from Lemma \ref{embedding} since the image of the embedding $C(G(\Gamma),\mathcal{W}_{\Delta})\to C$ contains a maximal cube whose dual hyperplanes are the required collection.
\end{proof}

There is a bijective map $\Phi$: $\{$maximal standard flats in $X(\Gamma)\}$ $\to$ $\{$maximal standard branched flats in $C\}$, since both sides of $\Phi$ are in 1-1 correspondence with maximal simplexes of $\mathcal{P}(\Gamma)$.

\subsection{An equivariant quasi-isometry}
Every point $x\in G(\Gamma)$ gives rise to a $0$-cube of $(X,\mathcal{W})$ by considering the halfspaces containing $x$. This induces an $H$-equivariant map $\phi:G(\Gamma)\to C$. Let $F\subset X(\Gamma)$ be a maximal standard flat, so $\phi(v(F))\subset \Phi(F)$ by $(\ref{contain})$. Actually $\phi(v(F))$ is exactly the collection of tips of $\Phi(F)$. To see this, note that $\phi|_{v(F)}$ can be written as a composition: $v(F)\to C(G(\Gamma),\mathcal{W}_{\Delta{F}})\to C(G(\Gamma,\mathcal{W}))$. Thus there is a natural splitting $\phi|_{v(F)}=\Pi_{i=1}^{n}\xi_{v_{i}}$, here $\{v_{i}\}_{i=1}^{n}$ are the vertices of $\Delta(F)$ and $\xi_{v_{i}}:\Bbb Z\to C(G(\Gamma),\mathcal{W}_{v_{i}})$ is the map in Lemma \ref{wallspace property}.

\begin{lem}
\label{surjective}
The map $\phi$ is coarsely surjective.
\end{lem}

\begin{proof}
By Lemma \ref{branched flat covers}, it suffices show there exists a constant $D$ which does not depend on the maximal flat $F$ such that $\phi(v(F))$ is $D$-dense in $\Phi(F)$. This follows from Lemma \ref{wallspace property}.
\end{proof}

\begin{prop}
\label{quasi-isometry}
The map $\phi:G(\Gamma)\to C$ is an $H$-equivariant injective quasi-isometry.
\end{prop}

\begin{proof}
We prove $\phi$ is a bi-Lipschitz embedding, then the proposition follows from Lemma \ref{surjective}. Pick $x,y\in G(\Gamma)$ and pick vertex $v\in\mathcal{P}(\Gamma)$. Suppose $l$ is a standard geodesic with $\Delta(l)=v$. Let $d_{v}(x,y)$ be the number of \textit{$v$-hyperplanes} in $X(\Gamma)$ which separate $x$ from $y$ ($h$ is a \textit{$v$-hyperplane} if and only if $l\cap h$ is one point), and let $d_{v}(\phi(x),\phi(y))$ be the number of walls in $\mathcal{W}_{v}$ that separate $x$ from $y$. It suffices to show there exists $L>0$ which does not depend on $x,y$ and $v$ such that
\begin{equation}
\label{bilip}
L^{-1}d_{v}(x,y)\le d_{v}(\phi(x),\phi(y))\le Ld_{v}(x,y).
\end{equation}

Let $x_{0}=\pi_{l}(x)$ and $y_{0}=\pi_{l}(y)$. Then $d_{v}(x,y)=d(x_{0},y_{0})$. Recall that in Lemma \ref{wallspace property} we define a map $\xi_{v}:v(l)\to C(G(\Gamma),\mathcal{W}_{v})\cong\beta_{v}$, which is $L$-bi-Lipschitz with $L$ independent of $v$. Moreover, $d_{v}(\phi(x),\phi(y))=d(\xi(x_{0}),\xi(y_{0}))$. Hence (\ref{bilip}) follows from Lemma \ref{wallspace property}.
\end{proof}

\begin{lem}
\label{intersect}
For any two maximal standard flats $F_{1},F_{2}\subset X(\Gamma)$, $F_{1}\cap F_{2}\neq\emptyset$ if and only if $\Phi(F_{1})\cap\Phi(F_{2})\neq\emptyset$.
\end{lem}

\begin{proof}
The only if direction follows from the previous discussion. If $F_{1}\cap F_{2}=\emptyset$, then there is a hyperplane $h$ separating $F_{1}$ from $F_{2}$. Let $l$ be a standard geodesic dual to $h$ and let $v=\Delta(l)$. By the maximality of $\Delta(F_{1})$, there exists a vertex $v_{1}\in\Delta(F_{1})$ which is not adjacent to $v$, and thus $v(F_{1})$ is contained is some $v$-level $V_{1}$ by Lemma \ref{projection}. Similarly, $v(F_{2})\subset V_{2}$ for some $v$-level $V_{2}\neq V_{1}$. If $W\in \mathcal{W}_{v}$  separates $V_{1}$ from $V_{2}$, then $W$ also separates $\Phi(F_{1})$ from $\Phi(F_{2})$ by (\ref{contain}).
\end{proof}

If   $x\in C$ is a vertex, since each edge that contains $x$ is inside a maximal standard branched flat (Lemma \ref{surjective}), and these maximal standard flats intersect each other,  Lemma \ref{intersect} and (5) of Lemma \ref{wallspace property} imply:
\begin{cor}
$C$ is uniformly locally finite.
\end{cor}

We have actually proved the following result since the condition on $\out(G(\Gamma))$ has not been used after Section \ref{subsec_flat-preserving}.

\begin{thm}
Let $H\curvearrowright G(\Gamma)$ be an $H$-action by flat-preserving bijections which are also $(L,A)$-quasi-isometries. Then there exists a uniformly locally finite $\cat(0)$ cube complex $C$ with $\dim(G(\Gamma))=\dim(C)$ such that the above $H$-action is conjugate to an isometric action $H\curvearrowright C$. If the $H$-action is proper or cobounded, then the resulting isometric action is proper or cocompact respectively. 
\end{thm}

\begin{cor}
\label{cubulation 1}
Suppose $\out(G(\Gamma))$ is finite and let $H\curvearrowright G(\Gamma)$ be a quasi-action. Then there exists a uniformly locally finite $\cat(0)$ cube complex $C$ with $\dim(G(\Gamma))=\dim(C)$ such that the above quasi-action is quasiconjugate to an isometric action $H\curvearrowright C$. If the quasi-action is proper or cobounded, then the resulting isometric action is proper or cocompact respectively. 
\end{cor}

Roughly speaking, the cube complex $C$ is obtained by replacing each standard geodesic  in $X(\Gamma)$ by a suitable branched line. 

\begin{remark}
The cubulation in Corollary~\ref{cubulation 1} is slightly different from the one in Theorem~\ref{thm_quasi_action1}. However, if we modify Definition~\ref{wall on tips} such that tips are the vertices of valence 1 and repeat the whole construction, then cubulation in Corollary~\ref{cubulation 1} coincides with the one in Theorem~\ref{thm_quasi_action1}.

Since we will not be using it, we will not give the argument. However, it is instructive to think about the following example. Given an isometric action $H\acts X(\Ga)$, the output cube complex is $X(\Ga)$ in Corollary~\ref{cubulation 1} and is $X_e(\Ga)$ in Theorem~\ref{thm_quasi_action1}. Here $X_e(\Ga)$ is the universal cover of exploded Salvetti complex defined in Section~\ref{subsection_canonical restriction quotient}. 
\end{remark}

\begin{cor}
\label{cor_cubulation_corollary}
Suppose $\out(G(\Gamma))$ is finite. If $G$ is a finitely generated group quasi-isometric to $G(\Gamma)$, then $G$ acts geometrically on a $\cat(0)$ cube complex.
\end{cor}

\subsection{Preservation of standard subcomplex} 
\label{subsec_preservation of std subcomplex}
We look at how standard subcomplexes behave under the map $\phi:G(\Gamma)\to C$ in Proposition \ref{quasi-isometry}.

Let $K\subset X(\Gamma)$ be a standard subcomplex and let $\mathcal{W}_{K}$ be the collection of walls in $\mathcal{W}$ that separate two vertices of $K$. Let $\Delta(K)$ be the same as in Lemma \ref{in one level}.

\begin{lem}
\label{K-walls}
Let $v\in\mathcal{P}(\Gamma)$ be a vertex. 
\begin{enumerate}
\item If $v\in \Delta(K)$, then $\mathcal{W}_{v}\subset \mathcal{W}_{K}$.
\item If $v\notin \Delta(K)$, then $\mathcal{W}_{v}\cap \mathcal{W}_{K}=\emptyset$.
\item There is a natural isometric embedding $C(G(\Gamma),\mathcal{W}_{K})\hookrightarrow C$.
\end{enumerate}
\end{lem}

\begin{proof}
If $v\in \Delta(K)$, then there exists a standard geodesic $l\subset K$ with $\Delta(l)=v$. Thus $\mathcal{W}_{v}\subset \mathcal{W}_{K}$. If $v\notin \Delta(K)$, it follows from Lemma \ref{in one level} that $K^{(0)}$ is contained in a $v$-level, thus (2) is true. It follows from (1) and (2) that for each wall $W\in \mathcal{W}\setminus\mathcal{W}_{K}$, we can orient $W$ such that $K^{(0)}\subset\sigma(W)$. The rest of the proof for (3) is identical to Lemma \ref{embedding}.
\end{proof}

\begin{thm}
Let $\phi:G(\Gamma)\to C$ be as in Proposition \ref{quasi-isometry}. Then there exists $D>0$ which only depends on the constants of the quasi-action such that for any standard subcomplex $K\subset X(\Gamma)$, there exists a convex subcomplex $K'\subset C$ such that $\phi(K^{(0)})$ is a $D$-dense subset of $K'$.
\end{thm}

\begin{proof}
By construction, the isometric embedding $C(G(\Gamma),\mathcal{W}_{K})\hookrightarrow C$ fits into the following commuting diagram:
\begin{center}
$\begin{CD}
K^{(0)}                          @>>>        G(\Gamma)\\
@VV\phi'V                                   @VV\phi V\\
C(G(\Gamma),\mathcal{W}_{K})         @>>>        C
\end{CD}$
\end{center}
Here $\phi'$ sends a vertex of $K$ to the 0-cube of $(G(\Gamma),\mathcal{W}_{K})$ which is consist of halfspaces containing this vertex. It suffices to show the image of $\phi'$ is $D$-dense.

This can be proved by the same arguments in Section \ref{subsec_structure}. Namely, for each simplex $\Delta$ which is maximal in $\Delta(K)$, there is a natural isometric embedding $C(G(\Gamma),\mathcal{W}_{\Delta})\to C(G(\Gamma),\mathcal{W}_{K})$, which gives rise to standard branched flats which are maximal in $C(G(\Gamma),\mathcal{W}_{K})$. These branched flats cover $C(G(\Gamma),\mathcal{W}_{K})$ and they are in 1-1 correspondence with standard flats which are maximal in $K$. Moreover, given a standard flat $F$ which is maximal in $K$, $\phi'(F^{(0)})$ is exactly the set of tips of the corresponding branched flat in $C(G(\Gamma),\mathcal{W}_{K})$. Now we can conclude in the same way as Lemma \ref{surjective}.
\end{proof}

\bibliography{cubulation}

\begin{thebibliography}{GMRS98}

\bibitem[AB08]{abramenko2008buildings}
P.~Abramenko and K.~S Brown.
\newblock {\em Buildings: theory and applications}.
\newblock Springer Science \& Business Media, 2008.

\bibitem[AGM13]{agol}
I.~Agol, D.~Groves, and J.~Manning.
\newblock The virtual {H}aken conjecture.
\newblock {\em Documenta Mathematica}, 18:1045--1087, 2013.

\bibitem[Ahl02]{ahlin}
A.~R. Ahlin.
\newblock The large scale geometry of products of trees.
\newblock {\em Geometriae Dedicata}, 92(1):179--184, 2002.

\bibitem[Bas72]{bass1972degree}
H.~Bass.
\newblock The degree of polynomial growth of finitely generated nilpotent
  groups.
\newblock {\em Proceedings of the London Mathematical Society}, 3(4):603--614,
  1972.

\bibitem[BB97]{bestvina1997morse}
M.~Bestvina and N.~Brady.
\newblock Morse theory and finiteness properties of groups.
\newblock {\em Inventiones mathematicae}, 129(3):445--470, 1997.

\bibitem[BH99]{bridson_haefliger}
M.~Bridson and A.~Haefliger.
\newblock {\em Metric spaces of non-positive curvature}, volume 319 of {\em
  Grundlehren der Mathematischen Wissenschaften [Fundamental Principles of
  Mathematical Sciences]}.
\newblock Springer-Verlag, Berlin, 1999.

\bibitem[BJN10]{MR2727658}
J.. Behrstock, T.~Januszkiewicz, and W.~D. Neumann.
\newblock Quasi-isometric classification of some high dimensional right-angled
  {A}rtin groups.
\newblock {\em Groups Geom. Dyn.}, 4(4):681--692, 2010.

\bibitem[BKS08a]{bks2}
M.~Bestvina, B.~Kleiner, and M.~Sageev.
\newblock The asymptotic geometry of right-angled {A}rtin groups. {I}.
\newblock {\em Geom. Topol.}, 12(3):1653--1699, 2008.

\bibitem[BKS08b]{bks1}
M.~Bestvina, B.~Kleiner, and M.~Sageev.
\newblock Quasiflats in {CAT} (0) complexes.
\newblock {\em arXiv preprint arXiv:0804.2619}, 2008.

\bibitem[BM00]{burger_mozes}
M.~Burger and S.~Mozes.
\newblock Lattices in product of trees.
\newblock {\em Inst. Hautes \'Etudes Sci. Publ. Math.}, (92):151--194 (2001),
  2000.

\bibitem[BM01]{brady2001connectivity}
N.~Brady and J.~Meier.
\newblock Connectivity at infinity for right angled {A}rtin groups.
\newblock {\em Transactions of the American Mathematical Society},
  353(1):117--132, 2001.

\bibitem[BN08]{behrstock2008quasi}
J.~Behrstock and W.~D. Neumann.
\newblock Quasi-isometric classification of graph manifold groups.
\newblock {\em Duke Mathematical Journal}, 141(2):217--240, 2008.

\bibitem[BW12]{bergeron_wise}
N.~Bergeron and D.~T. Wise.
\newblock A boundary criterion for cubulation.
\newblock {\em Amer. J. Math.}, 134(3):843--859, 2012.

\bibitem[CCV07]{charney_vogtmann_crisp}
R.~Charney, J.~Crisp, and K.~Vogtmann.
\newblock Automorphisms of 2-dimensional right-angled {A}rtin groups.
\newblock {\em Geom. Topol.}, 11:2227--2264, 2007.

\bibitem[CD95a]{MR1368655}
R.~Charney and M.~W. Davis.
\newblock Finite {$K(\pi, 1)$}s for {A}rtin groups.
\newblock In {\em Prospects in topology ({P}rinceton, {NJ}, 1994)}, volume 138
  of {\em Ann. of Math. Stud.}, pages 110--124. Princeton Univ. Press,
  Princeton, NJ, 1995.

\bibitem[CD95b]{charney_davis_kpi1_problem}
R.~Charney and M.~W. Davis.
\newblock The {$K(\pi,1)$}-problem for hyperplane complements associated to
  infinite reflection groups.
\newblock {\em J. Amer. Math. Soc.}, 8(3):597--627, 1995.

\bibitem[CF12]{charney2012random}
R.~Charney and M.~Farber.
\newblock Random groups arising as graph products.
\newblock {\em Algebr. Geom. Topol}, 12(2):979--995, 2012.

\bibitem[Cha]{charneyproblems}
R.~Charney.
\newblock Problems related to {A}rtin groups.
\newblock Preprint,
  \url{http://people.brandeis.edu/~charney/papers/Artin_probs. pdf}.

\bibitem[Cha07]{charney2007introduction}
R.~Charney.
\newblock An introduction to right-angled {A}rtin groups.
\newblock {\em Geometriae Dedicata}, 125(1):141--158, 2007.

\bibitem[CK00]{croke2000spaces}
C.~B. Croke and B.~Kleiner.
\newblock Spaces with nonpositive curvature and their ideal boundaries.
\newblock {\em Topology}, 39(3):549--556, 2000.

\bibitem[CS11]{caprace2011rank}
P.~Caprace and M.~Sageev.
\newblock Rank rigidity for {CAT} (0) cube complexes.
\newblock {\em Geom. Funct. Anal.}, 21(4):851--891, 2011.

\bibitem[Dav98]{davis_buildings_are_cat0}
M.~W. Davis.
\newblock Buildings are {${\rm CAT}(0)$}.
\newblock In {\em Geometry and cohomology in group theory ({D}urham, 1994)},
  volume 252 of {\em London Math. Soc. Lecture Note Ser.}, pages 108--123.
  1998.

\bibitem[Day12]{day_out}
M.~B. Day.
\newblock Finiteness of outer automorphism groups of random right-angled
  {A}rtin groups.
\newblock {\em Algebr. Geom. Topol.}, 12(3):1553--1583, 2012.

\bibitem[Dro87]{droms1987isomorphisms}
C.~Droms.
\newblock Isomorphisms of graph groups.
\newblock {\em Proceedings of the American Mathematical Society},
  100(3):407--408, 1987.

\bibitem[Dun85]{dunwoody1985accessibility}
M.~J. Dunwoody.
\newblock The accessibility of finitely presented groups.
\newblock {\em Inventiones mathematicae}, 81(3):449--457, 1985.

\bibitem[GMRS98]{gitik_widths}
R.~Gitik, M.~Mitra, E.~Rips, and M.~Sageev.
\newblock Widths of subgroups.
\newblock {\em Transactions of the American Mathematical Society},
  350(1):321--329, 1998.

\bibitem[Gro81a]{gromov1981groups}
M.~Gromov.
\newblock Groups of polynomial growth and expanding maps (with an appendix by
  jacques {T}its).
\newblock {\em Publications Math{\'e}matiques de l'IH{\'E}S}, 53:53--78, 1981.

\bibitem[Gro81b]{gromov_hyperbolic_manifolds_groups}
M.~Gromov.
\newblock Hyperbolic manifolds, groups and actions.
\newblock In {\em Riemann surfaces and related topics: {P}roceedings of the
  1978 {S}tony {B}rook {C}onference ({S}tate {U}niv. {N}ew {Y}ork, {S}tony
  {B}rook, {N}.{Y}., 1978)}, volume~97 of {\em Ann. of Math. Stud.}, pages
  183--213. Princeton Univ. Press, Princeton, N.J., 1981.

\bibitem[Gro87]{MR919829}
M.~Gromov.
\newblock Hyperbolic groups.
\newblock In {\em Essays in group theory}, volume~8 of {\em Math. Sci. Res.
  Inst. Publ.}, pages 75--263. Springer, New York, 1987.

\bibitem[Hag08]{haglund2008finite}
F.~Haglund.
\newblock Finite index subgroups of graph products.
\newblock {\em Geometriae Dedicata}, 135(1):167--209, 2008.

\bibitem[Hag14]{hagen_crystallographic}
M.~F. Hagen.
\newblock Cocompactly cubulated crystallographic groups.
\newblock {\em J. Lond. Math. Soc. (2)}, 90(1):140--166, 2014.

\bibitem[HP]{hagen_przytycki}
M.~F. Hagen and P.~Przytycki.
\newblock Cocompactly cubulated graph manifolds.
\newblock arXiv:1310.1309.

\bibitem[HP15]{hagen2015cocompactly}
Mark~F Hagen and Piotr Przytycki.
\newblock Cocompactly cubulated graph manifolds.
\newblock {\em Israel Journal of Mathematics}, 207(1):377--394, 2015.

\bibitem[Hua14a]{huang2014quasi}
J.~Huang.
\newblock Quasi-isometry rigidity of right-angled {A}rtin groups {I}: the
  finite out case.
\newblock {\em arXiv preprint arXiv:1410.8512}, 2014.

\bibitem[Hua14b]{huang_quasiflat}
J.~Huang.
\newblock Top dimensional quasiflats in {CAT}(0) cube complexes.
\newblock arXiv:1410.8195, 2014.

\bibitem[HW08]{haglund_wise_special}
F.~Haglund and D.~T. Wise.
\newblock Special cube complexes.
\newblock {\em Geom. Funct. Anal.}, 17(5):1551--1620, 2008.

\bibitem[HW14]{hruska2014finiteness}
G.~C. Hruska and D.~T. Wise.
\newblock Finiteness properties of cubulated groups.
\newblock {\em Compositio Mathematica}, 150(03):453--506, 2014.

\bibitem[J{\'S}01]{januszkiewicz_swiatkowski}
T.~Januszkiewicz and J.~{\'S}wiatkowski.
\newblock Commensurability of graph products.
\newblock {\em Algebr. Geom. Topol.}, 1:587--603 (electronic), 2001.

\bibitem[KK13]{kim2013embedability}
Sang-hyun Kim and T.~Koberda.
\newblock Embedability between right-angled {A}rtin groups.
\newblock {\em Geometry \& Topology}, 17(1):493--530, 2013.

\bibitem[KKL98]{kapovich1998quasi}
M.~Kapovich, B.~Kleiner, and B.~Leeb.
\newblock Quasi-isometries and the de {R}ham decomposition.
\newblock {\em Topology}, 37(6):1193--1211, 1998.

\bibitem[KL97a]{kapovich1997quasi}
M.~Kapovich and B.~Leeb.
\newblock Quasi-isometries preserve the geometric decomposition of {H}aken
  manifolds.
\newblock {\em Inventiones mathematicae}, 128(2):393--416, 1997.

\bibitem[KL97b]{kleiner1997rigidity}
B.~Kleiner and B.~Leeb.
\newblock Rigidity of quasi-isometries for symmetric spaces and {E}uclidean
  buildings.
\newblock {\em Comptes Rendus de l'Acad{\'e}mie des Sciences-Series
  I-Mathematics}, 324(6):639--643, 1997.

\bibitem[KL01]{kleiner2001groups}
B.~Kleiner and B.~Leeb.
\newblock Groups quasi-isometric to symmetric spaces.
\newblock {\em Communications in analysis and geometry}, 9(2):239--260, 2001.

\bibitem[KM98]{kapovich_millson}
M.~Kapovich and J.~J Millson.
\newblock On representation varieties of {A}rtin groups, projective
  arrangements and the fundamental groups of smooth complex algebraic
  varieties.
\newblock {\em Publications Math{\'e}matiques de l'Institut des Hautes
  {\'E}tudes Scientifiques}, 88(1):5--95, 1998.

\bibitem[KM12]{kahn_markovic}
J.~Kahn and V.~Markovic.
\newblock Immersing almost geodesic surfaces in a closed hyperbolic three
  manifold.
\newblock {\em Ann. of Math. (2)}, 175(3):1127--1190, 2012.

\bibitem[Lau95]{laurence1995generating}
M.~R Laurence.
\newblock A generating set for the automorphism group of a graph group.
\newblock {\em Journal of the London Mathematical Society}, 52(2):318--334,
  1995.

\bibitem[Lee95]{leeb}
B.~Leeb.
\newblock {$3$}-manifolds with(out) metrics of nonpositive curvature.
\newblock {\em Invent. Math.}, 122(2):277--289, 1995.

\bibitem[MSW03]{mosher2003quasi}
L.~Mosher, M.~Sageev, and K.~Whyte.
\newblock Quasi-actions on trees {I}. {B}ounded valence.
\newblock {\em Annals of mathematics}, pages 115--164, 2003.

\bibitem[Pan89]{pansu}
P.~Pansu.
\newblock M{\'e}triques de {C}arnot-{C}arath{\'e}odory et quasiisom{\'e}tries
  des espaces sym{\'e}triques de rang un.
\newblock {\em Annals of Mathematics}, pages 1--60, 1989.

\bibitem[PW02]{papasoglu2002quasi}
P.~Papasoglu and K.~Whyte.
\newblock Quasi-isometries between groups with infinitely many ends.
\newblock {\em Commentarii Mathematici Helvetici}, 77(1):133--144, 2002.

\bibitem[Ron09]{ronan2009lectures}
M.~Ronan.
\newblock {\em Lectures on Buildings: Updated and Revised}.
\newblock University of Chicago Press, 2009.

\bibitem[Sag95]{MR1347406}
M.~Sageev.
\newblock Ends of group pairs and non-positively curved cube complexes.
\newblock {\em Proc. London Math. Soc. (3)}, 71(3):585--617, 1995.

\bibitem[Sag12]{sageevnotes}
M.~Sageev.
\newblock {CAT}(0) cube complexes and groups.
\newblock {\em IAS/Park City Mathematics Series}, 2012.

\bibitem[Ser89]{servatius1989automorphisms}
H.~Servatius.
\newblock Automorphisms of graph groups.
\newblock {\em Journal of Algebra}, 126(1):34--60, 1989.

\bibitem[SS96]{scott1996algebraic}
P.~Scott and G.~A Swarup.
\newblock {\em An algebraic annulus theorem}.
\newblock Mathematical Sciences Research Inst., 1996.

\bibitem[Sta68]{stallings}
J.~R. Stallings.
\newblock On torsion-free groups with infinitely many ends.
\newblock {\em Annals of Mathematics}, pages 312--334, 1968.

\bibitem[Sul81]{sullivan}
D.~Sullivan.
\newblock On the ergodic theory at infinity of an arbitrary discrete group of
  hyperbolic motions.
\newblock In {\em Riemann surfaces and related topics: {P}roceedings of the
  1978 {S}tony {B}rook {C}onference ({S}tate {U}niv. {N}ew {Y}ork, {S}tony
  {B}rook, {N}.{Y}., 1978)}, volume~97 of {\em Ann. of Math. Stud.}, pages
  465--496. Princeton Univ. Press, Princeton, N.J., 1981.

\bibitem[Tuk86]{tukia}
P.~Tukia.
\newblock On quasiconformal groups.
\newblock {\em J. Analyse Math.}, 46:318--346, 1986.

\bibitem[Why10]{whyte2010coarse}
K.~Whyte.
\newblock Coarse bundles.
\newblock {\em arXiv preprint arXiv:1006.3347}, 2010.

\bibitem[Wis96]{wise1996non}
D.~T. Wise.
\newblock {\em Non-positively curved squared complexes aperiodic tilings and
  non-residually finite groups}.
\newblock Princeton University, 1996.

\bibitem[Wis11]{wise}
D.~T. Wise.
\newblock The structure of groups with a quasiconvex hierarchy, 2011.

\end{thebibliography}

\bibliographystyle{alpha}

\end{document}